\documentclass[reqno,11pt]{amsart}
\usepackage{amsmath,amssymb,latexsym,soul,cite,mathrsfs}
\numberwithin{equation}{section}  
\usepackage{color,enumitem,graphicx}
\usepackage{tikz}
\usetikzlibrary{arrows, patterns}
\usepackage[colorlinks=true,urlcolor=blue,
citecolor=red,linkcolor=blue,linktocpage,pdfpagelabels,
bookmarksnumbered,bookmarksopen]{hyperref}
\usepackage[english]{babel}
\usepackage[left=2.6cm,right=2.6cm,top=2.9cm,bottom=2.9cm]{geometry}
\newtheorem{theorem}{Theorem}[section]
\newtheorem{lemma}[theorem]{Lemma}
\newtheorem{corollary}[theorem]{Corollary}
\newtheorem{proposition}[theorem]{Proposition}
\newtheorem{remark}[theorem]{Remark}
\newtheorem{definition}[theorem]{Definition}

\newtheorem{theoremletter}{Theorem}

\newtheoremstyle{tttheorem}
{}                
{}                
{\slshape}        
{}                
{\bfseries}       
{'}               
{ }               
{}                
\theoremstyle{tttheorem}

\def\XXint#1#2#3{{\setbox0=\hbox{$#1{#2#3}{\int}$ }
		\vcenter{\hbox{$#2#3$ }}\kern-.6\wd0}}

\newcommand{\R}{\mathbb{R}}
\newcommand{\Ss}{\mathbb{S}}

\title[Constant $Q$-curvature metrics with Delaunay ends]{Constant $Q$-curvature metrics with Delaunay ends: The nondegenerate case}  

\author[J.H. Andrade]{Jo\~{a}o Henrique\ Andrade}
\author[R. Caju]{Rayssa Caju}
\author[J.M. do \'O]{Jo\~ao Marcos do \'O}
\author[J. Ratzkin]{Jesse Ratzkin}
\author[A. Silva Santos]{Almir Silva Santos}

\address[J.H. Andrade]{
    \newline\indent
    Department of Mathematics,
	University of British Columbia
	\newline\indent 
	V6T 1Z2, Vancouver-BC, Canada
	\newline\indent 
	and
	\newline\indent 
	Institute of Mathematics and Statistics,
	University of S\~ao Paulo
	\newline\indent 
	05508-090, S\~ao Paulo-SP, Brazil}
\email{\href{mailto:andradejh@math.ubc.ca}{andradejh@math.ubc.ca}}
\email{\href{mailto:andradejh@ime.usp.br}{andradejh@ime.usp.br}}

\address[R. Caju]{Department of Mathematics, 
	Federal University of Para\'{\i}ba
	\newline\indent 
	58051-900, Jo\~ao Pessoa-PB, Brazil}
\email{\href{mailto:rayssacaju@gmail.com}{rayssacaju@gmail.com}}

\address[J.M. do \'O]{Department of Mathematics,
	Federal University of Para\'{\i}ba
	\newline\indent 
	58051-900, Jo\~ao Pessoa-PB, Brazil}
\email{\href{mailto:jmbo@pq.cnpq.br}{jmbo@pq.cnpq.br}}

\address[J. Ratzkin]{Department of Mathematics,
	Universit\"{a}t W\"{u}rzburg
	\newline\indent
	97070, W\"{u}rzburg-BA, Germany}
\email{\href{mailto:jesse.ratzkin@mathematik.uni-wuerzburg.de}{jesse.ratzkin@mathematik.uni-wuerzburg.de}}

\address[A. Silva Santos]{Department of Mathematics, 
	Federal University of Sergipe
	\newline\indent 
	49100-000, Sao Cristov\~ao-SE, Brazil}
\email{\href{mailto: almir@mat.ufs.br}{almir@mat.ufs.br}}

\thanks{This work was partially supported by S\~ao Paulo Research Foundation (FAPESP) \#2020/07566-3 and \#2021/15139-0, Para\'iba State Research Foundation (FAPESQ) \#3034/2021, and National Council for Scientific and Technological Development (CNPq) \#312340/2021-4, \#403349/2021-4, and \#429285/2016-7}
\subjclass[2020]{35J60, 35B09, 35J30, 35B40}
\keywords{Paneitz operator, $Q$-curvature equation, Gluing construction, Prescribed asymptotics, Isolated singularities, Geometric equation}

\begin{document}
	
	\begin{abstract}
	We construct a one-parameter family of solutions to the positive singular $Q$-curvature problem on compact nondegenerate manifolds of dimension bigger than four with finitely many punctures. 
	If the  dimension is at least eight we assume that the Weyl tensor vanishes to sufficiently high order at the singular points. 
	On a technical level, we use perturbation methods and gluing techniques based on the mapping properties of the linearized operator both in a small ball around each singular point and in its exterior.
	Main difficulties in our construction include controlling the convergence rate of the Paneitz operator to the flat bi-Laplacian in conformal normal coordinates and matching the Cauchy data of the interior and exterior solutions; the latter difficulty arises from the lack of geometric Jacobi fields in the kernel of the linearized operator. 
	We overcome both these difficulties by constructing suitable auxiliary functions. 
	\end{abstract}
	
	\maketitle
	
	\begin{center}
		\footnotesize
		\tableofcontents
	\end{center}
	
	\section{Introduction}
	In the last several decades, much of the geometric analysis community has explored curvature problems similar to 
	and extending the classical Yamabe problem regarding scalar curvature. 
	The present paper constructs several new metrics on a punctured manifold with constant 
	(fourth order) $Q$-curvature. Let $(M^n,g)$ be a closed (compact and without boundary) 
	Riemannian manifold of dimension $n \geq 5$ and 		
	\begin{equation}\label{q_curvature}
	    Q_g=\displaystyle -\frac{1}{2(n-1)}\Delta_gR_g
	    +\frac{n^3-4n^2+16n-16}{8(n-1)^2(n-2)^2}R_g^2
	    -\frac{2}{(n-2)^2}|{\rm Ric}_g|^2,
	\end{equation}
	where $R_g$ is the scalar curvature, ${\rm Ric}_g$ is the 
	Ricci curvature, and $\Delta_g$ is the Laplace-Beltrami operator. 
	It is well-known that $Q_g$ transforms under a conformal 
	change of metric, denoted by $\widetilde g = u^{{4}/{(n-4)}} g$ with $u\in C^\infty(M)$ and $u>0$, according to the rule 
    \begin{equation}\label{transformationlaw} 
        Q_{\widetilde g} = \frac{2}{n-4} u^{- \frac{n+4}{n-4}} P_gu, 
    \end{equation} 
    where  the Paneitz operator $P_g$ is given by
    \begin{equation}\label{paneitz-branson} 
        P_gu=\Delta_g^2u + \operatorname{div}_g \left (\frac{4}{n-2} \operatorname{Ric}_g (\nabla u, \cdot) -
        \frac{(n-2)^2 + 4}{2(n-1)(n-2)} R_g \langle \nabla u, \cdot\rangle \right )+\frac{n-4}{2} Q_g u.
    \end{equation} 
    We refer the interested reader to \cite{MR832360,MR904819,MR3618119,MR2407527,MR2407525,MR2393291} for a thorough 
    background on the $Q$-curvature and on the Paneitz operator. 
    
    We are interested in the problem of prescribing the $Q$-curvature of a conformal 
    metric. For this, we adopt the normalization $Q_{\widetilde g} = {n(n^2-4)}/{8}$, 
    which makes this curvature equals that of the sphere ($\Ss^n,g_0)$ with 
    its standard round metric.
    With this normalization, we define the $Q$-curvature 
    operator $H_g:C^{4,\alpha}(M)\rightarrow C^{0,\alpha}(M)$, 
    for some $\alpha\in(0,1)$, by
    \begin{equation}\label{eq020}
        H_g(u)=P_gu-\frac{n(n-4)(n^2-4)}{16}|u|^{\frac{8}{n-4}}u,
    \end{equation}
    so that  \eqref{transformationlaw} can be reformulated as $H_g(u)=0$. 
    The operator $H_g$ is conformally covariant, in that 
    \begin{equation}\label{eq023}
        H_{\widetilde{g}}(\phi)=u^{-\frac{n+4}{n-4}}H_g(u\phi) \quad {\rm for \ all} \quad \phi\in C^\infty(M).
    \end{equation}
    For a given finite set of points $\Lambda$, our main goal in this paper is to find positive smooth solutions to
    the singular problem 
    \begin{equation}\label{ourequation}\tag{$\mathcal{Q}$}
	\begin{cases}
	H_g(u)=0 \quad {\rm on} \quad M\backslash\Lambda\\
	\displaystyle\liminf_{x\rightarrow p}u(x)=\infty\quad \mbox{for each }p\in\Lambda.
	\end{cases}
	\end{equation}

When $(M^n,g)$ is conformally equivalent to the round sphere $(\mathbb S^n,g_0)$ and $\Lambda$ is a single point there exists no solution to \eqref{ourequation}, as it is was proved by C. S. Lin \cite{Lin98} and 
J. Wei and X. Xu \cite{MR1679783}.
The first result proving 
existence of solutions to the positive singular $Q$-curvature problem is due to S. Baraket and S. Rebhi \cite{MR1936047}. 
They provided a partial answer to the existence problem as follows

\begin{theoremletter}[\cite{MR1936047}]
For each natural number $k$ there exists a finite set $\Lambda$ 
with cardinality $2k$ such that $\mathbb{S}^n \backslash \Lambda$ carries 
a complete, constant $Q$-curvature metric conformal to the round metric. 
\end{theoremletter}

Using variational bifurcation theory and topological methods R. G. Bettiol, P. Piccione and Y. Sire \cite{MR4251294} 
obtained the existence of infinitely many 
complete metrics with constant positive $Q$-curvature on $\mathbb S^n\backslash\mathbb S^k$ with $0\leq k\leq {(n-4)}/{2}$,
conformal to the round metric.
Recently, for arbitrary closed Riemannian manifolds, A. Hyder and Y. Sire \cite{MR4170788}
studied the positive singular $Q$-curvature problem when $\Lambda$ is a smooth submanifold of $M$  
with Hausdorff dimension ${\rm dim}_{\mathcal H}(\Lambda)$ strictly between $0$ and 
$(n-4)/2$. Motivated by the results in \cite{MR1425579}, they established the 
following 

\begin{theoremletter}[\cite{MR4170788}]\label{Hyder-Sire}
Let $(M^n,g)$ be a closed Riemannian manifold with $n\geq 5$ with semi-positive $Q$-curvature and nonnegative
scalar curvature.
Let $\Lambda$ be a connected smooth closed submanifold of $M$ with $0<{\rm dim}_{\mathcal H}(\Lambda)<{(n-4)}/{2}$.  Then, there 
exists an infinite-dimensional family of complete metrics on $M\backslash\Lambda$ with positive constant $Q$-curvature.
\end{theoremletter}
    
    Let $(M^n,g)$ be a closed Riemannian manifold with dimension $n\geq 5$ and let $\Lambda$ be a finite set of points.
    Under two natural conditions on the metric, namely, nondegeneracy and vanishing of the Weyl tensor $W_g$ up to some suitable order, we prove  
    existence of a one-parameter family of solutions to \eqref{ourequation}. We say that $g$ is {\it nondegenerate} if the 
	linearized operator $L_g:C^{4,\alpha}(M)\rightarrow C^{0,\alpha}(M)$ is invertible for some $\alpha\in(0,1)$. Here
	\begin{equation}\label{linearization}
	    L_{g}u = P_{g}u - \frac{n(n^2-4)(n+4)}{16}u,
	\end{equation}
	that is, $L_g$ is the linearization of $H_g$ at the function $1$. Let $\mathcal M$ be the space of smooth metrics in $M$. 
	For $5\leq n\leq 7$,
	we define $\mathcal W_\Lambda^d=\left\{g\in\mathcal M : g\mbox{ is nondegenerate}\right\}$. For $n\geq 8$, and $d\in\mathbb N$, we set
	
	\begin{equation*}
	    \mathcal{W}^d_{\Lambda}=\left\{ g\in\mathcal M : \begin{aligned}\ &\mbox{$\nabla_{g}^jW_{g}(p)=0$ for all $p\in\Lambda$ and $j=0,\dots, d$},\\
	    &\mbox{\quad\quad\quad\quad\quad and $g$ is nondegenerate.}
        \end{aligned}\right\}.
	\end{equation*}

    Our main theorem  extends the results in \cite{MR2639545,MR2010322} to the context of constant $Q$-curvature metrics
	
	\begin{theorem}\label{maintheorem1}
	    Let $\left(M^{n}, g\right)$ be a closed Riemannian manifold with dimension $n\geq 5$ and let $\Lambda$ be a finite set of points.
	    Assume that $Q_{g}={n(n^2-4)}/{8}$ and $g\in\mathcal{W}^d_{\Lambda}$ for $d=\left[{(n-8)}/{2}\right]$.
	    Then, there exist a constant $\varepsilon_{0}>0$ and a one-parameter family of complete 
	    metrics $\{\tilde g_{\varepsilon}\}_{\varepsilon \in\left(0, \varepsilon_{0}\right)}$ on $M\backslash\Lambda$, 
	    such 
	    that each $\tilde g_\varepsilon$ is conformal to $g$ with $Q$-curvature equal to $n(n^2-4)/8$. 
	    Near each singular point $p \in \Lambda$ the metric $\tilde g_\varepsilon$ is 
	    asymptotic to one of the Delaunay metrics described in Section \ref{sec:delaunaysolutions}. 
	    Moreover, $\tilde g_{\varepsilon} \rightarrow g$ uniformly on compact sets 
	    of $M\backslash\Lambda$ as $\varepsilon \rightarrow 0$.
	\end{theorem}
	
In contrast with Theorem~\ref{Hyder-Sire}, the assumption that the scalar curvature is nonnegative is not necessary.
This set of hypotheses allows one to recover a maximum principle for the Paneitz operator, and it was first introduced by M. Gursky and A. Malchiodi \cite{MR3420504}.
Our techniques are point-fixed based and do not rely on such a principle.
This condition was also used in \cite{arXiv:2101.11304} to prove compactness of solutions to \eqref{ourequation} in the spherical setting.
 
	The vanishing condition on the Weyl tensor also appears in the scalar curvature setting. 
	Indeed, in \cite{MR2639545}, a similar hypothesis was used to prove existence of solutions to the Yamabe problem 
	with finitely many singularities, which generalizes the results due to A. Byde \cite{MR2010322} 
	in the conformally flat setting, at least around the singularities (see also \cite{MR929283}).
	Both results were extended for the case of fully nonlinear equations and strongly coupled systems \cite{MR3663325,arXiv:2009.01787}.
	The vanishing condition is also one of the essential pieces in the program proposed 
	by R. M. Schoen to establish compactness results for  the Yamabe 
	equation \cite{MR1144528,MR2477893}, and it comes naturally from the Weyl vanishing 
	conjecture \cite{MR1144528}; this was proved to be
	true for $3\leq n\leq 24$  \cite{MR2197144,MR2164927,MR2309836,MR2477893}, and disproved  otherwise \cite{MR2551136}. 
	
	The study of compactness for solutions to the $Q$-curvature equation is still under development.
	Recently, G. Li \cite{MR4028770} obtained compactness of solutions for $5\leq n\leq 7$ assuming the conditions in \cite{MR3420504} and without any hypothesis on the Weyl tensor.
	Independently, this was also proved by
	Y. Y. Li and J. Xiong \cite{MR3899029} for $5\leq n\leq 9$  assuming the Weyl tensor to vanish at singular points of a sequence of blowing-up solutions and the weaker hypothesis in \cite{MR3518237}.
    In addition, since the Weyl tensor and its covariant derivatives appear in the expansion of the 
	Green function for Paneitz operator, they observed that for $n\geq 8$ its vanishing at the singular points is required.
	In our situation, a similar phenomenon happens.
	When $5\leq n\leq 7$, we do not need any assumption on the Weyl
	tensor in Theorem \ref{maintheorem1}.
	Nevertheless, for $n\geq8$ 
	we assume that it vanishes up to order $\left[{(n-8)}/{2}\right]$ at singular points. 
	This order comes up naturally in our method but is not known to be  
	optimal (see Remark \ref{remark003}).
	
	We expect that the set of nondegenerate metrics is Baire generic with respect to the 
	Gromov-Hausdorff topology, similarly to what Beig, Chru\'sciel and Schoen \cite{BCS}
	proved in the scalar curvature setting. 
	
	We construct the metrics in Theorem \ref{maintheorem1} by gluing together 
	known examples of constant $Q$-curvature metrics, a technique that 
	usually requires a nondegeneracy assumption. 
	For example, Y.-J. Lin \cite{MR3333110} used a similar definition of nondegeneracy to perform a connected sum construction of manifolds with constant $Q$-curvature. The most important example of a degenerate manifold is the standard 
	round sphere $(\mathbb S^n,g_0)$.
	To see this, notice that the linearized operator \eqref{linearization} is given by
	$$L_{g_0}=\left(\Delta_{g_0}-\frac{n^2-4}{2}\right)\left(\Delta_{g_0}+n\right),$$
	which annihilates the restrictions of linear functions on $\mathbb R^{n+1}$ to $\mathbb S^n$. 
	
	As mentioned above, there exists no solution to the singular $Q$-curvature problem in the round sphere with a 
	one-point singularity. 
	In contrast with the conformally flat case,  for a generic choice of the background metric, we 
	can find solutions to \eqref{ourequation} in a closed manifold with one point removed.

To provide an example of a nondegenerate metric, consider the round sphere $(\mathbb S^m(k),g_k)$ with sectional 
curvature $k={(n-1)}{(m-1)}$ and $n:=2m\geq 5$. This implies that the spectrum of the Laplacian is given 
by $\operatorname{Spec}(\Delta_{g_k})=\{i(m+i-1)k:\;i=0,1,\ldots\}$.  With this normalization, the product metric $g=g_k+g_k$ in 
 $\mathbb S^m(k)\times \mathbb S^m(k)$ satisfies ${\rm Ric}_{g}=(n-1)g$, $R_g=n(n-1)$ and $Q_g={n(n^2-4)}/{8}$, and 
$L_{g}=(\Delta_{g}-\frac{n^2-4}
{2})(\Delta_{g}+n)$,
where $\Delta_g=\Delta_{g_k}+\Delta_{g_k}$. 
Hence, it is not difficult to show 
\[\operatorname{Spec}(\Delta_g)=\left\{\frac{n-1}{m-1}\left(i(i+m-1)+j(j+m-1)\right):\;i,j=0,1,\ldots\right\}.\] 
This implies that $n=2m\not\in\mbox{Spec}(\Delta_g)$. Therefore, $(\mathbb S^3(k)\times \mathbb S^3(k),g)$ 
is nondegenerate, and our main theorem applies. 
Notice that in this dimension we do not require any conditions on the 
Weyl tensor, and that this manifold is not locally conformally flat.

    We employ a standard gluing strategy based on mapping properties of the 
    linearized operator around an approximate solution in proving the existence of 
    solutions to \eqref{ourequation}. This strategy typically has three parts: interior analysis, exterior analysis, and gluing procedure. 
	First, we fix a ball of sufficiently small radius around each isolated singularity. 
	The interior analysis then consists of constructing a solution inside this ball being uniformly bounded and satisfying proper estimates on its norm.
	To this end, the hypothesis on vanishing the Weyl tensor up to some orders is necessary. 
	Second, in the exterior analysis, we use the nondegeneracy condition to study \eqref{ourequation} on 
	the complement of this ball, proving that the solution operator is also uniformly bounded.
	Finally, we prove the existence of an isomorphism, which we call the Navier-to-Neumann operator, 
	and use it to match the interior and exterior solutions up to third order on the boundary. This 
	is accomplished in parts, projecting the equation on its low and high Fourier modes.
	
	We encounter several difficulties in carrying out the strategy outlined in the previous 
	paragraph. 
	First, one needs to control the convergence rate of the Paneitz operator to the flat 
	bi-Laplacian in conformal normal coordinates. In dimension $n=5,6,7,$ this 
	convergence comes naturally from local expansions in conformal normal coordinates, 
	but in dimension $n \geq 8,$ we only obtain reasonable convergence rates after  
	constructing an auxiliary function. 
	Gluing the interior and exterior solutions on the boundary  up to third order is also
	problematic because of the lack of geometric Jacobi fields on the kernel of the linearized operator 
	in the low-frequency mode, making the system involving the gluing variables under-determined. 
	To circumvent this problem, we construct auxiliary functions and insert the missing 
	variables into the problem, making it solvable.
	
	For the sake of clarity, we provide an intuitive picture of our method. Figure \ref{fig:orig_summands}  
	shows summands for our gluing construction before any modifications, namely the manifold 
	minus a small geodesic ball and a punctured ball with the flat metric. 
	
	\begin{figure}[ht]
\pdfinclusioncopyfonts=1
\centering
\def\svgwidth{10.5cm}
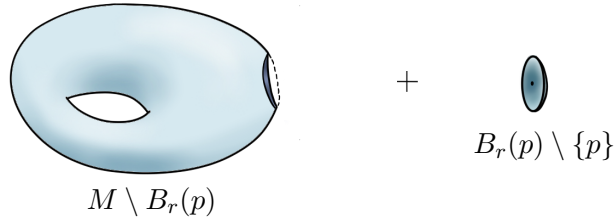
\caption{This figure shows the summands in their original states.}
\label{fig:orig_summands}
\end{figure}

    Figure \ref{fig:modified_summands} shows the summands after we have conformally deformed 
    each of them. We use the Green's function of the flat bi-Laplacian 
    on the left and the Delaunay metric on the right.

\begin{figure}[h]
    \centering
    \includegraphics[width=0.7\textwidth]{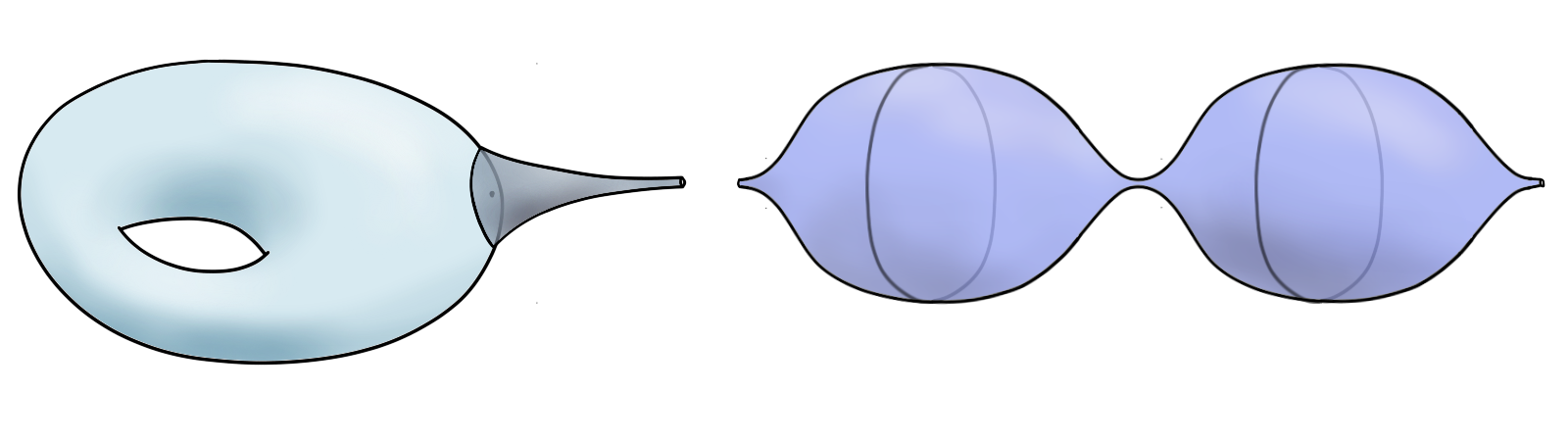}
    \caption{This figure shows the summands after conformal 
    modification. }
    \label{fig:modified_summands}
\end{figure}

Now we explain the summands in our gluing procedure more precisely.
The first summand is the underlying nondegenerate manifold with a point removed $M \backslash B_r(p)$, 
where $p\in \Lambda$ satisfies the Weyl vanishing hypotheses $\mathcal{W}_\Lambda^d$. 
The second summand is one end of a Delaunay metric. The Delaunay metric are all 
possible constant $Q$-curvature metrics on the cylinder $\R \times \Ss^{n-1}$ (or, equivalently, on 
$\Ss^n \backslash \{ p_1, p_2\}$) and we describe them in detail in Section \ref{sec:delaunaysolutions}. 
Most importantly, one can characterize a Delaunay metric by choosing a necksize parameter 
$\varepsilon$ and when this necksize is small the conformal factor of 
the Delaunay metric is close to the Green's
function for the flat bi-Laplacian. Thus, after conformally modifying the given 
metric $g$ on $M$ by a multiple of the Green's function whose pole is the gluing 
base point $p$, we see that the two summands are geometrically sufficiently close, making it 
possible for us to glue them together. 
After describing the Delaunay metrics, we introduce
some appropriate function spaces in Section \ref{sec:functionspaces}. 
Next, we 
construct and study our model operators in Section \ref{sec:poissonoperator}. Then, 
we construct our solution operators for the geometric problem as perturbations 
of these model operators. In this same section, we also prove that the associated 
Navier-to-Neumann operator is an isomorphism. In Section~\ref{sec:interioranalysis} we 
complete our interior analysis, using a fixed-point argument and the hypothesis on the 
Weyl tensor to obtain a solution to \eqref{ourequation} in the interior of a ball of small radius 
around a singular point with fine estimates. In Section~\ref{sec:exterioranalysis} we use 
the nondegeneracy hypothesis to find a solution to \eqref{ourequation} in the exterior 
of this ball.
In Section~\ref{sec:onegluingprocedure} we use the 
Navier-to-Neumann operator to match the interior and exterior solutions on the boundary 
of the small ball around a unique singular point. At the beginning of 
this section we explain why matching the interior and exterior solutions 
to third order across their common boundaries suffices to give 
us a smooth, global solution. 
In Section~\ref{sec:mainresults}, we discuss the necessary modifications 
to extend the gluing construction of the last section for the case of 
multiple points; this yields the proof of Theorem~\ref{maintheorem1}.

{\sc Acknowledgements:} We would like to thank the referees for 
a close read of and valuable comments regarding an earlier version of this 
manuscript. Additionally the first four authors would like to thank the last 
author for performing above and beyond the call of duty. 

	\section{Delaunay-type solutions}\label{sec:delaunaysolutions}
	We introduce the (logarithmic) cylindrical coordinates.
	We present the Delaunay-type solutions together with a classification result and some estimates.
    Finally, we consider deformations of the Delaunay-type solutions. These solutions will be used as approximating solutions to the interior problem.
    In what follows we will use dots to denote derivatives with respect to $t$.
    
	If $g=u^{\frac{4}{n-4}}\delta$ is a complete metric in $ (\mathbb R^n\backslash\{0\},\delta)$ 
	with constant $Q$-curvature $ Q_g=n(n^2-4)/8$, where $\delta$ is the Euclidean metric 
	in $\mathbb R^n$, then the function $u$ satisfies the following equation
	\begin{equation}\label{eq002}
	\Delta^2u=\frac{n(n^2-4)(n-4)}{16}u^{\frac{n+4}{n-4}}.
	\end{equation}
	It was proved by C.-S. Lin \cite{Lin98} that solutions to  \eqref{eq002} in $\mathbb R^{n}\backslash \{0\}$ with a nonremovable singularity at the origin are radially symmetric.
	If we consider the conformal diffeomorphism
	\begin{equation}\label{eq022}
	\Phi:(\mathbb R \times \mathbb S^{n-1},\, g_{\rm cyl}:=dt^{2} + d\theta^2)\rightarrow (\mathbb R^n\backslash\{0\}, \delta),
	\end{equation}
	given by $\Phi(t,\theta)=e^{-t}\theta$, we find $\Phi^*\delta=e^{-2t}g_{\rm cyl}$. Also, if $g=u^{\frac{4}{n-4}}\delta$, then $\Phi^*g=v^{\frac{4}{n-4}}g_{\rm cyl}$, where $$v(t)=e^{\frac{4-n}{2}t}u(e^{-t}\theta)=|x|^{\frac{n-4}{2}}u(x).$$
	
	In this logarithm cylindrical coordinates (also known as Endem-Fowler coordinates), 
	\eqref{eq002} is equivalent to 
	$$P_{\rm cyl} v = \frac{n(n-4)(n^2-4)}{16} v^{\frac{n+4}{n-4}},$$
	where 
	\begin{eqnarray*} 
	P_{\rm cyl} & = & \partial_t^4 + \Delta_{\mathbb{S}^{n-1}}^2 + 2 
	\Delta_{\mathbb{S}^{n-1}} \partial_t^2 - \frac{n^2-4n+8}{2} \partial_t^2 - \frac{n(n-4)}{2} \Delta_{\mathbb{S}^{n-1}} + \frac{n^2(n-4)^2}{16}. 
	\end{eqnarray*} 
	Restricting to radial functions reduces the last PDE to the ODE below
	\begin{equation}\label{eq001}
	\ddddot v-\frac{n^2-4n+8}{2}\ddot v+\frac{n^2(n-4)^2}{16}v-\frac{n(n-4)(n^2-4)}{16}v^{\frac{n+4}{n-4}}=0.
	\end{equation}
Following the notation of \cite{MR3869387}, we write this ODE as
	$$\ddddot v-A\ddot v=f(v),$$
	 with $A^2=4B+4(n-2)^2>4B$ and
	\begin{equation}\label{eq003}
	f(v)=Cv^{\frac{n+4}{n-4}}-Bv=\frac{n(n-4)}{16}v\left((n^2-4)v^{\frac{8}{n-4}}-n(n-4)\right),
	\end{equation}	
	where
\begin{equation}\label{eq019}
    A=\frac{n(n-4)+8}{2},\quad B=\frac{n^2(n-4)^2}{16}\quad\mbox{ and }\quad C=\frac{n(n-4)(n^2-4)}{16}.
\end{equation}

	\subsection{Classification results and some estimates}
	Now we will present some recent results concerning the classification of solutions to \eqref{eq001}.
	Notice that one can find a first integral for this ODE defined as
	\begin{equation}\label{eq004}
	H(x,y,z,w)=-xz+\frac{1}{2}y^2+\frac{n^2-4n+8}{4}z^2-\frac{n^2(n-4)^2}{32}w^2+\frac{(n-4)^2(n^2-4)}{32}w^{\frac{2n}{n-4}}.
	\end{equation}
    This Hamiltonian energy is constant along solutions to  \eqref{eq001}, {$H_\varepsilon=H(\dddot v_\varepsilon,\ddot v_\varepsilon,\dot v_\varepsilon,v_\varepsilon)=$ constant}. Let 
	$$v_{\rm sph}(t)=(\cosh t)^{\frac{4-n}{2}}\;\;\;{\rm and}\;\;\;v_{\rm cyl}=\left(\frac{n(n-4)}{n^2-4}\right)^{\frac{n-4}{8}}$$
	be the spherical and cylindrical solutions, whose energies are 
	\begin{equation*}
	   H_{\rm sph} = 0 \quad {\rm and} \quad H_{\rm cyl}=-\frac{(n-4)(n^2-4)}{8}\left(\frac{n(n-4)}{n^2-4}\right)^{\frac{n}{4}}<0,
	\end{equation*}
	respectively. {It was observed in \cite{arxiv.2001.07984} (see also \cite{arXiv:2002.05939}) that $H_\varepsilon<0$ and it is strictly decreasing function of $\varepsilon<0$.}

	In a recent paper, R. Frank and T. K\"{o}nig \cite{MR3869387} provided a full characterization to the solutions to  the ODE \eqref{eq001}, as stated below
	
	\begin{theoremletter}[\cite{MR3869387}]\label{Prop001}
		Let $v\in C^4(\mathbb R)$ be a solution to (\ref{eq001}). Then $\inf |v|\leq v_{\rm cyl},$ 
		with equality if and only if $v$ is a non-zero constant. Moreover, one of the following three alternatives holds:
		\begin{enumerate}
			\item[(a)] $v\equiv\pm v_{\rm cyl}$ or $v\equiv 0$.
			\item[(b)] $v(t)=\cosh^{\frac{4-n}{2}}(t-T)$ for some $T\in\mathbb R$.
			\item[(c)] Let $\varepsilon\in(0,v_{\rm cyl})$. Then there is a unique 
			(up to translations) bounded solution  $v_\varepsilon\in C^4(\mathbb R)$ 
			of (\ref{eq001}) with minimal value $\varepsilon$.  This solution is periodic, 
			has a unique local maximum and minimum per period and is symmetric 
			with respect to its local extrema.
		\end{enumerate}
	\end{theoremletter}
	
	We call $v_\varepsilon$ the \emph{Delaunay-type solution} 
	with necksize $\varepsilon$, and 
	denote its minimal period by $T_\varepsilon$. As $\varepsilon
	\rightarrow 0$ one has 
	$T_\varepsilon\rightarrow\infty$ and $v_\varepsilon(t+T_\varepsilon/2)
	\rightarrow \cosh^{\frac{4-n}{2}}(t)$ uniformly on compact sets. 
	Also, for all $\varepsilon\in(0,v_{\rm cyl})$, we have 
	that $0<v_\varepsilon<1$, 
	see \cite{arxiv.2001.07984} and \cite{arXiv:2003.03487}. We observe 
	that $\varepsilon=v_\varepsilon(0)
	=\min v_\varepsilon$, $\dot v_\varepsilon(0)=0$ and $\ddot v_\varepsilon(0)
	\geq 0$. We write the 
	corresponding solution to (\ref{eq002}) as $u_\varepsilon(x)=
	|x|^{\frac{4-n}{2}} v_\varepsilon(-\log|x|)$.
	
    Let then $\lambda$, $\mu$ and $\omega$ be real numbers such that 	$\lambda+\mu=A$, $\lambda\mu=\omega$ 
    and $A^2\geq 4\omega$.  We see that $\lambda$ and $\mu$ 	satisfy $\lambda^2-A\lambda+\omega=0$. Thus
	\begin{equation}\label{eq005}
	\lambda=\frac{A}{2}-\frac{1}{2}\sqrt{A^2-4\omega}\;\;\;\;\;\;{\rm and }\;\;\;\;\;\;\mu=\frac{A}{2}+\frac{1}{2}\sqrt{A^2-4\omega},
	\end{equation}
	implying that the equation (\ref{eq001}) is equivalent to
	\begin{equation}\label{eq027}
	\left\{\begin{array}{l}
	\ddot v-\lambda v=w\\
	\ddot w-\mu w=f(v)+\omega v,
	\end{array}\right.    
	\end{equation}
		where $f$ is defined in (\ref{eq003}). We can choose $\omega$ such that 
	\begin{equation*}
	B-\frac{n(n+4)(n^2-4)}{16}\varepsilon^{\frac{8}{n-4}}<\omega<\frac{A^2}{4}
	\end{equation*}
	and this implies that $f'(v)+\omega>0$.
	
	The proof of the next lemma is similar of \cite[Lemma 5]{BOUWEVANDENBERG} and we 
	include its proof for completeness.
	
	\begin{lemma}\label{lem002}
		Let $\lambda\leq \gamma\leq \mu$. If $v$ is a positive solution to (\ref{eq001}), then
		$$sign(\dddot v(t)-\gamma \dot v(t))=-sign(\dot v(t))$$
		for all $t\in \mathbb R$, where $sign (0)=0$.
	\end{lemma}
	\begin{proof}
		Let $t_0\in\mathbb R$. If $\dot v(t_0)=0$, {it follows by \cite[Corollary 5]{MR3869387} 
		that $v(t_0+t)=v(t_0-t)$ for all $t\in\mathbb R$. Thus} $\dddot v(t_0)=0$ the result follows. Then 
		assume that $\dot v(t_0)>0$. In this case, since $v$ is bounded there exist real numbers $t_1$ and 
		$t_2$ such that $\dot v(t_1)=\dot v(t_2)=0$ and $\dot v(t)>0$ for all $t\in(t_1,t_2)$.  Again we 
		have $\dddot v(t_1)=\dddot v(t_2)=0$. Define $y=\dddot v-\gamma \dot v$. Thus, using \eqref{eq027} we get
		$$\begin{array}{rcl}
		\dot y & = & \ddddot v-\gamma \ddot v=\ddddot v-\lambda \ddot v+(\lambda-\gamma)\ddot v=\ddot w+(\lambda-\gamma)\ddot v\\
		& = & \mu w+f(v)+\omega v+(\lambda-\gamma)\ddot v
		\end{array}$$
				and
	$$\begin{array}{rcl}
		\ddot y & = &\mu \dot w+\dot v(f'(v)+\omega)+(\lambda-\gamma)\dddot v\\
		& = & \mu(\dddot v-\lambda \dot v)+\dot v(f'(v)+\omega)+(\lambda-\gamma)\dddot v\\
		& = & \mu(y+\gamma \dot v-\lambda \dot v)+\dot v(f'(v)+\omega)+(\lambda-\gamma)(y+\gamma \dot v)\\
		& = & (A-\gamma)y+\dot v(f'(v)+\omega+ (\gamma-\lambda)(\mu-\gamma)).
	\end{array}$$
			This implies that
		$$\left\{\begin{array}{l}
		\ddot y-(A-\gamma) y=\dot v(f'(v)+\omega+(\gamma-\lambda)(\mu-\gamma))>0\\
		y(t_1)=y(t_2)=0.
		\end{array}\right.$$
		Therefore, by the strong maximum principle, $y<0$ in $(t_1,t_2)$. In particular, $y(t_0)<0$.
		
		If $\dot v(t_0)<0$, the proof is similar. 
	\end{proof}

	Using the Lemma \ref{lem002} and that the energy \eqref{eq004} is negative along the Delaunay solution we get that 
	\begin{equation}\label{eq029}
	    \left(\frac{A}{2}-\lambda\right)\dot v_\varepsilon^2+\frac{1}{2}\ddot v_\varepsilon^2\leq \dot v_\varepsilon(\dddot v_\varepsilon-\lambda\dot v_\varepsilon)+\frac{B}{2}v_\varepsilon^2-\frac{n-4}{2n}Cv_\varepsilon^{\frac{2n}{n-4}}< \frac{B}{2}v_\varepsilon^2,
	\end{equation}
	with $A/2-\lambda>0$, see \eqref{eq005}.
	
\begin{proposition}\label{Prop002}
For any $\varepsilon\in(0,v_{\rm cyl})$ and for all $t\geq 0$ the Delaunay-type solution $v_\varepsilon$ satisfies the estimates
\begin{align*}
    \left|v_\varepsilon(t)-\alpha_\varepsilon\cosh\left(\frac{n-4}{2}t\right)-\beta_\varepsilon \cosh\left(\frac{n}{2}t\right)\right|\leq & ~  c_n\varepsilon^{\frac{n+4}{n-4}}e^{\frac{n+4}{2}t},\\
    \left|\dot v_\varepsilon(t)-\frac{n-4}{2}\alpha_\varepsilon\sinh\left(\frac{n-4}{2}t\right)-\frac{n}{2}\beta_\varepsilon \sinh\left(\frac{n}{2}t\right)\right|\leq & ~  c_n\varepsilon^{\frac{n+4}{n-4}}e^{\frac{n+4}{2}t},\\
    \left|\ddot v_\varepsilon(t)-\left(\frac{n-4}{2}\right)^2\alpha_\varepsilon\cosh\left(\frac{n-4}{2}t\right)-\left(\frac{n}{2}\right)^2\beta_\varepsilon \cosh\left(\frac{n}{2}t\right)\right|\leq & ~  c_n\varepsilon^{\frac{n+4}{n-4}}e^{\frac{n+4}{2}t},\\
    \left|\dddot v_\varepsilon(t)-\left(\frac{n-4}{2}\right)^3\alpha_\varepsilon\sinh\left(\frac{n-4}{2}t\right)-\left(\frac{n}{2}\right)^3\beta_\varepsilon \sinh\left(\frac{n}{2}t\right)\right|\leq & ~  c_n\varepsilon^{\frac{n+4}{n-4}}e^{\frac{n+4}{2}t},\\
    \left|\ddddot v_\varepsilon(t)-\left(\frac{n-4}{2}\right)^4\alpha_\varepsilon\cosh\left(\frac{n-4}{2}t\right)-\left(\frac{n}{2}\right)^4\beta_\varepsilon \cosh\left(\frac{n}{2}t\right)\right|\leq & ~  c_n\varepsilon^{\frac{n+4}{n-4}}e^{\frac{n+4}{2}t},
\end{align*}
for all $t\geq 0$, where
\begin{align}
        \frac{n}{2(n-2)}\varepsilon<\alpha_\varepsilon &\displaystyle :=\frac{1}{2(n-2)}\left(\frac{n^2}{4}\varepsilon-\ddot v_\varepsilon(0)\right)<\frac{n}{4}\varepsilon\label{eq037}\\
         -\frac{n-4}{4}\varepsilon<\beta_\varepsilon&\displaystyle :=\frac{1}{2(n-2)}\left(\ddot v_\varepsilon(0)-\frac{(n-4)^2}{4}\varepsilon\right)<-\frac{(n-4)(n+2)}{8n}\varepsilon^{\frac{n+4}{n-4}}\label{eq0001},
\end{align}
	for some positive constants $C_n$ and  $c_n$  which depends only on $n$.
	\end{proposition}
	\begin{proof}
First, note that by \eqref{eq029} we have
\begin{equation}\label{segunda-derivada-v}
    |\ddot v_\varepsilon(0)|<\frac{n(n-4)}{4}\varepsilon
\end{equation}
and then we obtain \eqref{eq037}. First, note that if $v_\varepsilon$ is any solution to (\ref{eq001}), then
	$$e^{\frac{n}{2}t}\left(e^{-nt}\left(e^{2t}\left(e^{(n-4)t}\left(e^{\frac{4-n}{2}t}v_\varepsilon\right)'\right)'\right)'\right)'=Cv_\varepsilon^{\frac{n+4}{n-4}}.$$
	Thus
	\begin{align}
	v_\varepsilon(t)  = & \alpha_\varepsilon\cosh\frac{n-4}{2}t+\beta_\varepsilon\cosh\frac{n}{2}t\label{eq007}\\
	& +\frac{n(n-4)(n^2-4)}{16}e^{\frac{n-4}{2}t}\int_0^te^{(4-n)s}\int_0^se^{-2x}\int_0^xe^{ny}\int_0^ye^{-\frac{n}{2}z}v_\varepsilon^{\frac{n+4}{n-4}}dzdydxds. \nonumber
	\end{align}
Using that $v_\varepsilon\geq\varepsilon$ we obtain
\begin{align}
	v_\varepsilon(t)  & \geq  \alpha_\varepsilon\cosh\frac{n-4}{2}t+\beta_\varepsilon\cosh\frac{n}{2}t\nonumber\\
	& +\frac{n(n-4)(n^2-4)}{16}\varepsilon^{\frac{n+4}{n-4}}\left(\frac{2}{n^2(n-2)}\cosh\frac{n}{2}t-\frac{2}{(n-2)(n-4)^2}\cosh\frac{n-4}{2}t+\frac{16}{n^2(n-4)^2}\right),\nonumber
	\end{align}
for all $t\geq 0$. Since $\alpha_\varepsilon>0$ we conclude the upper bound in \eqref{eq037} and \eqref{eq0001}. Using \eqref{segunda-derivada-v} we obtain the lower bound of \eqref{eq0001}. Also, using \eqref{eq007} we get

\begin{equation}\label{v-bounded-by-cosh}
    0<v_\varepsilon(t)<\varepsilon\cosh\left(\frac{n-4}{2}t\right)\leq \varepsilon e^{\frac{n-4}{2}|t|}
\end{equation}

This and \eqref{eq007} implies the estimates.
	\end{proof}

	For the next result we will write $f=O^{(l)}(K|x|^k)$ to mean $|\partial_r^if|\leq c_nK|x|^{k-i}$, for all $i\in\{1,\ldots,l\}$ and for some positive constant $c_n$ which depends only on $n$. Here $K>0$ is a constant.
	
	\begin{corollary}\label{cor001}
    For any $\varepsilon\in(0,v_{\rm cyl})$ and any $x\in\mathbb R^n\backslash\{0\}$ with $|x|\leq 1$, the Delaunay-type solution $u_\varepsilon(x)$ satisfies the estimates
    $$u_\varepsilon(x)=\frac{\alpha_\varepsilon}{2}(1+|x|^{4-n})+\frac{\beta_\varepsilon}{2}(|x|^2+|x|^{2-n})+O^{(4)}(\varepsilon^{\frac{n+4}{n-4}}|x|^{-n}),$$
    where $\alpha_\varepsilon$ and $\beta_\varepsilon$ are given by \eqref{eq037}.
\end{corollary}
\begin{proof}
Remember that $u_\varepsilon(x)=|x|^{\frac{4-n}{2}}v_\varepsilon(-\log|x|)$. 			Besides, since $t=-\log |x|$, we get
\begin{align*}
    |x|^{\frac{4-n}{2}}\cosh\left(\frac{n-4}{2}t\right) & =\frac{1}{2}(1+|x|^{4-n}),\\
    |x|^{\frac{4-n}{2}}\cosh\left(\frac{n}{2}t\right) & =\frac{1}{2}(|x|^2+|x|^{2-n}).\\
\end{align*}
From this we get the expansion for $u_\varepsilon(x)$, since $0<|x|\leq 1$ implies that $t=-\log|x|>0$. Also we have
\begin{align*}
    |x|^{\frac{4-n}{2}}\sinh\left(\frac{n-4}{2}t\right) & =\frac{1}{2}(|x|^{4-n}-1)
\end{align*}
and
\begin{equation}\label{eq031}
    |x|\partial_ru_\varepsilon(x)=\frac{4-n}{2}u_\varepsilon(x)-|x|^{\frac{4-n}{2}}\dot v_\varepsilon(-\log|x|).
\end{equation}
Using the estimates of the derivative of $v_\varepsilon$ given by Proposition \ref{Prop002} we obtain the estimates for the radial derivatives of $u_\varepsilon$.
\end{proof}

	\color{black}

	\subsection{Variations of Delaunay-type solutions}
	Applying rigid motions to the solutions to  the PDE \eqref{eq002}, we can obtain some important families of solutions. We focus in two types of variations: translations along the Delaunay axis and translations at infinity. 
	
	We can obtain the first family of solutions by simply noting that if $u$ is a solution to \eqref{eq002}, 
	then $R^{\frac{4-n}{2}}u(R^{-1}x)$, for $R>0$, also solves this equation. Applying this transformation to
	a Delaunay-type solution yields the family
$$\mathbb R^+\ni R\mapsto |x|^{\frac{4-n}{2}}v_\varepsilon(-\log|x|+\log R).$$
	
	To construct the second family, define the standard inversion  $I:\mathbb R^n\backslash\{0\}
	\rightarrow\mathbb R^n\backslash\{0\}$ given by $I(x)=x|x|^{-2}$. It is well known that 
	$I^*\delta=|x|^{-4}\delta$. Hence, if $g=u^{\frac{4}{n-4}}\delta$, then $I^*g=
	(|x|^{4-n}u\circ I)^{\frac{4}{n-4}}\delta.$
	Recall that the Kelvin transform $\mathcal{K}$ is given by 
	$$\mathcal K(u)(x)=|x|^{4-n}u(x|x|^{-2}).$$ 
	
	By \eqref{transformationlaw}, we have
	$$\mathcal K(u)^{-\frac{n+4}{n-4}}P_\delta(\mathcal K(u))=(u\circ I)^{-\frac{n+4}{n-4}}P_\delta u\circ I,$$
	$$\Delta^2(\mathcal K(u))=\mathcal K(|x|^8\Delta^2u).$$
	
	Now suppose that $u$ is a solution to \eqref{eq002}, then
	$$\begin{array}{rcl}
	\displaystyle\Delta^2\mathcal K(u)(x) & = & \displaystyle \mathcal K(|x|^8\Delta^2u)(x)=\mathcal \mathcal K(c(n)|x|^8u^{\frac{n+4}{n-4}})
	\\
	& = & \displaystyle c(n)|x|^{-4-n}(u\circ I(x))^{\frac{n+4}{n-4}}= c(n)(\mathcal K(u)(x))^{\frac{n+4}{n-4}}.
	\end{array}$$
	Therefore, $\mathcal K(u)$ still is a solution to \eqref{eq002}. The function $u(x-a)$ also 
	solves \eqref{eq002}, but with a singularity in $a\in\mathbb R^n$ instead of the origin.

For the purpose of this paper, we consider the family of solutions 
	\begin{equation}\label{approx}
	u_{\varepsilon,R,a}(x)=\mathcal K(\mathcal K(u_\varepsilon)(\cdot-a))(x)=|x-a|x|^2|^{\frac{4-n}{2}}v_\varepsilon\left(-\log|x|+\log\left|\frac{x}{|x|}-a|x|\right|+\log R\right).  
	\end{equation}
This function has a singular point when $x=0$ and $x=a/|a|$.  In the particular case 
when $a=0$, as a direct consequence of Corollary \ref{cor001} the following asymptotic expansion

\begin{align}
    u_{\varepsilon,R}(x)  & =\frac{\alpha_\varepsilon}{2}\left(R^{\frac{4-n}{2}}+R^{\frac{n-4}{2}}|x|^{4-n}\right)+\frac{\beta_\varepsilon}{2} (R^{-\frac{n}{2}}|x|^2+R^{\frac{n}{2}}|x|^{2-n})+O^{(4)}(R^{\frac{n+4}{2}}\varepsilon^{\frac{n+4}{n-4}}|x|^{-n}),\label{eq059}
\end{align}
where $\alpha_\varepsilon$ and $\beta_\varepsilon$ are given by \eqref{eq037}. 
\color{black}

\begin{proposition}\label{lem006}
There exists a constant $r_0\in(0,1)$, such that for any $x$ and $a$ in $\mathbb R^n$ with $|x|\leq 1$, $|a||x|<r_0$    , $R>0$ and $\varepsilon\in(0,v_\varepsilon)$ the solution $u_{\varepsilon,R,a}$ satisfies the estimate
\begin{equation}\label{eq032}
    u_{\varepsilon,R,a}(x)  =  u_{\varepsilon,R}(x)+\left((n-4)u_{\varepsilon,R}(x)+|x|\partial_ru_\varepsilon(x)\right)\langle a,x\rangle+O(|a|^2|x|^{\frac{8-n}{2}}),
\end{equation}
and if $R\leq |x|$ the estimate
\begin{equation}\label{eq033}
    u_{\varepsilon,R,a}(x)  =  u_{\varepsilon,R}(x)+\left((n-4)u_{\varepsilon,R}(x)+|x|\partial_ru_\varepsilon(x)\right)\langle a,x\rangle+O(|a|^2\varepsilon R^{\frac{4-n}{2}}|x|^{2}).
\end{equation}
\end{proposition}
\begin{proof}
First, using the Taylor expansion we get
\begin{equation}\label{eq028}
    |x-a|x|^2|^{\frac{4-n}{2}}=|x|^{\frac{4-n}{2}}+\frac{n-4}{2}\langle a, x\rangle |x|^{\frac{4-n}{2}}+O(|a|^2|x|^{\frac{8-n}{2}})
\end{equation}
and
\begin{equation}\label{eq0003}
    \log\left|\frac{x}{|x|}-a|x|\right|=-\langle a, x\rangle+O(|a|^2|x|^2),
\end{equation}
for $|a||x|<r_0$ and some $r_0\in(0,1)$. Also
\begin{align*}
    v_\varepsilon & \left(-\log|x|  +\log\left|\frac{x}{|x|}-a|x|\right|+\log R\right)= v_\varepsilon\left(-\log|x|+\log R\right) \\
   &   +\dot v_\varepsilon(-\log|x|+\log R)\log\left|\frac{x}{|x|}-a|x|\right|+\ddot v_\varepsilon(-\log|x|+\log R+t_{a,x})\left(\log\left|\frac{x}{|x|}-a|x|\right|\right)^2\\
   & =v_\varepsilon\left(-\log|x|+\log R\right) -\dot v_\varepsilon(-\log|x|+\log R)\langle a,x\rangle+\dot v_\varepsilon(-\log|x|+\log R)O(|a|^2|x|^2)\\
   & +\ddot v_\varepsilon(-\log|x|+\log R+t_{a,x})O(|a|^2|x|^2),
\end{align*}
for some $t_{a,x}\in\mathbb R$ with $0<|t_{a,x}|<\left|\log\left|\frac{x}{|x|}-a|x|\right|\right|$. Note that $t_{a,x}\rightarrow 0$ when $|a||x|\rightarrow 0$. 

Note that, using \eqref{eq029} and \eqref{segunda-derivada-v} it follows that 
\begin{equation}\label{eq048}
|\dot v_\varepsilon|\leq c_nv_\varepsilon\;\;\;\;{\rm and }\;\;\;\;  |\ddot v_\varepsilon|\leq c_nv_\varepsilon
\end{equation}
for some positive constant $c_n$ that depends only on $n$. Thus, by \eqref{approx} and \eqref{eq028} we obtain
$$u_{\varepsilon,R,a}(x)  =  u_{\varepsilon,R}(x)+\left(\frac{n-4}{2}u_{\varepsilon,R}(x)-|x|^{\frac{4-n}{2}}\dot v_\varepsilon(-\log|x|+\log R)\right)\langle a,x\rangle+O(|a^2||x|^{\frac{8-n}{2}}).$$
By \eqref{eq031} this implies \eqref{eq032}. Moreover, by \eqref{eq029} and \eqref{v-bounded-by-cosh}, since $-\log|x|+\log R\leq 0$ for $R\leq |x|$, then we get that $v_\varepsilon(-\log|x|+\log R)$ and $v_\varepsilon(-\log|x|+\log R+t_{a,x})$ are bounded by $c_n\varepsilon R^{\frac{4-n}{2}}|x|^{\frac{n-4}{2}}$, for some positive constant $c_n$ which depends only on $n$. Therefore, we get \eqref{eq033}.
\end{proof}

	\section{Function spaces}\label{sec:functionspaces}
In this section, we define some function spaces that will be useful in this work. The first one is 
the weighted H\"older spaces in the punctured ball. It is well-established in the literature that 
these spaces are the most convenient spaces to define the linearized operator. The second one appears 
so naturally in our results that it is more helpful to put its definition here. Finally, the third 
one is the weighted H\"older spaces in which the exterior analysis will be carried out. These 
are the same weighted spaces as in \cite{Mazzeo-pacard-1999, MR2194146}.

\begin{definition}\label{def1}
For each $k\in\mathbb{N}$, $r>0$, $0<\alpha<1$ and $\sigma\in(0,r/2)$, let 
$u\in C^k(B_r(0)\backslash\{0\})$, set
$$\|u\|_{(k,\alpha),[\sigma,2\sigma]}=\sup_{|x|\in[\sigma,2\sigma]}
\left(\sum_{j=0}^{k}\sigma^j |\nabla^ju(x)|\right)+\sigma^{k+\alpha}
\sup_{|x|,|y|\in[\sigma,2\sigma]} \frac{|\nabla^ku(x)-\nabla^ku(y)|}{|x-y|^{\alpha}}.$$
Then, for any $\mu\in\mathbb{R}$, the space $C^{k,\alpha}_{\mu}(B_r(0)\backslash\{0\})$ 
is the collection of functions $u$ that are locally in $C^{k,\alpha}(B_r(0)
\backslash\{0\})$ and for which the norm
$$\|u\|_{(k,\alpha),\mu,r}=\sup_{0<\sigma\leq\frac{r}{2}}\sigma^{-\mu} 
\|u\|_{(k,\alpha),[\sigma,2\sigma]}$$
is finite.
\end{definition}

The one result about these that we shall use frequently, and without comment, 
is that to check if a function $u$ is an element of some $C_{\mu}^{0,\alpha}$, say, 
it is sufficient to check that $|u(x)|\leq C|x|^\mu$ and $|\nabla u(x)|\leq C|x|^{\mu-1}$. 
In particular, the function $|x|^{\mu}$ is in $C_{\mu}^{k,\alpha}$ for any $k,$ $\alpha$, or $\mu$.

Notice that $C_{\mu}^{k,\alpha}\subseteq C_{\delta}^{l,\alpha}$ if $\mu\geq\delta$ 
and $k\geq l$, and $\|u\|_{(l,\alpha),\delta}\leq C\|u\|_{(k,\alpha),\mu}$ for 
all $u\in C_{\mu}^{k,\alpha}$.

\begin{definition}
For each $k\in\mathbb{N}$, $0<\alpha<1$ and $r>0$. Let $\phi\in C^k(\mathbb{S}_r^{n-1})$, set
$$\|\phi\|_{(k,\alpha),r}:=\|\phi(r\cdot)\|_{C^{k,\alpha}(\mathbb{S}^{n-1})}.$$
Then, the space $C^{k,\alpha}(\mathbb{S}^{n-1}_r)$ is the collection of functions $\phi\in C^k(\mathbb{S}^{n-1}_r)$ for which the norm $\|\phi\|_{(k,\alpha),r}$ is finite.
\end{definition}

Next, consider an $n$-dimensional compact Riemannian manifold $(M^n, g)$ and  
a coordinate system $\Psi: B_{r_1}(0) \rightarrow M$ on $M$ centered at a point 
$p = \Psi(0) \in M$, where $B_{r_1}(0)\subset\mathbb{R}^n$ is the ball of radius $r_1$. For $0<r<s\leq r_1$ define
$M_r:=M\backslash\Psi(B_r(0))$ and $\Omega_{r,s}:=\Psi(A_{r,s}),$
where $A_{r,s}:=\{x\in\mathbb{R}^n;r\leq|x|\leq s\}$.

\begin{definition}\label{def4}
For all $k\in\mathbb{N}$, $\alpha\in(0,1)$ and $\nu\in\mathbb{R}$, the space $C_{\nu}^{k,\alpha}(M\backslash\{p\})$ is the space of functions $v\in C_{\rm loc}^{k,\alpha}(M\backslash\{p\})$ for which the following norm is finite
$$\|v\|_{C_\nu^{k,\alpha}(M\backslash\{p\})}:= \|v\|_{C^{k,\alpha}(M_{r_1/2})} +\|v\circ\Psi\|_{(k,\alpha),\nu,r_1},$$
where the norm $\|\cdot\|_{(k,\alpha),\nu,r_1}$ is the one defined in Definition \ref{def1}.
\end{definition}

For all $0<r<s\leq r_1$, we can also define the spaces $C_\mu^{k,\alpha}(\Omega_{r,s})$ and $C_\mu^{k,\alpha}(M_r)$ to be the space of restriction of elements of $C_\mu^{k,\alpha}(M\backslash\{p\})$ to $M_r$ and $\Omega_{r,s}$, respectively. These spaces is endowed with the following norm
$$\|f\|_{C^{k,\alpha}_\mu(\Omega_{r,s})}:=\sup_{r\leq \sigma\leq\frac{s}{2}}\sigma^{-\mu}\|f\circ \Psi\|_{(k,\alpha),[\sigma,2\sigma]}$$
and
$$\|h\|_{C^{k,\alpha}_\mu(M_r)}:=\|h\|_{C^{k,\alpha}(M_{r_1/2})}+ \|h\|_{C^{k,\alpha}_\mu(\Omega_{r,r_1})}.$$
Notice that these norms are independent of the extension of the functions $f$ and $h$ to $M_r$.
	
\section{Model operators}\label{sec:poissonoperator}

In this section, we follow \cite[Sections 11 and 13]{jleli} to construct 
a Poisson operator for the bi-Laplacian in the Euclidean space, which 
will be crucial to perform the gluing construction in Section~\ref{sec:onegluingprocedure}. 

Let us set some notation.
We introduce $(e_j,\lambda_j)$, 
	the eigendata of the Laplacian on $\mathbb S^{n-1}$. That 
	is, $\Delta_{\mathbb S^{n-1}}e_j+\lambda_je_j=0$. We assume 
	that the eigenvalues are counted with multiplicity and 
	that the eigenfunctions are normalized so that their $L^2$-norm 
	is equal to 1. It is well known that $(e_j,-\lambda_j^2)$ 
	are the eigendata of the bi-Laplacian on $\mathbb S^{n-1}$, 
	that is, $\Delta_{\mathbb S^{n-1}}^2e_j=\lambda_j^2e_j$. 
	Thus, for $\phi\in L^2(\mathbb S^{n-1})$ we use the following
decomposition
	$$\phi(\theta)=\sum_{j=0}^\infty 
	\phi_j e_j(\theta).$$
	
    Define the orthogonal projections 
    $$\pi', \pi'' : L^2 (\Ss_{r}^{n-1}) 
    \rightarrow L^2 (\Ss_{r}^{n-1})$$
    by 
    \begin{equation} \label{fourier_mode_projections}  
    \pi' \left ( \sum_{j=0}^\infty \phi_j e_j \right ) 
    = \sum_{j=0}^n \phi_j e_j, \quad \pi''(\phi) = \phi-\pi'(\phi).
    \end{equation}
It is appropriate to consider the product of H\"{o}lder spaces 
$C_{\delta}^{4,\alpha}(\mathbb{S}_r^{n-1},\mathbb{R}^2)
:=C_{\delta}^{4,\alpha}(\mathbb{S}_r^{n-1})\times C_{\delta}^{4,\alpha}(\mathbb{S}_r^{n-1})$, where $r>0$, 
$\alpha\in(0,1)$ and $\delta\in\mathbb{R}$. We also write 
the projection onto the high Fourier modes as 
\begin{equation*}
    \pi''(C^{4,\alpha}(\mathbb{S}_r^{n-1},\mathbb{R}^2)):=
    \pi''(C^{4,\alpha}(\mathbb{S}_r^{n-1}))\times 
    \pi''(C^{4,\alpha}(\mathbb{S}_r^{n-1})).
\end{equation*}

\subsection{Interior Poisson operator}
First, we can prove the existence of the Poisson operator for 
the interior of the punctured unit ball.

	\begin{proposition}\label{poissoninterior}
		Let $0<\alpha<1$ be a fixed constant. There exists a bounded 
		linear operator $
		    \mathcal P_1: \pi'' (C^{4,\alpha}(\mathbb{S}^{n-1}))
		    \times C^{4,\alpha}(\mathbb{S}^{n-1}) \rightarrow 
		    C^{4,\alpha}_{2}(B_1(0)\backslash\{0\}),$
		such that for all $\phi_0 \in \pi'' (C^{4,\alpha} 
		(\mathbb{S}^{n-1}))$ and $\phi_2 \in C^{4,\alpha} 
		(\mathbb{S}^{n-1})$ it holds
	\begin{equation}\label{eq075}
	    \left\{
	\begin{array}{rclcc}
		\Delta^2 \mathcal P_1(\phi_0,\phi_2) & = & 0  &{\rm in } 
		& B_1(0)\backslash\{0\} \\
		\pi''(\mathcal P_1(\phi_0,\phi_2)) & = 
		&\displaystyle \phi_0 & {\rm on } & \partial B_{1} \\ 
		\Delta \mathcal P_1(\phi_0,\phi_2) & = & \phi_2 & 
		{\rm on } & \partial B_{1}.
		\end{array}
		\right.
	\end{equation}
Moreover, if $\pi''(\phi_2)=0$, then
\begin{equation}\label{eq083}
    \mathcal P_1(0,\phi_2)=
    |x|^2 \left ( \frac{(\phi_2)_0}{2n} + \frac{1}
    {6n-4} \sum_{j=1}^n (\phi_2)_j x_j \right ). 
\end{equation}

	\end{proposition}
	
\begin{proof}
By linearity, we reduce the proof to two cases.

\medskip

\noindent{\bf Case 1:} $\phi_2\equiv0$.

\medskip

By \cite[Proposition 11.25]{jleli}  
and \cite[Lemma 6.2]{MR1763040}, there exists a smooth function $v_{\phi_0}\in C_2^{2,\alpha}(B_1(0)\backslash\{0\})$ satisfying the following
$$\begin{cases}
    \Delta v_{\phi_0}=0&\quad\mbox{ in } B_1(0)\backslash\{0\}\\
    \pi''(v_{\phi_0})=\phi_0 &\quad\mbox{ on }\partial B_1.
\end{cases}$$
Thus, we define $\mathcal P_1(\phi_0,0):=v_{\phi_0}$.

\medskip

\noindent{\bf Case 2:} $\phi_0\equiv0$. This case is divided into two steps.

\medskip

\noindent{\bf Step 1:} $\pi''(\phi_2)=0$.

\medskip

By a simple calculation we see that the function in \eqref{eq083} 
satisfies \eqref{eq075} and
$$\sup_{B_1\backslash\{0\}}|x|^{-2}\mathcal P_1(0,\phi_2)
\leq c \|\phi_{1}\|_{C^{4,\alpha}(\mathbb{S}^{n-1})}$$
for some constant $c>0$ independent of $\phi_2$.

\medskip
	
\noindent{\bf Step 2:} $\pi'(\phi_2)=0$.

Initially, by linearity, we may assume that $\|{\phi}_2\|_{{C}^{4, \alpha}
\left(\mathbb{S}^{n-1}\right)}=1$. 
Hence the solution to the problem \eqref{eq075}, denoted by $v_0$,
can be obtained as the limit $\rho\rightarrow0$ of solutions $v_{\rho}$ to 

\begin{equation}\label{eq084}
    \begin{cases}
        \Delta^2v_{\rho}=0 & \quad {\rm in} \quad B_1\setminus B_{\rho}\\
        v_{\rho}=0 & \quad {\rm on} \quad \partial B_1\\
        \Delta v_{\rho}=\phi_2 & \quad {\rm on} \quad \partial B_1\\
        \Delta v_{\rho}=v_{\rho}=0 & \quad {\rm on} \quad \partial B_\rho.\\
    \end{cases}
\end{equation}

\noindent{\bf Claim 1:} There exists a constant $c>0$, independent 
of $\phi_2$ and $\rho$, such that
$$\sup_{B_1\setminus B_{\rho}}|x|^{-2}\left|v_{\rho}\right|\leq c.$$

Arguing by contradiction, there would exist a sequence of numbers 
$\{\rho_i\}$ and functions $v_{\rho_i}$ solving \eqref{eq084} for $\rho_i$ 
with $\sup_{B_1\setminus B_{\rho_i}}|x|^{-2}\left|
v_{\rho_i}\right|\rightarrow\infty$ as $i\rightarrow\infty$. 
For each $r\in(\rho_i,1)$ the function $v_{\rho_i}(r\cdot)$ is 
$L^2$-orthogonal to $e_0,\ldots,e_n$ on $\mathbb S^{n-1}$. Now 
choose $(r_i,\theta_i)\in(\rho_i,1)\times\mathbb S^{n-1}$ such that
$$A_i:=\sup_{\left(\rho_i, 1\right) \times\mathbb{S}^{n-1}}r^{-2}
\left| v_{\rho_i}(r\theta)\right|=r_i^{-2}\left| v_{\rho_i}(r_i\theta_i)
\right| \rightarrow\infty \quad \mbox{as} \quad i\rightarrow\infty.$$

For $|x|\in [\rho_ir_i^{-1},r_i^{-1}]$, define
\begin{equation*}
    \widetilde{v}_{\rho_i}(x)=r_i^{-2}{A_i^{-1}} 
    v_{\rho_i}\left(r_ix\right)
\end{equation*}
which satisfies
\begin{equation}\label{eq085}
    \sup_{\left(\rho_ir_i^{-1},r_i^{-1}\right)\times
    \mathbb{S}^{n-1}}r^{-2}\left| 
    \widetilde{v}_{\rho_i}(r\theta)\right|=|\widetilde 
    v_{\rho_i}(\theta_i)|=1.
\end{equation}

Arguing as in the proof of Lemma \ref{lem004}, in page \pageref{pg001}, 
we conclude that the sequences $\rho_ir_i^{-1}$ and $r_i^{-1}$ remain 
bounded away from 1. Up to a subsequence, we can assume that the 
sequence $\rho_ir_i^{-1}\rightarrow \tau_1\in[0,1)$ and $r_i^{-1}
\rightarrow\tau_2\in(1,+\infty]$. Using \eqref{eq085} and Schauder 
estimates like in \cite{Gazzola_Grunau_Sweers_2010}, by Arzelà--Ascoli 
Theorem we can suppose that $\widetilde v_{\rho_i}$ converges to 
some biharmonic function $v_\infty\in C^4(B_{\tau_2}\backslash 
B_{\tau_1})$ satisfying 
$v_\infty(\theta_\infty)=1$, for some $\theta_\infty$. If $\tau_1 
= 0$ we consider $B_0 = \{ 0 \}$. Moreover, 
$v_\infty=\Delta v_\infty=0$ on the boundary and by \eqref{eq085} we 
conclude that $v_\infty$ is not identically zero and $|v_\infty(x)|
\leq |x|^2$. Using an analogous argument as in the end of the proof 
of the Lemma \ref{lem004}, we obtain a contradiction.

The estimates for the derivatives follow from Schauder's regularity (see \cite{Gazzola_Grunau_Sweers_2010, Nicolaescu_2007}).
This finishes the proof of the proposition.
\end{proof}

Using a simple scaling argument, we can also define the Poisson operator in a more general context.

\begin{corollary}\label{cor:poissoninterior}
Let $0<\alpha<1$, $\mu \leq 2$ and $r>0$ be a fixed constants. There 
exists a bounded linear operator
$\mathcal P_r: \pi'' (C^{4,\alpha}(\mathbb{S}_r^{n-1}))\times 
C^{4,\alpha}(\mathbb{S}_r^{n-1})) \rightarrow 
C^{4,\alpha}_{2}(B_r(0)\backslash\{0\})$ satisfying that for each 
$\phi_{0}\in \pi''(C^{4,\alpha}(\mathbb S_r^{n-1}))$ and $\phi_2
\in C^{4, \alpha} (\mathbb{S}_r^{n-1})$
    \begin{align*}
		\left\{
		\begin{array}{rclcc}
		\Delta^2 \mathcal P_r(\phi_{0},\phi_{2}) & = & 0  
		&{\rm in } & B_{r}(0)\backslash\{0\} \\
		\pi''(\mathcal P_r(\phi_{0},\phi_{2})) & = &  
		\displaystyle  \phi_0 & {\rm on } & \partial B_{r} \\ 
		\Delta \mathcal P_r(\phi_{0},\phi_{2}) & = & r^{-2}\phi_{2} & 
		{\rm on } & \partial B_{r} .\\
		\end{array}
		\right.
	\end{align*}
    Moreover, there exists a constant $C>0$ independently of $r$ such that
        \begin{equation*}
            \left\|\mathcal{P}_{r}\left(\phi_{0},\phi_{2}\right)
            \right\|_{(4, \alpha), \mu, r} \leq C r^{-2}
            \left\|(\phi_{0},\phi_{2})\right\|_{(4, \alpha), r},
        \end{equation*}
        and if $\pi''(\phi_2)=0$, then
$$
    \mathcal P_r(0, \phi_2)= r^{-2}|x|^2 \left ( \frac{(\phi_2)_0}{2n} 
    + \frac{1}{6n-4} r^{-1}\sum_{j=1}^n (\phi_0)_j x_j \right ). 
  $$
\end{corollary}

\begin{proof}
    Let us consider 
    \begin{equation*}
        \mathcal{P}_{r}\left(\phi^r_{0},\phi^r_{2}\right)(x)=
        \mathcal{P}_{1}(\phi_0,\phi_2)\left(r^{-1} x\right),
    \end{equation*}
    where $\phi^r_i(\theta):=\phi_{i}(r \theta)$ for $i=0,2$. 
    This operator is obviously bounded and satisfies
    \begin{equation*}
        \left\|\mathcal{P}_{r}\left(\phi^r_{0},\phi^r_{2}\right)
        \right\|_{(4, \alpha), \mu, r}=r^{-\mu}
        \left\|\mathcal{P}_{1}(\phi^r_{0},\phi^r_{2})
        \right\|_{(4, \alpha), \mu, 1},
    \end{equation*}
    which concludes the proof.
\end{proof}

\subsection{Exterior Poisson operator}\label{EPO}

    Similar to the last proposition, we prove the existence of 
    a Poisson operator for the exterior problem.
    This time we denote our boundary data as $\psi_0, \psi_2 \in C^{4,\alpha}(\mathbb{S}^{n-1})$. 
	
	\begin{proposition} For each $0<\alpha<1$ there exists a bounded linear operator 
		$
		    	\mathcal{Q}_1: C^{4,\alpha}(\mathbb S^{n-1},\mathbb{R}^2)
		    	\rightarrow C^{4,\alpha}_{4-n}(\mathbb R^{n}\backslash B_1)
		$
		such that for all $(\psi_0,\psi_2)\in C^{4,\alpha}
		(\mathbb S^{n-1},\mathbb{R}^2)$
		\begin{align}\label{eq076}
		\left\{
		\begin{array}{rclcc}
		\Delta^2 \mathcal Q_1(\psi_0,\psi_2) & = & 0  &{\rm in } 
		& \mathbb R^{n}\backslash B_1\\
		\mathcal Q_1(\psi_0,\psi_2) & = & \psi_0 & 
		{\rm on } & \partial B_{1} \\ 
		\Delta \mathcal Q_1(\psi_0,\psi_2) & = & \psi_2 
		& {\rm on } & \partial B_{1} .\\
		\end{array}
		\right.
		\end{align}
		More specifically, we have that
		\begin{equation}\label{eq086}
		    \mathcal Q_1(\psi_0,\psi_2)= |x|^{2-n} 
		    \sum_{j=0}^\infty |x|^{-j} 
		    \left ( (\psi_0)_j + \frac{(\psi_2)_j}{c_2(n,j)} (|x|^2 - 1) \right ), 
		\end{equation}
		where 
		$$c_2(n,j) = \left ( 4-n+\frac{n-\sqrt{n^2 +8+4\lambda_j}}{2} 
		\right ) \left ( 2 + \frac{n - \sqrt{n^2+8+4\lambda_j}}{2} 
		\right ) $$ 
		and $\lambda_j$ is the $j$th eigenvalue of $\Delta_{\mathbb{S}^{n-1}}$.  
	\end{proposition}
	\begin{remark} Recall that the eigenvalues of 
	$\Delta_{\mathbb{S}^{n-1}}$ all have the form $\lambda_j 
	= l (n-2+l)$ for some $l = 0,1,2,3,\dots$ In this form, 
	the expression for $c_2(n,j)$ simplifies to $(4-n-l)(2-l)$. 
	\end{remark} 
	
\begin{proof} 
Since the expression \eqref{eq086} satisfies \eqref{eq076}, we need
only to prove  that
$$\|\mathcal Q_1(\psi_0,\psi_2)\|_{(4,\alpha),4-n}\leq 
c\|(\psi_0,\psi_2)\|_{(4,\alpha)}.$$
To prove this, we consider two cases as follows.

\medskip

\noindent{\bf Case 1:} $\psi_2\equiv0$.
    
    \medskip
    
This case is similar to the Case 1 of Proposition \ref{poissoninterior}. See 
\cite[Proposition 13.25]{jleli}.

\medskip

\noindent{\bf Case 2:} $\psi_0\equiv0$.
    
    \medskip
    
    We verify directly that there exists $c>0$ such that
    \begin{equation*}
        \sup _{\mathbb{R}^{n}\backslash B_1} 
        |x|^{n-4}|(|x|^{4-n}-|x|^{2-n})(\psi_2)_0| 
        \leq c \|\psi_{2}\|_{L^{\infty}\left(\mathbb S^{n-1}\right)}.
    \end{equation*}
    Notice that since $\mathbb{S}^{n-1}$ is compact, using 
    Morrey's inequality, we have $H^s(\mathbb{S}^{n-1})\hookrightarrow C^{0,\alpha}
        (\mathbb{S}^{n-1}) \hookrightarrow L^{\infty}(\mathbb{S}^{n-1}),$
    for $s>n/2$ and $\alpha=1-\frac{n}{2s}$.
    Considering $(\psi_{2})_0= 0$, by the definition of 
    Sobolev fractional norm on the sphere (see, for instance, 
    \cite[page 406]{MR2119285}), there exists $c>0$ independently 
    of $r>0$, such that
    \begin{align*}
        \sup_{\mathbb{S}^{n-1}} r^{n-4}|\mathcal 
        Q_1(0, \psi_2)(r\;\cdot)|\leq c r^{n-4}\|
        \mathcal Q_1(0, \psi_2)(r\;\cdot)\|_{H^{s}
        \left(\mathbb S^{n-1}\right)}\leq c
        \left[\sum_{j=1}^{\infty}\left(1+\lambda_{j}^{s}
        \right)\left|(\psi_2)_j\right|^{2}r^{-j}\right]^{1/2}.
    \end{align*}
    Also, since $j>0$, observe that
    \begin{equation*}
        \sup _{r \geq 2} \sup _{j \geq 1}
        \left(1+\lambda_{j}^{s}\right)r^{-j}<\infty.
    \end{equation*}
    Thus, we have
    \begin{equation*}
        \sup _{\mathbb{R}^{n}\backslash B_{2}} 
        |x|^{n-4}|\mathcal Q_1(0,\psi_2)| \leq 
        c\|\psi_{2}\|_{L^{2}\left(\mathbb{S}^{n-1}\right)} 
        \leq c\|\psi_2\|_{L^{\infty}
        \left(\mathbb{S}^{n-1}\right)}.
    \end{equation*}
    The maximum principle in \cite{MR312040}, then implies that
    \begin{equation*}
        \sup _{\mathbb{R}^{n}\backslash B_{1}} 
        |x|^{n-4}|\mathcal Q_1(0,\psi_2)| \leq c\|\psi_2\|_{L^{\infty}
        \left(\mathbb{S}^{n-1}\right)}.
    \end{equation*}
   Finally, using the last estimate, we can use standard (rescaled) 
   elliptic estimates to prove the estimates for the derivatives, 
   which completes the proof of the proposition.
\end{proof}	

\begin{corollary}\label{cor:poissonexterior}
Let $0<\alpha<1$ and $r>0$. 
There exists a bounded linear operator 
		\begin{equation*}
		    	\mathcal{Q}_r: C^{4,\alpha}
		    	(\mathbb S^{n-1}_r,\mathbb{R}^2)\rightarrow 
		    	C^{4,\alpha}_{4-n}(\mathbb R^{n}\backslash B_r)
		\end{equation*}
		satisfying that for each $(\psi_{0},\psi_{2})\in 
		C^{4,\alpha}(\mathbb S_r^{n-1},\mathbb{R}^2)$
		\begin{align*}
		\left\{
		\begin{array}{rclcc}
		\Delta^2 \mathcal Q_r(\psi_{0},\psi_{2}) & = & 0  
		&{\rm in } & \mathbb R^{n}\backslash B_r\\
		\mathcal Q_r(\psi_{0},\psi_{2}) & = & \psi_{0} 
		& {\rm on } & \partial B_{r} \\ 
		\Delta \mathcal Q_r(\psi_{0},\psi_{2}) & = 
		& r^{-2}\psi_{2} & {\rm on } & \partial B_{r} .\\
		\end{array}
		\right.
		\end{align*}
		Moreover, there exists a constant $C>0$ such that
		\begin{equation*}
		    \|\mathcal Q_r(\psi_{0},\psi_{2})
		    \|_{C^{4,\alpha}_{4-n}(\mathbb R^{n}
		    \backslash B_1)}\leq Cr^{n-4}\|(\psi_{0},\psi_{2})
		    \|_{C^{4,\alpha}(\mathbb S_r^{n-1},\mathbb{R}^2)}.
		\end{equation*}
\end{corollary}

    \begin{proof}
        As before, let us consider 
         \begin{equation*}
            \mathcal{Q}_{r}\left(\psi^r_{0},\psi^r_{2}\right)
            (x)=\mathcal{Q}_{1}(\psi_0,\psi_2)\left(r^{-1} x\right),
        \end{equation*}
        where $\psi^r_i(\theta):=\psi_{i}(r \theta)$ for $i=0,2$.
        The proof now follows the same lines as in the interior case.
    \end{proof}

\subsection{Navier-to-Neumann operator}

    Finally we prove the main result of this section, which is 
    the construction of an isomorphism we call the 
    Navier-to-Neumann operator. 
    To this end, we combine a Liouville-type result for bounded 
    entire biharmonic functions and a standard result from the 
    theory of pseudodifferential operators. 

	\begin{proposition}
	The operator $\mathcal{Z}_1:\pi''(C^{4,\alpha}(\mathbb S^{n-1},\mathbb{R}^2))
	\rightarrow \pi''(C^{1,\alpha}(\mathbb S^{n-1},\mathbb{R}^2))$ given by
		\begin{equation*}
		\mathcal{Z}_1(\phi_0,\phi_2) =
		\left(\partial_r(\mathcal{P}_1(\phi_0,\phi_2) 
		- \mathcal{Q}_1(\phi_0,\phi_2)), \partial_r
		(\Delta \mathcal{P}_1(\phi_0,\phi_2) - \Delta \mathcal{Q}_1(\phi_0,\phi_2)) \right) 
		\end{equation*}
		is an isomorphism.
	\end{proposition}
	
	\begin{proof}
	   Notice that 
	   $\mathcal{Z}_1$ is a linear third order elliptic 
	   self-adjoint pseudodifferential operator with principal 
	   symbol $\sigma_{\mathcal{Z}_1}(\xi)=-2(|\xi|,|\xi|^3)$.
	   Hence, since  $\sigma_{\mathcal{Z}_1}(\xi)\neq0$ 
	   whenever $\xi\neq0$, $\mathcal{Z}_1$ is surjective 
	   by a classical result from the 
	   theory of linear operators.
	   
	   Now it is sufficient to prove that this operator is also 
	   injective. Let us take $(\phi_0,\phi_2) \in 
	   {C}^{4, \alpha}\left(\mathbb{S}^{n-1},\mathbb{R}^2\right)$ 
	   satisfying $\mathcal{Z}_1(\phi_0,\phi_2)=0$ and define 
	   $w_{(\phi_0,\phi_2)} \in C^{1,\alpha}(\mathbb{R}^n)$ as
	   \begin{equation*}
	   w_{(\phi_0,\phi_2)}(x)=
	       \begin{cases}
             \overline{w}_{(\phi_0,\phi_2)}, 
             \; {\rm if} \;  x\in B_1\\
             \widetilde{w}_{(\phi_0,\phi_2)} 
             \; {\rm if} \;  x\in \mathbb{R}^n\setminus B_1,

	       \end{cases}
	   \end{equation*}
	   where 
	   $$\overline{w}_{(\phi_0,\phi_2)}
	   :=r^{1-n} \sum_{j=2}^\infty \left ( (\phi_0)_j e_j 
	    + \frac{(\phi_2)_j}{c_2(n,j)} (r^2-1) \right ). $$ 
	   and $\widetilde{w}_{(\phi_0,\phi_2)}$ is the bounded 
	   entire biharmonic function defined 
	   in the exterior region $\mathbb{R}^n\setminus B_1$. 
	   Using the Lioville theorem in \cite{MR274792}, it follows 
	   that $w\equiv0$ in $C^{1,\alpha}(\mathbb{R}^n)$, and so 
	   $(\phi_0,\phi_2)\equiv 0$ in ${C}^{4, \alpha}
	   \left(\mathbb{S}^{n-1},\mathbb{R}^2\right)$,which 
	   proves the proposition.
	\end{proof}
	
	As in the proofs of 
	Corollaries \ref{cor:poissoninterior} and \ref{cor:poissonexterior} 
	we obtain
	$$r\partial_r(\mathcal P_r(\phi^r_{0},\phi^r_{2})
	-\mathcal Q_r(\phi^r_{0},\phi^r_{2}))(r\theta)=\partial_r
	(\mathcal P_1(\phi_{0},\phi_{2})-
	\mathcal Q_1(\phi_{0},\phi_{2}))(\theta)$$
	and
	$$r^3\partial_r\Delta(\mathcal P_r(\phi^r_{0},\phi^r_{2})-
	\mathcal Q_r(\phi^r_{0},\phi^r_{2}))(r\theta)=
	\partial_r(\Delta\mathcal P_1(\phi_{0},\phi_{2})-
	\Delta\mathcal Q_1(\phi_{0},\phi_{2}))(\theta).$$
	Therefore, we obtain
	\begin{corollary}\label{cor002}
	The operator $\mathcal{Z}_r:\pi''(C^{4,\alpha}
	(\mathbb S_r^{n-1},\mathbb{R}^2))\rightarrow 
	\pi''(C^{1,\alpha}(\mathbb S_r^{n-1},\mathbb{R}^2))$ given by
		\begin{equation*}
		\mathcal{Z}_r(\phi_0,\phi_2) 
		= \left(r\partial_r(\mathcal{P}_r(\phi_0,\phi_2) 
		- \mathcal{Q}_r(\phi_0,\phi_2)),
		r^3 \partial_r(\Delta \mathcal{P}_r(\phi_0,\phi_2) 
		- \Delta \mathcal{Q}_r(\phi_0,\phi_2)) \right) 
		\end{equation*}
		is an isomorphism, with norm bounded independent of $r$.
	\end{corollary}
	
\section{Interior problem}\label{sec:interioranalysis}

Following the strategy described in the introduction, our goal in this section is 
to construct a solution in a small punctured ball centered at the singular point. 
Due to the conformal invariance of the $Q$-curvature equation, we will use conformal 
normal coordinates centered at the singular point. In these coordinates, we show 
the metric is sufficiently close to the Euclidean metric using the vanishing assumption 
on the Weyl tensor. We use these coordinates to show that the first terms in the expansion of $P_g(u_{\varepsilon,R})$ are $L^2$-orthogonal to the low Fourier modes on the sphere, see Lemma \ref{lem008}. This allows us to use an auxiliary function which will be important to define the fixed point problem. By the expansion of 
the metric, it should be possible to perturb a Delaunay-type solution 
to an exact solution to equation \eqref{ourequation}.

\subsection{Conformal normal coordinates}\label{sec001}
	
Let $(M,g_0)$ be a smooth Riemannian manifold of dimension $n\geq 5$. 
Given a point $p\in M$, 
there exist $r>0$ and a smooth function $\mathcal F$ on $M$ such 
that the conformal metric $g=\mathcal F^{\frac{4}{n-4}}g_0$ satisfies 
\begin{equation}\label{eq038}
    \det g=1,\quad {\rm in} \quad B_r(p)
\end{equation}
in $g$-normal coordinates. In these coordinates $\mathcal F=1+O(|x|^2)$
and $g = \delta + O(|x|^2)$. Additionally, 
\begin{equation}\label{eq036}
    R_{ij}(0)=0,\;\;
    R_g=O(|x|^2)\;\;{\rm and }\;\;\Delta_gR_g(0)=-\frac{1}{6}|W_g(0)|^2,
\end{equation}
where $W_g$ is the Weyl tensor of the metric $g$. 
J. M. Lee and T. H. Parker \cite{MR888880} first proved the existence 
of such coordinates, and 
later J. G. Cao \cite{MR991959} and M. G\"{u}nther \cite{MR1225437} 
derived refined expansions of various curvature terms in 
conformal normal coordinates. 

\begin{remark}
For the remainder of the paper we let $d=\left[\dfrac{n-8}{2}\right]$ 
for $n\geq 8$ and $d=-1$ for $5\leq n\leq 7$. 
\end{remark}

Using the results in \cite{MR1427765} and \cite{MR1216009} we obtain 
improved estimates on curvature terms if we assume that the Weyl 
tensor vanishes to sufficiently high 
order. If the Weyl tensor satisfies  
\begin{equation}\label{eq026}
	    \nabla^lW_{g}(p)=0,\;\;\mbox{for }l=0,1,\ldots,d,
	\end{equation}
then	
\begin{equation}\label{eq035}
	g=\delta+O(|x|^{d+3}).
	\end{equation}

This implies that $R_g=O(|x|^{\max\{2,d+1\}})$ and ${\rm Ric}_g=O(|x|^{\max\{1,d+1\}})$. 
Therefore, by \eqref{q_curvature}, we have $Q_g=O(|x|^{\max\{0,d-1\}})$

Note that the Weyl tensor is conformally invariant, which implies that the 
condition \eqref{eq026} is satisfied for any metric in the conformal class of $g$. 
Also, for dimensions $5\leq n\leq 7$, the condition on $W_g$ is 
vacuous.

In conformal normal coordinates, we will always write
\begin{equation*}
    g=\exp(h),
\end{equation*}
where $h$ is a two-tensor on $\mathbb R^n$ satisfying  
$h_{ij}(x)x_j=0$ and $tr(h(x))=0$ (See \cite{MR2477893}). Note 
that by $\eqref{eq035}$, we have $h=O(|x|^{d+3})$.

Motivated by results in \cite{MR1719213,MR2425176} we prove the following lemma.

\begin{lemma}
   We have
\begin{equation}\label{eq039}
    R_{jk}=\frac{1}{2}(\partial_i\partial_jh_{ki}+\partial_i\partial_kh_{ij}
    -\partial_i\partial_ih_{jk})+O(|\partial h|^2)+O(|h||\partial^2 h|)
\end{equation}
   and
   \begin{equation}\label{eq040}
       R_g=\partial_i\partial_jh_{ij}+O(|\partial h|^2)+O(|h||\partial^2 h|)
   \end{equation}
\end{lemma}
\begin{proof}
First we observe that $g^{-1}=\exp(-h)$. The Ricci tensor is given by
$$R_{jk}=\partial_i\Gamma_{jk}^i-\partial_j\Gamma_{ik}^i+\Gamma_{jk}^p\Gamma_{ip}^i-\Gamma_{ik}^p\Gamma_{jp}^i.$$
Note that $\Gamma_{ik}^i=\frac{1}{2}g^{il}\partial_kg_{il}=\frac{1}{2}\partial_k\log\det g=0$, 
by \eqref{eq038}. Therefore, $R_{jk}=\partial_i\Gamma_{jk}^i-\Gamma_{ik}^p\Gamma_{jp}^i,$ 
where the Christoffel symbols are given by
$\Gamma_{jk}^i=\frac{1}{2}g^{il}(\partial_jg_{kl}+\partial_kg_{lj}-\partial_lg_{jk}).$
Thus, $\Gamma_{ij}^k=O(|\partial h|)$ and
$$\partial_i\Gamma_{jk}^i=\frac{1}{2}(\partial_i\partial_jh_{ki}
+\partial_i\partial_kh_{ij}-\partial_i\partial_ih_{jk})
+O(|\partial h|^2)+O(|h||\partial^2 h|).$$
This implies \eqref{eq039}. Since $R_g=g^{jk}R_{jk}$ and $tr(h(x))=0$ we obtain the result.
\end{proof}

\begin{lemma}\label{lem011}
    Let $\{ x_i\}$ be conformal normal coordinates 
    centered at $p$. If $u$ is  radial with respect to these
    coordinates then in a small geodesic ball centered at $p$ it holds
    \begin{equation}\label{eq041}
     \nabla_i\nabla_ju=\frac{x_ix_j}{r^2}u''-\frac{x_ix_j}{r^3}u'+\frac{\delta_{ij}}{r}u'+O(|\partial h|)|u'|, 
    \end{equation}
    and
    \begin{equation*}
        \Delta_gu=\Delta u=u''+\frac{n-1}{r}u',
    \end{equation*}
    where $\Delta$ is the Euclidean laplacian and we use $'$ to 
    denote differentiation with respect to $r$. 
\end{lemma}

The equation \eqref{eq041} follows from \eqref{eq038} 
and \cite[Lemma 2.6]{MR3420504}.

\begin{lemma}\label{lem007}
    We have
$x_jx_k\partial_i\partial_jh_{ki}=x_jx_k\Delta h_{jk}=x_ix_j\partial_i\partial_kh_{jk}=0.$
\end{lemma}
\begin{proof}
Recall that $h_{ki}x_k=0$ and $tr h=0$. Taking a derivative, we get 
$\partial_jh_{ki}x_k+h_{ij}=0$ and
\begin{equation}\label{eq044}
x_k\partial_ih_{ki}=0,   
\end{equation}
which implies that $\partial_jh_{ki}x_kx_j=0$. Differentiating 
again we see $\partial_i\partial_jh_{ki}x_kx_j+\partial_jh_{ii}x_j+
\partial_ih_{ki}x_k=0$, and so  
$x_jx_k\partial_i\partial_jh_{ki}=0$. 
Thus $\partial_ih_{jk}x_j+h_{ik}=0$, which implies that 
$\partial_ih_{jk}x_jx_k=0$. Finally, 
$\partial_i\partial_ih_{jk}x_jx_k+\partial_ih_{ik}x_k+\partial_ih_{ji}x_j=0$. 
Therefore $x_jx_k\Delta h_{jk}=0$.
\end{proof}

\begin{lemma}
   The functions $\partial_i\partial_jh_{ij}$,  
   $\Delta\partial_i\partial_jh_{ij}$ and $x_k\partial_k\partial_i
   \partial_jh_{ij}$ are $L^2$-orthogonal to the functions 
   $\{1,x_1,\ldots,x_n\}$ in $\mathbb S^{n-1}_r$.
\end{lemma}
\begin{proof}
First observe that $\partial_i\partial_jh_{ij}=
\operatorname{div}(\partial_jh_{ij}e_i)$ and $\partial_jh_{kj}
=\operatorname{div}(h_{kj}e_j)$. Using integration by parts 
and \eqref{eq044}, we get
\begin{equation}\label{eq045}
    \int_{B_r}\partial_i\partial_jh_{ij}
    =\frac{1}{r}\int_{\mathbb S^{n-1}_r}x_i\partial_jh_{ij}=0
\end{equation}
and
$$\int_{B_r}x_k\partial_i\partial_jh_{ij}=-\int_{B_r}
\partial_jh_{kj}+\frac{1}{r}\int_{\mathbb S^{n-1}_r}
x_i\partial_jh_{ij}=-\frac{1}{r}
\int_{\mathbb S^{n-1}_r}h_{kj}x_j=0.$$

Since $\partial_k\partial_i\partial_jh_{ij}
=\operatorname{div}(\partial_k\partial_jh_{ij}e_i)$ 
and $\partial_k\partial_jh_{lj}
=\operatorname{div}(\partial_kh_{lj}e_j)$, using 
Lemma \ref{lem007}, \eqref{eq044} and \eqref{eq045}, 
we obtain
$$\int_{B_r}x_k\partial_k\partial_i\partial_jh_{ij}
=-\int_{B_r}\partial_k\partial_jh_{kj}+\frac{1}{r}
\int_{\mathbb S^{n-1}_r}x_kx_i\partial_k\partial_jh_{ij}=0$$
and
$$\begin{array}{rcl}
     \displaystyle \int_{B_r}x_lx_k\partial_k\partial_i
     \partial_jh_{ij} & = & \displaystyle -\int_{B_r}
     (x_l\partial_k\partial_jh_{kj}+x_k\partial_k
     \partial_jh_{lj})+\frac{1}{r}\int_{\mathbb S^{n-1}_r}
     x_lx_kx_i\partial_k\partial_jh_{ij} \\
     \\
     & = &\displaystyle -\int_{B_r}x_k\partial_k\partial_jh_{lj}
     =\int_{B_r}\partial_kh_{lk}-\frac{1}{2}
     \int_{\mathbb S^{n-1}_r}x_kx_j\partial_kh_{lj}=0.
\end{array}$$

By the previous results we get
$$\int_{B_r}\Delta\partial_i\partial_jh_{ij}
=\frac{1}{r}\int_{S_r}x_k\partial_k\partial_i\partial_jh_{ij}=0$$
and
$$\int_{B_r}x_l\Delta\partial_i\partial_jh_{ij}
=-\int_{B_r}\partial_l\partial_i\partial_jh_{ij}
=-\frac{1}{r}\int_{\mathbb S^{n-1}_r}
x_j\partial_l\partial_ih_{ij}=\frac{1}{r}
\int_{\mathbb S^{n-1}_r}\partial_ih_{il}=0.$$
\end{proof}

\begin{lemma}\label{lem008}
In conformal normal coordinates at $p$, for sufficiently small $r=|x|$, one has
$$P_gu_{\varepsilon,R}=\Delta^2 u_{\varepsilon,R}+\mathscr{P} (u_{\varepsilon,R})+\mathscr{R}$$
where $\Delta$ is the Euclidean laplacian and
$$\begin{array}{rcl}
     \mathscr{P}(u_{\varepsilon,R}) & = &\displaystyle  \frac{4}{n-2}
     \left(\frac{u'_{\varepsilon,R}}{r} -\frac{(n-2)^2+4}{8(n-1)}
     \Delta u_{\varepsilon,R}\right)\partial_i\partial_jh_{ij}
     +\frac{6-n}{2(n-1)}x_i\partial_i\partial_j\partial_kh_{jk}
     \frac{u'_{\varepsilon,R}}{r}\\
     \\
     & & \displaystyle -\frac{n-4}{4(n-1)}u_{\varepsilon,R}
     \Delta\partial_i\partial_jh_{ij},
\end{array}$$
is $L^2$-orthogonal to the functions $\{1,x_1,\ldots,x_n\}$ and $\mathscr{R}=O(r^{d+4-\frac{n}{2}+\max\{d,0\}})$.
\end{lemma}
\begin{proof}
By \eqref{paneitz-branson} 
\begin{equation}\label{eq049}
    \begin{array}{rcl}
P_gu_{\varepsilon,R} &= &\displaystyle \Delta^{2}_{g}
u_{\varepsilon,R} + \frac{4}{n-2}\langle {\rm Ric}_g, \nabla^{2}
u_{\varepsilon,R}\rangle -\frac{(n-2)^2+4}{2(n-1)(n-2)}
R_g\Delta_gu_{\varepsilon,R} \\
& &\displaystyle + \frac{6-n}{2(n-1)}\langle 
\nabla R_g,\nabla u_{\varepsilon,R}\rangle+ \frac{n-4}{2}
Q_{g}u_{\varepsilon,R}.    \end{array}
\end{equation}
Now by \eqref{eq039} and \eqref{eq041}, we obtain
\begin{align*}
     \langle {\rm Ric}_g, \nabla^{2}u_{\varepsilon,R}\rangle 
     = & ~ g^{ij}g^{kl}R_{jk}\nabla_i\nabla_lu_{\varepsilon,R}
     \\
     = & ~ \displaystyle g^{ij}g^{kl}R_{jk}\left(\frac{x_ix_l}{r^2}u_{\varepsilon,R}''
     -\frac{x_ix_l}{r^3}u_{\varepsilon,R}'+\frac{\delta_{il}}{r}u_{\varepsilon,R}'
     +O(|\partial h|)|u_{\varepsilon,R}'|\right)
     \\
     = & ~ \displaystyle R_g\frac{u_{\varepsilon,R}'}{r}+ R_{jk}x_jx_k
     \left(\frac{u_{\varepsilon,R}''}{r^2}-\frac{u_{\varepsilon,R}'}{r^3}\right) +O(|h||Ric|)
     \left(\frac{|u_{\varepsilon,R}'|}{r}+|u_{\varepsilon,R}''|\right)\\
     & 
     +O(|Ric||\partial h|)|u_{\varepsilon,R}'|.
\end{align*}
Using that $h=O(r^{d+3})$, $Ric=O(r^{\max\{1,d+1\}})$ and $u_{\varepsilon,R}=O(r^{\frac{4-n}{2}})$, as well \eqref{eq039} and \eqref{eq040} we obtain
\begin{equation}\label{eq050}
    \begin{array}{rcl}
\langle {\rm Ric}_g, \nabla^{2}u_{\varepsilon,R}\rangle& = 
& \displaystyle \partial_i\partial_jh_{ij}\frac{u_{\varepsilon,R}'}{r}+
\left(\partial_i\partial_jh_{ki}x_jx_k
-\frac{1}{2}\partial_i\partial_ih_{jk}x_jx_k\right)
\left(\frac{u_{\varepsilon,R}''}{r^2}-\frac{u_{\varepsilon,R}'}{r^3}\right)
\\
& & +O(r^{d+4-\frac{n}{2}+\max\{d,0\}})
\\
& = & \displaystyle \partial_i\partial_jh_{ij}\frac{u_{\varepsilon,R}'}{r} +O(r^{d+4-\frac{n}{2}+\max\{d,0\}}),
\end{array}
\end{equation}
where in the last equality we used \eqref{eq044} and \eqref{eq045}. 
Now observe that
$$\langle \nabla R_g,\nabla u_{\varepsilon,R} \rangle = g^{ij}
\nabla_iR_g\nabla_ju_{\varepsilon,R}=
\partial_i\partial_j\partial_kh_{jk}x_i
\frac{u'_{\varepsilon,R}}{r}+O(r^{2d+4-\frac{n}{2}}).$$

Using \eqref{q_curvature} and \eqref{eq040} we get
$$R_g\Delta_gu_{\varepsilon,R}=\partial_i
\partial_jh_{ij}\Delta u_{\varepsilon,R}+O(r^{2d+4-\frac{n}{2}})$$
and
$$Q_g=-\frac{1}{2(n-1)}\Delta\partial_i\partial_jh_{ij}+O(r^{2d+2}).$$
This finishes the proof.

\end{proof}

\begin{lemma}\label{lem009}
In conformal normal coordinates one has 
$$(\Delta^2-\Delta_g^2)u_{\varepsilon,R,a}=O(|x|^{d+2-\frac{n}{2}}),$$
where $\Delta$ is the Euclidean Laplacian.
\end{lemma}
\begin{proof}
First we write
$$(\Delta^2-\Delta_g^2)u_{\varepsilon,R,a}=(\Delta^2-\Delta_g^2)
(u_{\varepsilon,R,a}-u_{\varepsilon,R})+(\Delta^2-\Delta_g^2)
u_{\varepsilon,R}.$$

Since $u_{\varepsilon,R}$ is a radial function, 
using \eqref{eq038} and Lemma 2.6 in \cite{MR3420504} we have $\Delta_g^2
u_{\varepsilon,R}=\Delta^2u_{\varepsilon,R}$. 

Now, note that by Lemma \ref{lem006} we have $v=u_{\varepsilon,R,a}(x)-u_{\varepsilon,R}(x)=
O(|x|^{\frac{6-n}{2}})$. Using the Laplacian in coordinates 
and \eqref{eq035} we obtain
    \begin{equation}\label{eq067}
        \Delta_gv=\Delta v+\partial_ig^{ij}\partial_jv+(g^{ij}-\delta^{ij})
        \partial_i\partial_jv,
    \end{equation}
which implies
\begin{equation}\label{eq052}
    \Delta_g^2v=\Delta^2v+\partial_ig^{ij}\partial_j\Delta v
    +(g^{ij}-\delta^{ij})\partial_i\partial_j\Delta v
    +\Delta_g(\partial_ig^{ij}\partial_jv+(g^{ij}
    -\delta^{ij})\partial_i\partial_jv).
\end{equation}

This implies the result.
\end{proof}

	\subsection{Linear analysis}
		In order to solve the singular Q-curvature problem we seek  functions $u$ that 
		are close to a function $u_0$ such that $(u+u_0)(x)\rightarrow +\infty$ as 
		$x\rightarrow p$, $u+u_0$ is a positive function, and $H_g(u_0 + u) = 0$. This 
		is done by considering 
		the linearization about $u_0$ of the operator $H_g$, defined in \eqref{eq020}, 
		which is given by
	\begin{equation}\label{eq021}
	L_{g}^{u_{0}}(u) = \frac{\partial}{\partial t}\bigg|_{t=0}H_{g}
	(u_0 + tu)= P_{g}u - \frac{n(n^2-4)(n+4)}{16}u_0^{\frac{8}{n-4}}u .
	\end{equation}
	Thus we can write
	\begin{equation*}
	H_{g}(u_0 + u) = H_{g}(u_0) + L_{g}^{u_0}(u) + \mathcal R^{u_0}(u),
	\end{equation*}
	where the nonlinear term $\mathcal R^{u_0}(u)$ is independent of the metric and given by
	\begin{equation}\label{eq024}
	\begin{array}{rcl}
	    \mathcal R^{u_{0}}(u) & = & \displaystyle -\frac{n(n^2-4)(n-4)}{16}\left[|u_0+u|^{\frac{8}{n-4}}(u_0+u) - u_0^{\frac{n+4}{n-4}}- \frac{n+4}{n-4}u_{0}^{\frac{8}{n-4}}u\right].
	\end{array}
	    	\end{equation}
	
	This section is devoted to functional properties of the linearized operator around a Delaunay 
	solution in $\mathbb R^n$. In this case the $Q$-curvature operator \eqref{eq020} is given by
	$$H_\delta(u)=\Delta^2u-\frac{n(n^2-4)(n-4)}{16}|u|^{\frac{8}{n-4}}u,$$
	and its linearization \eqref{eq021} around a solution $u_{\varepsilon,R,a}$ 
	(see \eqref{approx}), denoted by  $L_{\varepsilon,R,a}$, is given by
	\begin{equation}\label{eq025}
	L_{\varepsilon,R,a}(v)=\Delta^2 v-\frac{n(n^2-4)(n+4)}{16}u_{\varepsilon,R,a}^{\frac{8}{n-4}}v.
	\end{equation}
	We abbreviate $L_{\varepsilon,R,0} = L_{\varepsilon, R}$ and $L_{\varepsilon, 1,0} = 
	L_{\varepsilon}$. 
	
	Remember that if $g=u^{\frac{4}{n-4}}\delta$, 
	then $\Phi^*g=v^{\frac{4}{n-4}}g_{\rm cyl}$, where
	$v(t)=e^{\frac{4-n}{2}t}u(e^{-t}\theta)=|x|^{\frac{n-4}{2}}u(x)$, 
	{\it i.e.}  $u(x)=|x|^{\frac{4-n}{2}}v(-\log|x|)$, where 
	$\Phi$ is defined in \eqref{eq022}. Hence, using \eqref{eq023} 
	we see
$$L_{\varepsilon,R}(u)=|x|^{-\frac{n+4}{2}}\mathcal L_{\varepsilon,R}(e^{\frac{n-4}{2}t}u\circ\Phi)\circ\Phi^{-1},$$
where
$$\mathcal L_{\varepsilon,R}(v) =  P_{g_{\rm cyl}}(v)-\displaystyle\frac{n(n+4)(n^2-4)}{16}v_{\varepsilon,R}^{\frac{8}{n-4}}v$$
and
	$$P_{g_{\rm cyl}}(v)=\Delta_{g_{\rm cyl}}^2v
	-\frac{n(n-4)}{2}\Delta_{g_{\rm cyl}}v-
	4\partial_t^2 v+\frac{n^2(n-4)^2}{16}v.$$

	Using $\Delta_{g_{\rm cyl}}v=\partial_t^2 v+
	\Delta_{\mathbb S^{n-1}}v$ we obtain
	\begin{equation}\label{eq034}
	    \begin{array}{rcl}
	\displaystyle\mathcal L_{\varepsilon,R}(v) & = 
	&\displaystyle \partial_t^4v +\Delta_{\mathbb S^{n-1}}^2v
	+2\Delta_{\mathbb S^{n-1}}\partial_t^2v-
	\frac{n(n-4)}{2}\Delta_{\mathbb S^{n-1}}v-
	\frac{n(n-4)+8}{2}\partial_t^2 v
	\\
	& & \displaystyle +\frac{n^2(n-4)^2}{16}v-
	\frac{n(n+4)(n^2-4)}{16}v_{\varepsilon,R}^{\frac{8}{n-4}}v.
	\end{array}
	\end{equation}
	
	Now we focus on solving the following linear problem
	with Navier boundary conditions
	\begin{equation}\label{eq008}
	\left\{\begin{array}{lcl}
	\mathcal L_{\varepsilon}(w) = f & {\rm in} & D_R\\
	\Delta w=w=0  & {\rm on} & \partial D_R,
	\end{array}
	\right.
	\end{equation}
	where $D_R:=(-\log R,\infty)\times\mathbb S^{n-1}$ 
	and $\mathcal L_\varepsilon$ is given by \eqref{eq034} with $R=1$.
	
	We will follow closely the program due to 
	R. Mazzeo and F. Pacard \cite{Mazzeo-pacard-1999}. Due to the geometry 
	of the domain, it is natural to approach our problem by 
	means of a classical separation of variables, decomposing 
	both $w$ and $f$ in Fourier series and using the 
	projections \eqref{fourier_mode_projections}. 

	Projecting the linear problem \eqref{eq008} along (the space generated by) 
	each eigenfunction $e_j$ we obtain an infinite dimensional system of ODE's
	$$\left\{\begin{array}{lcl}
	\mathcal L^j_{\varepsilon}(w_j) = f_j & {\rm in} & (-\log R,\infty)\\
	\Delta w_j(-\log R)=w_j(-\log R)=0,
	\end{array}
	\right.$$
	where the ordinary differential operator $\mathcal L^j_{\varepsilon}$ is given by
	\begin{equation}\label{eq015}
	\begin{array}{rcl}
	\mathcal L_{\varepsilon}^j(w_j)& = & \displaystyle 
	\ddddot w_j-\left(2\lambda_j+\frac{n(n-4)+8}{2}\right)\ddot w_j
	\\
	& &\displaystyle +\left(\lambda_j^2+\frac{n(n-4)}{2}
	\lambda_j+\frac{n^2(n-4)^2}{16}-\frac{n(n+4)(n^2-4)}{16}
	v_\varepsilon^{\frac{8}{n-4}}\right)w_j.
	\end{array}
	\end{equation}
	
	The problem will be divided in three cases: $j>n$, $j=0$ 
	and $j=1,\ldots,n$.
	
	\medskip
	
	\subsubsection{Case $j>n$}
	We consider the projections 
	$w'' = \pi''(w)$ and $f'' = \pi''(f)$ of $w$ and $f$ to 
	the high Fourier modes 
	\begin{equation}\label{eq018}
	\left\{\begin{array}{lcl}
	\mathcal L_{\varepsilon}(w'') = f'' & {\rm in} & D_R\\
	\Delta w''= w^{''}=0  & {\rm on} & \partial D_R.
	\end{array}
	\right.
	\end{equation}
	
	First we solve the problem
	\begin{equation}\label{eq009}
	\left\{\begin{array}{rcl}
	\mathcal L_{\varepsilon}(w_T)(t,\theta) = \overline f & {\rm in} & D_R^T\\
	\Delta w_T=w_T=0 & {\rm on} & \partial D_R^T,
	\end{array}
	\right.
	\end{equation}
	where $D_R^T:=(-\log R,T)\times\mathbb S^{n-1}$. This solution 
	can be obtained variationally. Consider the energy function 
	$\mathcal E_T$ given by
	$$\mathcal E_T(w)=\langle P_{g_{\rm cyl}}w,
	w\rangle-\int_{D_R^T}\left(\frac{n(n+4)(n^2-4)}{16}
	v_\varepsilon^{\frac{8}{n-4}}w^2+fw\right)d\theta dt,$$
	where
	$$\begin{array}{rcl}
	\langle P_{g_{\rm cyl}}w,w\rangle & = & 
	\displaystyle\int_{D_R^T}\left(\Delta_{g_{\rm cyl}}
	w\Delta_{g_{\rm cyl}}u-4A_{g_{\rm cyl}}(\nabla w,\nabla u)\right.\\
	& = & \displaystyle\int_{D_R^T}\left((\Delta_{g_{\rm cyl}}w)^2+\frac{n(n-4)}{2}|\nabla_{g_{\rm cyl}} w|^2+4\dot w^2+ \frac{n^2(n-4)^2}{16}w^2\right)d\theta dt.
	\end{array}$$

	It is easy to see that critical points of the functional $\mathcal E_T$ are weak solutions to  \eqref{eq009}. Observe that
	\begin{align*}
	\mathcal E_T(w) \geq & ~ \displaystyle 
	\int_{D_R^T} \left(\lambda^2w^2+\ddot w^2-2\ddot ww\lambda
	+ \frac{n(n-4)}{2} \lambda w^2+\frac{n(n-4)+8}{2}\dot w^2\right.
	\\
	 & ~ \displaystyle\left. +\frac{n^2(n-4)^2}{16}
	w^2-\frac{n(n+4)(n^2-4)}{16}
	v_\varepsilon^{\frac{8}{n-4}}w^2-fw\right)d\theta dt
	\\
	 \geq & ~ \displaystyle \int_{D_R^T} 
	\left(\ddot w^2+\left(2\lambda+\frac{n(n-4)+8}{2}
	\right)\dot w^2\right.
	\\
	 & ~ \displaystyle\left. +\left(\lambda^2+
	\frac{n(n-4)}{2}\lambda+ \frac{n^2(n-4)^2}{16}
	-\frac{n(n+4)(n^2-4)}{16}v_\varepsilon^{\frac{8}{n-4}}
	\right)w^2-fw\right)d\theta dt.    
	\end{align*}
	
	When $\lambda\geq 2n$ we have
	\begin{equation}\label{eq014}
	\lambda^2+\frac{n(n-4)}{2}\lambda+ \frac{n^2(n-4)^2}{16}
	-\frac{n(n+4)(n^2-4)}{16}>0.
	\end{equation}
	
	This implies that $\mathcal E_T$ is coercive and convex in the space
	\begin{align}
	    H^{2}_0(D_R^T)^\perp:=\left\{u\in H^2(D_R^T); 
	    \Delta u=u=0\mbox{ in }
	  \partial D_R^T \ {\rm and } \ \int_{\mathbb S^{n-1}}
	  u(\cdot,\theta)e_j(\theta)d\theta=0,
	  \;\;j=0,\ldots,n\right\}.  \nonumber  
	\end{align}
	The existence of a unique minimizer for $\mathcal E_T$ is immediate. The standard elliptic
	theory yields the expected regularity issues for $w_T$ in 
	terms of the regularity of $f$.
	
	\begin{lemma}\label{lem004}
		For $2-n<\mu<2$, let $\delta=\frac{n-4}{2}+\mu$. Then 
		there exist constants $\varepsilon_0>0$ and $C>0$ independent 
		of $T$ and $R$ such that for all $\varepsilon
		\in(0,\varepsilon_0)$, we have
		\begin{equation}\label{eq010}
		\sup_{(t,\theta)\in D_R^T}e^{\delta t}|w_T(t,\theta)|
		\leq C\sup_{(t,\theta)\in D_R^T}e^{\delta t}| 
		\overline f(t,\theta)|
		\end{equation}
		
	\end{lemma}
	\begin{proof}
		The proof is by contradiction. Suppose that the 
		estimate \eqref{eq010} does not hold. This implies 
		that we can find a sequence $(\varepsilon_i,T_i,R_i,
		w_{T_i},f_i)$ such that
		\begin{enumerate}
			\item $\varepsilon_i\rightarrow 0$ 
			as $i\rightarrow\infty$.
			\item $\left\{\begin{array}{lcl}
			\mathcal L_{\varepsilon}(w_{T_i}) = 
			f_i & {\rm in} & D_{R_i}^{T_i}\\
			\Delta w_{T_i}=w_{T_i}=0 & {\rm on} & \partial D_{R_i}^{T_i}
			\end{array}\right.,
			$ for every $i\in\mathbb N$.
			\item $\displaystyle\sup_{(t,\theta)\in 
			D_{R_i}^{T_i}}| e^{\delta t} f_i''(t,\theta)|=1$, 
			for every $i\in\mathbb N$.
			\item $\displaystyle \sup_{(t,\theta)\in 
			D_{R_i}^{T_i}}e^{\delta t}|w_{T_i}(t,\theta)|
			\rightarrow\infty$ as $i\rightarrow\infty$.
		\end{enumerate}
		
		Now choose $t_i\in[-\log R_i,T_i]$ such that 
		$$A_i:=\sup_{(t,\theta)\in D_{R_i}^{T_i}}
		e^{\delta t}|w_{T_i}(t,\theta)|=\sup_{\theta\in
		\mathbb S^{n-1}}e^{\delta t_i}|w_{T_i}(t_i,\theta)|.$$
		Define
$w_i(t,\theta):=A_i^{-1}e^{\delta t_i}w_{T_i}(t+t_i,\theta)$
		and $f_i(t,\theta):=A_i^{-1}e^{\delta t_i}
		f''(t+t_i,\theta),$ for all $i\in\mathbb N$ and 
		all $(t,\theta)\in D_i:=(-\log R_i-t_i,T_i-t_i)\times
		\mathbb S^{n-1}$. Then
		\begin{equation}\label{eq011}
		\sup_{(t,\theta)\in D_i}e^{\delta t}|w_i(t,\theta)|
		=\sup_{\theta\in\mathbb S^{n-1}}|w_i(0,\theta)|=1
		\end{equation}
		and 
		\begin{equation}\label{eq012}
		\sup_{(t,\theta)\in D_i}e^{\delta t}f_i(t,\theta)
		\rightarrow 0\;\;\mbox{ as }i\rightarrow\infty.
		\end{equation}
		Additionally
		\begin{equation*}
		\left\{\begin{array}{lcl}
		\mathcal L_{i}(w_i) =  f_i & {\rm in}  & D_i\\
		\Delta w_i=w_i=0 & {\rm on} & \partial D_i,
		\end{array}
		\right.
		\end{equation*}
		where
		$$\mathcal L_i(w) = \displaystyle \Delta_{g_{\rm cyl}}^2
		w-\displaystyle\frac{n(n-4)}{2}\Delta_{g_{\rm cyl}}w-
		4\partial^2_t w+\displaystyle\frac{n^2(n-4)^2}{16}
		w -\frac{n(n+4)(n^2-4)}{16}v_{\varepsilon_i}
		(t+t_i)^{\frac{8}{n-4}}w.$$
		
\label{pg001}	Up to a subsequence, we can assume that 
		$-\log R_i-t_i\rightarrow\tau_1\in[-\infty,0]$ 
		and $T_i-t_i\rightarrow\tau_2\in[0,+\infty]$, as 
		$i\rightarrow+\infty$. Suppose $\tau_1=0$. Since 
		$w_i(-\log R_i-t_i,\theta)=0$ for every $\theta\in
		\mathbb S^{n-1}$ and \eqref{eq011} holds, then $\{ 
		|\nabla w_i|\}$ 
		blows up a region of the type $[-\log R_i-t_i,
		-\log R_i-t_i+1]\times\mathbb S^{n-1}$ as $i\rightarrow\infty$. 
		On the other hand, using \eqref{eq011} and \eqref{eq012} we 
		see that $|w_i|$ and $|f_i|$ are bounded by a positive 
		constant times $e^{\delta(\log R_i+t_i)}$ in the region 
		$[-\log R_i-t_i,-\log R_i-t_i+1]\times\mathbb S^{n-1}$. 
		Using estimates as in 
		\cite{Gazzola_Grunau_Sweers_2010, Nicolaescu_2007}, we see 
		that $|\nabla w_i|$ is also bounded by the same quantities 
		in $[-\log R_i-t_i,-\log R_i-t_i+1]\times\mathbb S^{n-1}$, 
		which implies that $|\nabla w_i|$ is uniformly bounded as 
		$i\rightarrow+\infty$, contradicting the blow-up of 
		$|\nabla w_i|$. Therefore 
		$\tau_1\not=0$. In the same way we conclude that $\tau_2\not=0$.
		
		Set $v_i(t):=v_{\varepsilon_i}(t+t_i)$. By Proposition 
		\ref{Prop002} we have that over every compact set of 
		$\mathbb R$ the $C^4$-norm of the functions $v_i$ are 
		uniformly bounded. Using the Arzelà–Ascoli Theorem, we 
		deduce that there is a subsequence, still denoted by $v_i$, 
		and a function $v_\infty\in C_{\rm loc}^3(\mathbb R)$ such 
		that $v_i\rightarrow v_\infty$ in $C_{\rm loc}^3(\mathbb R)$.
		
		\medskip
		
		\noindent{\bf Claim 1: } Either $v_\infty\equiv 0$ or 
		$0<v_\infty<1$. In the second case we have  
		$\displaystyle\lim_{t\rightarrow\pm\infty}v_\infty(t)=0$.
		
		Since for each $i\in\mathbb N$ the function $v_i$ satisfies the ODE 
		$$\ddddot{v}_i=\frac{n^2-4n+8}{2}\ddot v_i 
		-\frac{n^2(n-4)^2}{16}v_i+\frac{n(n-4)(n^2-4)}{16}
		v_i ^{\frac{n+4}{n-4}}$$
		and the right hand side converges to $\frac{n^2-4n+8}{2}
		\ddot v_\infty-\frac{n^2(n-4)^2}{16}v_\infty+
		\frac{n(n-4)(n^2-4)}{16}v_\infty ^{\frac{n+4}{n-4}}$ 
		in $C^0_{\rm loc}(\mathbb R)$, we conclude that $v_i\rightarrow 
		v_\infty$ in $C^4_{\rm loc}(\mathbb R)$. Therefore, $v_\infty$ 
		satisfies the ODE \eqref{eq001}. Now, the claim follows by 
		Theorem \ref{Prop001}, since $v_i\rightarrow v_\infty$ 
		in $C^4_{\rm loc}(\mathbb R)$ as $\varepsilon_i:=\min 
		v_i\rightarrow 0$.
		
		\medskip
		
		Combining \eqref{eq011}, \eqref{eq012}, and Schauder 
		estimates (see \cite{Gazzola_Grunau_Sweers_2010}) the 
		Arzela-Ascoli theorem gives us a subsequence, which we 
		still denote by $w_i$, that converges to a function 
		$w_\infty\in C^{2,\alpha}_{\rm loc}((\tau_1,\tau_2)\times
		\mathbb S^{n-1})$ satisfying the equation
		\begin{equation}\label{eq013}
		\mathcal L_{\infty}(w_\infty) =  0\;\mbox{ in }
	(\tau_1,\tau_2)\times\mathbb S^{n-1}
		\end{equation}
		in the sense of distributions, where
		\begin{align*}
		    \displaystyle\mathcal L_\infty = & ~\displaystyle 
		\Delta_{g_{\rm cyl}}^2-\displaystyle\frac{n(n-4)}{2}
		\Delta_{g_{\rm cyl}}-4\partial_t^2
		+\displaystyle\frac{n^2(n-4)^2}{16}-
		\displaystyle\frac{n(n+4)(n^2-4)}{16}v_{\infty}^{\frac{8}{n-4}}\\
		 = & ~ \displaystyle \partial_t^4+\Delta_{\mathbb S^{n-1}}^2
		+2\Delta_{\mathbb S^{n-1}}\partial_t^2
		-\frac{n(n-4)}{2}\Delta_{\mathbb S^{n-1}}
		-\frac{n(n-4)+8}{2}\partial_t^2
		\\
		 & ~ \displaystyle +\frac{n^2(n-4)^2}{16}-\frac{n(n+4)(n^2-4)}{16}
		v_\infty^{\frac{8}{n-4}}.
		\end{align*}
		
		If $\tau_i$ is a real number, then the boundary condition 
		becomes $\Delta w_\infty(\tau_i)=w_\infty(\tau_i)=0$. It 
		is important to point out that by \eqref{eq011} we have 
		that	$\displaystyle\sup_{\theta\in 
		\mathbb S^{n-1}}|w_\infty(0,\theta)|=1.$ Thus, $w_\infty$ 
		is nontrivial. Decompose $w_\infty$ as
		$$w_\infty(t,\theta)=\sum_{j\geq n+1}w_\infty^j(t)e_j(\theta).$$

		Projecting \eqref{eq013} along the eigenfuctions $e_j$, for 
		$j\geq n+1$, we obtain
	\begin{align*}
		\ddddot w_\infty^j-&\left(2\lambda_j+\frac{n(n-4)+8}{2}
		\right)\ddot w_\infty^j\\
		&+\left(\lambda_j^2+\frac{n(n-4)}{2}\lambda_j
		+\frac{n^2(n-4)^2}{16}-\frac{n(n+4)(n^2-4)}{16}
		v_\infty^{\frac{8}{n-4}}\right)w_\infty^j=0.
		\end{align*}

		After multiplying this equation by $w_\infty^j$, we use 
		integration by parts and \eqref{eq014} to deduce that 
		$w_\infty^j\equiv 0$, which is a contradiction. To justify the
		integration by parts, note that if both $\tau_1$ and 
		$\tau_2$ are finite, we use that $w_\infty^j(\tau_1)=0$. 
		Otherwise, since $\displaystyle\lim_{r\rightarrow\pm\infty}
		v_\infty=0$ in case $0<v_\infty<1$, we deduce that 
		$$w_\infty^j(t)\sim e^{\pm\mu_j^{\pm}t},\;\;\;
		\mbox{ as }|t|\rightarrow+\infty,$$
		where 
		\begin{equation*}
	\mu_j^\pm:=\frac{\sqrt{n(n-4)+8+4\lambda_j\pm 4 \sqrt{(n-2)^2+4\lambda_j}}}{2},
		\end{equation*}
	and $\lambda_j\geq 2n$. We remark that $\mu_j^\pm\geq \frac{n}{2}$. 
	On the other hand, the condition $|w_\infty|\leq e^{-\delta t}$, 
	together with the fact that $|\delta|<\frac{n}{2}$, implies 
	that $|w_\infty^j(t)|\leq Ce^{-|t|\frac{n}{2}}$, as $|t|
	\rightarrow+\infty$. Therefore, in this case we can also 
	integrate by parts. This finishes the proof of the lemma
	\end{proof}
	
	Using a scaling argument as in \cite[Proposition 6.2]{Mazziri-Ndiaye} 
	we can see that if $f''\in C^{0,\alpha}_\delta(D_R)$, 
	then $w_T\in C^{4,\alpha}_\delta(D_R^T)$. If $\gamma\in(-n/2,n/2)$ 
	then there exist $\varepsilon_0>0$ and a positive constant $C$, 
	which does not depend on $\varepsilon$, $R$ and $T$, such that for all $\varepsilon\in(0,\varepsilon_0]$ we have
	$$\|w_T\|_{C^{2,\alpha}_\delta(D_R^T)}\leq C\|\overline f\|_{C^{0,\alpha}_\delta(D_R^T)}.$$
	
	Therefore, letting $T\rightarrow+\infty$, we obtain the existence 
	of a solution $\ddot w$ to \eqref{eq018} fulfilling the estimate
	$$\|w^{''}\|_{C^{4,\alpha}_\delta(D_R)}\leq C\| 
	f''\|_{C^{0,\alpha}_\delta(D_R)},$$
	with $C>0$ independent of $R$ and $\varepsilon$.
	
	\medskip
	
	\subsubsection{ Case $j=0$} We start by considering the projection 
	of the problem  \eqref{eq008}  along the eigenfunction $\varphi_0$. 
	We get
	$$\left\{\begin{array}{lcl}
	\mathcal L^0_{\varepsilon}(w_0) = f_0 & {\rm in} & (-\log R,\infty)\\
	\ddot w_0(-\log R)=0\\
	w_0(-\log R)=0,
	\end{array}
	\right.$$
	where 
	\begin{equation*}
	\begin{array}{rcl}
	\mathcal L_{\varepsilon}^0(w_0)& = & \displaystyle \ddddot w_0-\frac{n(n-4)+8}{2}\ddot w_0 +\left(\frac{n^2(n-4)^2}{16}-\frac{n(n+4)(n^2-4)}{16}v_\varepsilon^{\frac{8}{n-4}}\right)w_0.
	\end{array}
	\end{equation*}
	Since we no longer have \eqref{eq014}, it is necessary to use a 
	different approach. Choose an extension of $f_0$ to $\mathbb R$, 
	still denoted by  $f_0$. For $T>-\log R$ consider the unique 
	solution $w_T$ of the problem 
	$$\mathcal L^0_{\varepsilon}(w_T) = f_0 \mbox{ in }
	(-\infty,T)$$
	with $\dddot w_T(T)=
	\ddot w_T(T)=
	\dot w_T(T)=
	w_T(T)=0.$
	
	\begin{lemma}\label{lem005}
		For $\mu>0$, let $\delta=\frac{n-4}{2}+\mu$. Then there are 
		constants $\varepsilon_0>0$ and $C>0$ independent of $T$ 
		and $R>0$, such that, for $\varepsilon\in(0,\varepsilon_0]$, 
		we have
		\begin{equation}\label{eq016}
		\sup_{t\in(-\log R,T)}e^{\delta t}|w_T(t)|\leq 
		C\sup_{t\in(-\log R,T)}e^{\delta t}|f_0(t)|.
		\end{equation}
	\end{lemma}
	\begin{proof}
		The proof is similar to that of Lemma \ref{lem004}. 
		Suppose that \eqref{eq016} does not hold. Then there 
		is a sequence $(\varepsilon_i,T_i,R_i,w_{T_i},f_{0,i})$ such that
		\begin{enumerate}
			\item $\varepsilon_i\rightarrow 0$ as $i\rightarrow\infty$;
			\item $\mathcal L^0_{\varepsilon_i}(w_{T_i}) = 
			f_{0,i}$ in $(-\infty,T_i)$, with $\dddot w_{T_i}(T_i)=
			\ddot w_{T_i}(T_i)=
			\dot w_{T_i}(T_i)=
			w_{T_i}(T_i)=0,$ for all $i\in \mathbb N$.
			\item $\displaystyle\sup_{t\in(-\log R_i,T_i)}
			e^{\delta t}|  f_0(t,\theta)|=1$, for every $i\in\mathbb N$;
			\item $\displaystyle\sup_{t\in(-\log R_i,T_i)}
			e^{\delta t}|  w_{T_i}(t,\theta)|\rightarrow\infty$ 
			as $i\rightarrow\infty$.
		\end{enumerate}
		
		There exists $t_i\in[-\log R_i,T_i]$ such that
		$$A_i:=\sup_{t\in[-\log R_i,T_i]}e^{\delta t}|w_{T_i}(t)|=e^{\delta t_i}|w_{T_i}(t_i)|.$$
		Now set 		$w_i(t)=A_i^{-1}e^{\delta t_i}w_{T_i}(t+t_i)$
		and		$f_i(t)=A_i^{-1}e^{\delta t_i}f_{0,i}(t+t_i),$
		for $t\in [-\log R_i-t_i,T_i-t_i]$ and every $i\in\mathbb N$. Note that $|w_i(t)|$ and $|f_i(t)|$ are bounded by a positive constant times $e^{-\delta t}$. Besides		
		\begin{equation}\label{eq017}
	\mathcal L_{i}^0(w_i) =  f_i  \textrm{ in } (-\log R_i-t_i,T_i-t_i)	\end{equation}
	with $\dddot w_i(T_i-t_i)=	\ddot w_i(T_i-t_i)=	\dot w_i(T_i-t_i)=	
	w_i(T_i-t_i)=0$, as well as 	
	$$\sup_{t\in[-\log R_i-t_i,T_i-t_i]}e^{\delta t}
	|w_i(t)|=|w_i(0)|=1$$
		and
		$$\sup_{t\in[-\log R_i-t_i,T_i-t_i]}e^{\delta t}
		|f_i(t)|\rightarrow 0,\;\;\mbox{ as }i\rightarrow\infty,$$
		where
		$$\mathcal L_i^0(w_i)=\ddddot w_i-\frac{n(n-4)+8}{2}
		\ddot w_i +\left(\frac{n^2(n-4)^2}{16}-
		\frac{n(n+4)(n^2-4)}{16}v_i^{\frac{8}{n-4}}\right)w_i$$
		and $v_i(t)=v_{\varepsilon_i}(t+t_i)$. Up to a subsequence, 
		$-\log R_i-t_i\rightarrow\tau_1\in[-\infty,0]$ 
		and $T_i-t_i\rightarrow\tau_2\in[0,\infty]$.
		
		\medskip
		
		\noindent{\bf Claim:} $\tau_2>0$.
		
		If $\tau_2=0$, since $|w_i(0)|=1$ and $w_i(T_i-t_i)=0$, then 
		$\dot w_i$ must blow up in $[T_i-t_i-1,T_i-t_i]$, as $i
		\rightarrow\infty$. Using that $|w_i(t)|,|f_i(t)|
		\leq Ce^{-\delta(T_i-t_i)}$ in this interval and the 
		boundary condition in \eqref{eq017}, we integrate the 
		ODE to obtain that
		$$|\dot w_i(t)|\leq Ce^{-\delta (T_i-t_i)},$$
		in the same interval. Thus, since $T_i-t_i\rightarrow0$, 
		we obtain a contradiction. The claim follows.
		
		\medskip
		
		The claim implies that the interval $(\tau_1,\tau_2)$ is not empty.
		
		As in the proof of Lemma \ref{lem004} we can suppose that $v_i\rightarrow v_\infty$ in $C^4_{\rm loc}(\mathbb R)$, where 
		$v_\infty$ satisfies \eqref{eq001} and is such that 
		either $v_\infty\equiv 0$ or $0<v_\infty<1$ with 
		$\displaystyle\lim_{t\rightarrow\pm\infty}v_\infty(t)=0$. 
		Passing to a subsequence if necessary, there exists a 
		function $w_\infty$ such that $w_i\rightarrow w_\infty$ 
		in $C^3_{\rm loc}(\tau_1,\tau_2)$ and satisfies the equation
		$$\mathcal L_\infty^0(w_\infty)=\ddddot w_\infty
		-\frac{n(n-4)+8}{2}\ddot w_\infty+\left(\frac{n^2(n-4)^2}{16}
		-\frac{n(n+4)(n^2-4)}{16}v_\infty^{\frac{8}{n-4}}\right)
		w_\infty=0$$
		in $(\tau_1,\tau_2)$. By the normalization 
		$|w_i(0)| = 1$ we have $|w_\infty(0)|=1$, and so  
		$w_\infty$ is not trivial.
		
		If $\tau_2<\infty$, then $\dddot w_\infty(\tau_2)=
		\ddot w_\infty(\tau_2)=\dot w_\infty(\tau_2)=w_\infty
		(\tau_2)=0$ implies that $w_\infty\equiv 0$, which is 
		a contradiction. If $\tau_2=\infty$, we use the fact 
		that $v_\infty\rightarrow 0$ as $|t|\rightarrow\infty$ to 
		conclude that 
		$$w_\infty(t)\sim e^{\pm\mu t}\quad \mbox{ as } \quad t
		\rightarrow+\infty,$$
		where $\mu=\frac{n}{2}$ or $\frac{n-4}{2}$. 
		Using that $|w_\infty(t)|\leq e^{-\delta t}$ with $\delta>
		\frac{n-4}{2}$, we get that $w_\infty\equiv 0$, which again 
		is a contradiction.
		\end{proof}
		
		Since the estimate \eqref{eq016} is independent of the 
		parameters $T$, $R$ and $\varepsilon$, we let $T\rightarrow
		+\infty$ to obtain a function $w_0$ which satisfies
		$$\mathcal L_\varepsilon^0(w_0)=f_0$$
		in $\mathbb R$ with the estimate
		\begin{equation*}
		\sup_{t\in(-\log R,+\infty)}e^{\delta t}|w_0(t)|
		\leq C\sup_{t\in(-\log R,+\infty)}e^{\delta t}|f_0(t)|,
		\end{equation*}
		where $\delta>\frac{n-4}{2}$ and the positive constant $C$ 
		does not depend on $\varepsilon\in(0,\varepsilon_0)$ and $R$. 
		If $f_0$ belongs to $C^{0,\alpha}_\delta(-\log R,+\infty)$, 
		then $w_0$ belongs to $C^{4,\alpha}_\delta(-\log R,+\infty)$ 
		and there exists a positive constant $C$ which does not 
		depend on $\varepsilon$ and $R$ such that
		$$\|w_0\|_{C^{4,\alpha}_\delta(-\log R,+\infty)}
		\leq C\|f_0\|_{C^{0,\alpha}_\delta(-\log R,+\infty)}.$$
		
		There is no reason why the function $w_0$ satisfies the 
		boundary condition $\ddot w_0=w_0=0$ in $t=-\log R$. Although
		it would be nice to find a right inverse with these conditions, 
		we can not control the boundary condition in the right 
		function space.
		
		\medskip
		
	\subsubsection{Case $j=1,\ldots,n$}
	
	Projecting \eqref{eq008} along the eigenfunction $\varphi_j$, 
	with $j=1,\ldots,n$, we get 
	
	$$\left\{\begin{array}{lcl}
	\mathcal L^j_{\varepsilon}(w_j) = f_j & {\rm in} &  (-\log R,\infty)\\
	\ddot w_j(-\log R)=0\\
	w_j(-\log R)=0,
	\end{array}
	\right.$$
	where $\mathcal L_\varepsilon^j$ is given by \eqref{eq015}. The proof 
	is similar to the case $j=0$. First we consider an extension of $f_j$, still denoted by $f_j$, to $\mathbb R$ and find the unique solution $w_T$ of 
	$$\mathcal L^j_{\varepsilon}(w_T) = f_j  \textrm{ in } (-\infty,T),$$
	with $\dddot w_T(T)=\ddot w_T(T)=\dot w_T(T)=	w_T(T)=0.$ In a similar manner of the case $j=0$ we can prove the following lemma.
	
	\begin{lemma}
		For $\mu>1$, let $\delta=\frac{n-4}{2}+\mu$. Then there are constants $\varepsilon_0>0$ and $C>0$ independent of $T$ and $R>0$, such that, for all $\varepsilon\in(0,\varepsilon_0]$, we have
		\begin{equation*}
		\sup_{t\in(-\log R,T)}e^{\delta t}|w_T(t,\theta)|\leq C\sup_{t\in(-\log R,T)}e^{\delta t}|  f_0(t,\theta)|.
		\end{equation*}
	\end{lemma}
	
	The proof of this lemma is similar to that of the Lemma \ref{lem005}, so we will not include it here. The fact that $\lambda_j = n-1$ for 
	each $j=1,2,\dots, n$ forces us to choose a weight $\delta 
	> \frac{n-2}{2}$. As in the case of $j=0$, as $T\rightarrow+\infty$ 
	we get, for every $j=1,\ldots,n$ a solution $w_j$ to the equation 
	$$\mathcal L_\varepsilon^j(w_j)=f_j$$
	in the whole $\mathbb R$. Moreover, if $f_j\in 
	C^{0,\alpha}_\delta(-\log R,+\infty)$ with $\delta>\frac{n-2}{2}$, 
	then $w_j\in C^{4,\alpha}_\delta(-\log R,+\infty)$ and there exists 
	a positive constant  $C$, which does not depend on $R$ and 
	$\varepsilon$, such that 
	$$\|w_j\|_{C^{4,\alpha}_\delta(-\log R,+\infty)}
	\leq C\|f_j\|_{C^{0,\alpha}_\delta(-\log R,+\infty)}.$$
	
	Combining the previous three lemmas we obtain the 
	following proposition, 
	\begin{proposition}
		Let $R>0$, $\alpha\in(0,1)$ and $\mu\in(1,2)$. Then there 
		exists $\varepsilon_0>0$ such that, for all 
		$\varepsilon\in(0,\varepsilon_0)$, there is an operator
		$$G_{\varepsilon,R}:C^{0,\alpha}_{\mu-4}(B_1(0)
		\backslash\{0\})\rightarrow C^{4,\alpha}_{\mu}
		(B_1(0)\backslash\{0\})$$
		with the norm bounded independently of $\varepsilon$ and 
		$R$, such that for all $f\in C^{0,\alpha}_{\mu-4}
		(B_1(0)\backslash\{0\})$, the function $w:=G_{\varepsilon,R}
		(f)$ solves the equation 
	\begin{equation}\label{eq047}
	    \left\{\begin{array}{rcl}
		L_{\varepsilon,R}(w)=f & {\rm in} & B_1(0)\backslash\{0\}   \\
		\pi''(w|_{\mathbb S^{n-1}})=0 & {\rm on} & \partial B_1(0) \\
		\pi''(\Delta w|_{\mathbb S^{n-1}})=0 & {\rm on} & \partial B_1(0).
		\end{array}\right.
	\end{equation}
		Moreover, if $f\in \pi''(C^{0,\alpha}_{\mu-4}
		(B_1(0)\backslash\{0\}))$, then $w\in \pi''
		(C^{4,\alpha}_{\mu}(B_1(0)\backslash\{0\}))$ and we may 
		take $\mu\in (2-n,2)$.
	\end{proposition}

We will work with the solution $u_{\varepsilon,R,a}$ given 
in \eqref{approx} and with functions defined in $B_r(0)
\backslash\{0\}$ with $0<r<1$. Then it is convenient to study 
the operator $L_{\varepsilon,R,a}$ in weighted function spaces 
defined in $B_r(0)\backslash\{0\}$.

Our construction requires a right-inverse for 
	the linearized operator in a ball of radius $r>0$, which 
	will be very small. We obtain the solution operator 
	in $B_r(0) \backslash \{ 0 \}$ from the solution 
	operator in the punctured unit ball by rescaling. 

First, we observe that if $f\in C^{0,\alpha}_{\mu-4}(B_1(0)
\backslash\{0\})$ and $w\in C^{4,\alpha}_{\mu}(B_1(0)\backslash
\{0\})$ satisfy \eqref{eq047}, then the function $f_r(x)=
r^{-4}f(r^{-1}x)$ and $w_r(x)=w(r^{-1}x)$ satisfy
	$$\left\{\begin{array}{rcl}
		L_{\varepsilon,rR}(w_r)=f_r & {\rm in} & B_r(0)\backslash\{0\}   \\
		\pi''(w_r|_{\mathbb S_r^{n-1}})=0 & {\rm on} & \partial B_r(0) \\
		\pi''(\Delta w_r|_{\mathbb S_r^{n-1}})=0 & {\rm on} 
		& \partial B_r(0).
		\end{array}\right.$$
Also, it is not difficulty to see that 
$$\|w_r\|_{(4,\alpha),\mu,r}\leq c\|f_r\|_{(0,\alpha),\mu-4,r},$$
where $c>0$ is a constant that does not depend on $\varepsilon$, 
$R$, $a$ and $r$.

Finally using a perturbation argument we obtain the following two propositions.
\begin{proposition}\label{Prop003}
		Let $R>0$, $\alpha\in(0,1)$ and $\mu\in(1,2)$. Then there exists $\varepsilon_0>0$ such that, for all $\varepsilon\in(0,\varepsilon_0)$, $a\in\mathbb R^n$ and $0<r\leq 1$ with $|a|r\leq r_0$ for some $r_0\in(0,1)$, there is an operator
		$$G_{\varepsilon,R,a,r}:C^{0,\alpha}_{\mu-4}(B_r(0)\backslash\{0\})\rightarrow C^{4,\alpha}_{\mu}(B_r(0)\backslash\{0\})$$
		with the norm bounded independently of $\varepsilon$, $R$, $a$ and $r$, such that for all $f\in C^{0,\alpha}_{\mu-4}(B_r(0)\backslash\{0\})$, the function $w:=G_{\varepsilon,R,a,r}(f)$ solves the equation 
	\begin{equation*}
	    \left\{\begin{array}{rcl}
		L_{\varepsilon,R,a}(w)=f & {\rm in} & B_r(0)\backslash\{0\}   \\
		\pi''(w|_{\mathbb S_r^{n-1}})=0 & {\rm on} & \partial B_r(0) \\
		\pi''(\Delta w|_{\mathbb S_r^{n-1}})=0 & {\rm on} & \partial B_r(0).
		\end{array}\right.
	\end{equation*}
Moreover, if $f\in \pi''(C^{0,\alpha}_{\mu-4}(B_r(0)\backslash\{0\}))$, then $w\in \pi''(C^{4,\alpha}_{\mu}(B_r(0)\backslash\{0\}))$ and we may take $\mu\in (2-n,2)$.
\end{proposition}
\begin{proof}
Observe that
$$(L_{\varepsilon,R,a}-L_{\varepsilon,R})v=c(n)(u_{\varepsilon,R,a}^{\frac{8}{n-4}}-u_{\varepsilon,R}^{\frac{8}{n-4}})v.$$

By \eqref{approx} we have
$$u_{\varepsilon,R,a}^{\frac{8}{n-4}}(x)=|x-a|x|^2|^{-4}v_\varepsilon^{\frac{8}{n-4}}\left(-\log|x|+\log\left|\frac{x}{|x|}-a|x|\right|+\log R\right).$$

However, $|x-a|x|^2|^{-4}=|x|^{-4}+O(|a||x|^{-3}),$ $\log\left|x/|x|-a|x|\right|=O(|a||x|)$
and
\begin{align*}
    v_\varepsilon^{\frac{8}{n-4}}\left(-\log|x|+\log\left|\frac{x}{|x|}-a|x|\right|+\log R\right)=v_\varepsilon^{\frac{8}{n-4}}\left(-\log|x|+\log R\right)\\
    +\frac{8}{n-2}\int_0^{\log\left|\frac{x}{|x|}-a|x|\right|}\left(v_\varepsilon^{\frac{12-n}{n-4}}\dot v_\varepsilon\right)\left(-\log|x|+\log R+t\right)dt.
\end{align*}

Therefore, using \eqref{eq048} and fact that $0<v_\varepsilon<1$ we get
\begin{equation}\label{eq058}
    \begin{array}{rcl}
         |(u_{\varepsilon,R,a}^{\frac{8}{n-4}}-u_{\varepsilon,R}^{\frac{8}{n-4}})(x)| & \leq &\displaystyle c_n|x|^{-4}\int_0^{O(|a||x|)}v_\varepsilon^{\frac{8}{n-4}}(-\log|x|+\log R+t)dt+O(|a||x|^{-3})
         \\
         & \leq &  c_n|a||x|^{-3}.
    \end{array}
\end{equation}
This implies that
$$\|u_{\varepsilon,R,a}^{\frac{8}{n-4}}(x)-u_{\varepsilon,R}^{\frac{8}{n-4}}(x)\|_{(0,\alpha),[\sigma,2\sigma]}\leq c_n|a|\sigma^{-3},$$
which implies that
$$\|(L_{\varepsilon,R,a}-L_{\varepsilon,R})v\|_{(0,\alpha),\mu-4,r}\leq c_n|a|r\|v\|_{(4,\alpha),\mu,r}.$$
Therefore, by an perturbation argument we obtain the result.
\end{proof}

\subsection{Nonlinear analysis}\label{sec:nonlinear-analysis}

The main goal of this section is to construct constant $Q$-curvature 
metrics in the punctured ball $B_r(p)$ that are close to the 
deformed Delaunay metrics $g_{\varepsilon,R,a} = 
u_{\varepsilon, R, a}^{\frac{4}{n-4}} \delta$ with 
partially prescribed boundary Navier boundary data. More precisely, 
given $\phi_0, \phi_2 \in C^{4,\alpha} (\mathbb{S}^{n-1}_r)$ 
we want to find a complete metric $g_{\rm int}$ in the punctured ball 
$B_r(p) \backslash \{ p\}$ with 
$Q$-curvature equal to $\frac{n(n^2-4)}{8}$ where 
$$g_{\rm int} = (u_{\varepsilon, R, a} {+\Upsilon}+ v_{\phi_0, \phi_2} 
+ w_{\varepsilon, R} + v)^{\frac{4}{n-4}} g.$$ 
Here 
{
\begin{equation}\label{eq0008}
    \Upsilon:=-\frac{\beta_\varepsilon}{2}R^{-\frac{n}{2}}|x|^2,
\end{equation}
 where $\beta_\varepsilon$ is defined in \eqref{eq0001},} $v_{\phi_0, \phi_2} = \mathcal{P}_r (\phi_0, \phi_2)$, 
where $\mathcal{P}_r$ is the Poisson operator we constructed 
in Corollary \ref{poissoninterior}, and 
$$w_{\varepsilon,R} = \left \{ \begin{array}{rl} 
0, & 5 \leq n \leq 9 \\ -G_{\varepsilon, R, a} (\mathscr{P}
(u_{\varepsilon,R} )), & n \geq 10 , \end{array} \right. $$
where $G_{\varepsilon, R, a}$ is given in Proposition \ref{Prop003}
and $\mathscr{P}$ is given in Lemma \ref{lem008}. {Note that $\Upsilon$ is biharmonic}.

Following \eqref{eq020} we arrive at the boundary value problem 
\begin{equation}\label{initialpro}
	H_{g}(u_{\varepsilon, R,a}{+\Upsilon} +v_{\phi_0,\phi_2}  
	+ w_{\varepsilon,R} + v) = 0 , 
	\quad \pi''(\left. v \right |_{\partial B_r}) = 0, \quad 
	\pi''(\left. (\Delta v )\right |_{\partial B_r}) = 0 .
	\end{equation}
for some $0 < r \leq r_1$. 
We can expand \eqref{initialpro} as  
	\begin{equation}\label{p2}
	    \begin{array}{r}
	L_{\varepsilon,R,a}(v)  = \displaystyle (\Delta^2 - P_g)
	(u_{\varepsilon,R,a}{+\Upsilon} + v_{\phi_0,\phi_2} + w_{\varepsilon,R} + 
	v)- \mathcal R_{\varepsilon,R,a}({\Upsilon+}v_{\phi_0,\phi_2} 
	+ w_{\varepsilon,R} + v)
	\\
	  
	{-L_{\varepsilon,R,a}(w_{\varepsilon,R})}+ 
	\dfrac{n(n^2-4)(n+4)}{16}u_{\varepsilon,R,a}^{\frac{8}{n-4}}
	({\Upsilon+}v_{\phi_0,\phi_2} + w_{\varepsilon,R})  =: \mathscr{M}_{\varepsilon,R,a}(v),
	\end{array}
	\end{equation}
	where $L_{\varepsilon,R,a}$ is defined in \eqref{eq025}, $P_g$ is 
	the Paneitz-Branson operator \eqref{paneitz-branson}, 
	$\mathcal R_{\varepsilon,R,a}:=\mathcal R^{u_{\varepsilon,R,a}}$ 
	is defined in \eqref{eq024}. 
	
	We first show that for a suitable choice of parameters the 
	right hand side $\mathscr{M}_{\varepsilon,R,a} (v)$ lies 
	in the space $C^{0,\alpha}_{\mu-4} 
	(B_{r_{\varepsilon}}(0)\backslash\{0\})$, which is the domain 
	of the right inverse $G_{\varepsilon, R, a}$ 
 constructed in Proposition \ref{Prop003}. This will allow us 
 to consider the map $\mathcal{N}_{\varepsilon}$, defined as 
\begin{equation}\label{eq053}
\mathcal{N}_\varepsilon (R,a,\phi_0, \phi_2, \cdot) : 
\mathscr{B} \rightarrow C^{4,\alpha}_\mu (B_{r_\varepsilon}(0) 
\backslash \{ 0 \} ), \quad 
    \mathcal N_{\varepsilon}(R,a,\phi_0,\phi_2,v) 
    = G_{\varepsilon,R,a,r}\left(\mathscr{M}_{\varepsilon,R,a}(v)\right),
\end{equation}
where $\mathscr B$ is a small ball in $C^{0,\alpha}_{\mu-4}(
B_{r_{\varepsilon}}(0)\backslash \{0\})$. Therefore, to solve the 
equation \eqref{initialpro}, it is sufficient to find a fixed point 
of the map $\mathcal N_{\varepsilon}$. We use the two auxiliary 
functions $w_{\varepsilon,R}$ and $\Upsilon$ to account for 
different phenomena. Without $w_{\varepsilon,R}$ 
the right hand side of \eqref{p2} has terms of order 
$O(|x|^{d+1-\frac{n}{2}})$, which is not enough to assure that 
$\mathscr{M}_{\varepsilon,R,a}(v)\in C_{\mu-4}^{0,\alpha}
(B_{r_\varepsilon}(0)\backslash\{0\})$ for $\mu\in(1,2)$. 
 {We choose
$\Upsilon$ precisely to cancel the third term in the expansion 
of the Delaunay-type solution given in \eqref{eq059}.}

\begin{remark}\label{remark002}
    Throughout this work let $\delta_0>0$ and $m>\delta_2>\delta_1>0$ be sufficiently 
    small, but fixed, constants. We fix $s = \frac{2}{n-4}-\delta_0$ 
    and $r_\varepsilon = \alpha_\varepsilon^s$ and choose $a \in 
    \mathbb{R}^n$ such that $|a|r_\varepsilon\leq 1$. Finally we let
    $b\in \mathbb{R}$ satisfy $|b|\leq 1/2$ and choose 
  $R=\left(\frac{\alpha_\varepsilon}{2+2b}\right)^{\frac{2}{n-4}}$.
\end{remark}
\color{black}

From the parameter choices in Remark \ref{remark002} and \eqref{approx} 
it follows that there are positive 
constants $c_1$ and $c_2$ independent of 
$R$ and $a$ such that
\begin{equation*}
c_1\varepsilon|x|^{\frac{4-n}{2}}\leq u_{\varepsilon,R,a}(x)
\leq c_2 |x|^{\frac{4-n}{2}},
\end{equation*}
for each $x$ in $B_{r_\varepsilon}(0)\backslash\{0\}$. 
{Furthermore, since $s>8/(n(n-4))$ when $\delta_0$ is 
sufficiently small, there exists a constant
$c_3$ such that
\begin{equation}\label{eq0009}
    c_1\varepsilon |x|^{\frac{4-n}{2}}\leq u_{\varepsilon,R,a}+\Upsilon\leq c_3|x|^{\frac{4-n}{2}}.
\end{equation}}

\begin{lemma}\label{rest} Let $\mu \in (1,3/2)$. There exist 
$\varepsilon_0>0$ { and $m>\delta_1>0$}, such that for each $\varepsilon \in 
(0,\varepsilon_0)$, $a\in \mathbb R^n$ with $|a|r_\varepsilon
\leq 1$, and for all $v_i \in C^{4,\alpha}_{\mu}(B_{r_{\varepsilon}}(0)
\backslash \{0\})$, $i=0,1$, and $w \in C^{4,\alpha}_{5+d-\frac{n}{2}}
(B_{r_{\varepsilon}}(0)\backslash \{0\})$ with {$\|w\|_{(4,\alpha),5+d-\frac{n}{2},r_\varepsilon}\leq c$ 
 and
$
  \|v_i\|_{(4,\alpha),\mu,r_\varepsilon}\leq 
  cr_{\varepsilon}^{m-\mu-\delta_1}$} there exists a 
  constant $C>0$ indepenent of $\varepsilon$, $R$ and $a$ such that 
  $\mathcal R_{\varepsilon,R,a}$ satisfies 

$$ \|\mathcal R_{\varepsilon,R,a}({\Upsilon+}w + v_1) - 
\mathcal R_{\varepsilon,R,a}({\Upsilon+}w + v_0) 
\|_{(0,\alpha),\mu-4 ,r_{\varepsilon}} 
\leq   C\varepsilon^{\delta'}
\|v_{1}-v_{0}\|_{(4,\alpha),\mu,r_{\varepsilon}}$$
		and
	\begin{equation*}
		\|\mathcal R_{\varepsilon,R,a}
		({\Upsilon+}v_1+w) \|_{(0,\alpha),\mu-4,r_{\varepsilon}} 
		\leq{ C\varepsilon^{\delta'}
		r_{\varepsilon}^{m-\mu}},
	\end{equation*}
	\color{black}
		for some $\delta'>0$.
	\end{lemma}
\begin{proof}
First, by \eqref{eq024}, we can write 
\begin{align*}
\mathcal R_{\varepsilon,R,a}({\Upsilon+}w+v_{1}) - \mathcal R_{\varepsilon,R,a}({\Upsilon+}w+v_{0}) = -\frac{n(n+4)^{2}}{2}(v_{1}-v_{0})\int_{0}^{1}\int_{0}^{1}(u_{\varepsilon,R,a} + sz_{t})^{\frac{12-n}{n-4}}z_{t}dtds
\end{align*}
where $z_{t} = {\Upsilon+}w + tv_{1} + (1-t)v_{0}$.  Combining this last 
displayed equation with { \eqref{eq0009} } implies 
\begin{equation*}
    c\varepsilon|x|^{\frac{4-n}{2}} \leq u_{\varepsilon,R,a}(x){+\Upsilon(x)} + v_{i}(x) + w(x) \leq C|x|^{\frac{4-n}{2}}
\end{equation*}
for sufficiently small $\varepsilon >0$. Thus 
\begin{equation*}
    \max_{0\leq s,t \leq 1}\|(u_{\varepsilon,R,a} +sz_{t})^{\frac{12-n}{n-4}}\|_{(0,\alpha),[\sigma,2\sigma]} \leq C\varepsilon^{\lambda_n}\sigma^{\frac{n-12}{2}},
\end{equation*}
where $\lambda_n=0$ for $5\leq n\leq 12$ and $\lambda_{n} = \frac{12-n}{n-4}$ for $n\geq 13$, 
which in turn yields 

\begin{align*}
 & \sigma^{4-\mu}\|\mathcal R_{\varepsilon,R,a}({\Upsilon+}w+v_1) -  \mathcal R_{\varepsilon,R,a}({\Upsilon+}w+v_0)\|_{(0,\alpha),[\sigma,2\sigma]} \leq \\
    &\leq C\varepsilon^{\lambda_{n}}\sigma^{\frac{n-4}{2}}\|v_{1}-v_{0}\|_{(4,\alpha),\mu,r_{\varepsilon}}\left( \|w\|_{(0,\alpha),[\sigma,2\sigma]} + \|v_1\|_{(0,\alpha),[\sigma,2\sigma]}+ \|v_{2}\|_{(0,\alpha),[\sigma,2\sigma]}{+\|\Upsilon\|_{(0,\alpha),[\sigma,2\sigma]}} \right)\\
    &\leq C\varepsilon^{\lambda_{n}-\frac{4}{n-4}+\frac{n}{2}s}\|v_{1}-v_{0}\|_{(4,\alpha),\mu,r_{\varepsilon}}
\end{align*}
and 

\begin{align*}
    \sigma^{4-\mu}\|\mathcal R_{\varepsilon,R,a}(\Upsilon+v_1+w)\|_{(0,\alpha),[\sigma,2\sigma]}  \leq & ~ C\varepsilon^{\lambda_n}\sigma^{\frac{n-4}{2}-\mu}(\|v_1\|^{2}_{(0,\alpha),[\sigma,2\sigma]}+\|w\|^{2}_{(0,\alpha),[\sigma,2\sigma]}+\|\Upsilon\|^{2}_{(0,\alpha),[\sigma,2\sigma]})\\
     \leq &~ C\varepsilon^{\lambda_n-\frac{8}{n-4}+s\frac{n+4}{2}}r_\varepsilon^{-\mu},
\end{align*}
with $\lambda_n-\frac{4}{n-4}+\frac{n}{2}s>0$ and $\lambda_n-\frac{8}{n-4}+s\frac{n+4}{2}>0$. From these inequalities we obtain the result.
\end{proof}

\color{black}	
\begin{lemma}\label{lem010} Let $5 \leq n \leq 9$, let $\mu \in (1,3/2)$, and 
let $c$ and $\kappa$ be fixed, positive constants. 
There exists $\varepsilon_{0}>0$ such that for each $\varepsilon \in 
(0,\varepsilon_0)$, for all $v \in C^{4,\alpha}_{\mu}(B_{r_{\varepsilon}}(0)\backslash \{0\})$, 
$\phi_0 \in C^{4,\alpha}(\mathbb{S}_{r_{\varepsilon}}^{n-1})$ and $\phi_2 \in \pi''(C^{4,\alpha}
(\mathbb{S}_{r_{\varepsilon}}^{n-1}))$, with $\|v\|_{(4,\alpha),\mu,r_{\varepsilon}} \leq 
cr_{\varepsilon}^{{m}-\mu-\delta_1}$ and $\|(\phi_0,\phi_2)\|_{(4,\alpha),r_{\varepsilon}}
\leq \kappa r_{\varepsilon}^{{m}-\delta_1}$,
 the right-hand side of $\eqref{p2}$ belongs to $C^{0,\alpha}_{\mu-4}(B_{r_{\varepsilon}}(0)\backslash \{0\})$. 
\end{lemma}
\begin{proof}
By Corollary \ref{cor:poissoninterior}, the norm of $v_{\phi_0,\phi_2}+v$ is bounded by a 
positive constant times $r_\varepsilon^{{m}-\delta_1}$, so 
Lemma \ref{rest} implies the term $\mathcal R_{\varepsilon,R,a}({\Upsilon+}v_{\phi_0,\phi_2}+v)$ belongs 
to $C^{0,\alpha}_{\mu-4}(B_{r_{\varepsilon}}(0)\backslash \{0\})$. We also have
$u_{\varepsilon,R,a}^{\frac{8}{n-4}}v_{\phi_0,\phi_2}=O(|x|^{-2})=O(|x|^{\mu-4})$. 

To prove the remaining term of \eqref{p2} has the right decay, we write
$$(\Delta^2-P_g)(u_{\varepsilon,R,a}+v_{\phi_0,\phi_2}+v) = 
(\Delta^2-\Delta^2_g)u_{\varepsilon,R,a}+(\Delta^2-P_g)
(v_{\phi_0,\phi_2}+v)+(\Delta_g^2-P_g)u_{\varepsilon,R,a}$$
and estimate each of these terms separately. 
Observe that $v_{\phi_0,\phi_2}=O(|x|^2)$ and $v=O(|x|^\mu)$,
and so the second terms decays like $O(|x|^{\mu-4})$. 
By Lemma \ref{lem009} we also have the estimate 
$(\Delta^2-\Delta^2_g)u_{\varepsilon,R,a}=O(|x|^{\mu-4})$, and 
so it only remains to estimate $(\Delta_g- P_g) u_{\varepsilon,R,a}$. 

Observe that 
\eqref{eq036} gives us $R_g=O(|x|^2)$, ${\rm Ric}_g=O(|x|)$ and $Q_g=O(1)$, 
and so from \eqref{eq049} we obtain directly that $(\Delta_g^2-P_g)u_{\varepsilon,R,a}
=O(|x|^{1-\frac{n}{2}})$ for $5\leq n\leq 7$. For dimension $8$ and $9$, 
\eqref{eq050} yields
$$\langle {\rm Ric}_g,\nabla^2u_{\varepsilon,R} 
\rangle=R_g\frac{u'}{r}+O(|x|^{4-\frac{n}{2}}).$$
Thus, using that $R_g=O(|x|^2)$, $Q_g=O(1)$ and \eqref{eq049} we 
obtain $(\Delta_g^2-P_g)u_{\varepsilon,R}=O(|x|^{2-\frac{n}{2}})$. 
Now, using \eqref{eq032} and \eqref{eq049} we get $(\Delta_g^2-P_g)(
u_{\varepsilon,R,a}-u_{\varepsilon,R})=O(|x|^{2-\frac{n}{2}})$. Therefore, 
$(\Delta_g^2-P_g)u_{\varepsilon,R,a}=O(|x|^{2-\frac{n}{2}})$ for $n=8$ or 
$9$. In either case we have $(\Delta_g^2-P_g)u_{\varepsilon,R,a}=O(|x|^{\mu-4})$. 
This completes the proof{, since $\Upsilon=O(|x|^2)$.}
\end{proof}

Let us now consider the case $n\geq 10$. Since $(P_g-\Delta_g^2)u_{\varepsilon,R,a}=
O(|x|^{d+1-\frac{n}{2}})$, unfortunately $(P_g-\Delta_g^2)u_{\varepsilon,R,a}\not\in 
C^{0,\alpha}_{\mu-4}(B_{r_\varepsilon}(0)\backslash\{0\})$ for any $\mu>1$, 
and so we cannot use the right inverse of $L_{\varepsilon,R,a}$ directly. To 
overcome this difficulty we will consider the expansion \eqref{eq032} and 
the fact that, by Lemma \ref{lem008}, $\mathscr{P}(u_{\varepsilon,R})$ 
is orthogonal to $\{1,x_1,
\ldots,x_n\}$. If follows from this fact and 
Proposition \ref{Prop003} that there exists $w_{\varepsilon
,R}\in \pi''(C^{4,\alpha}_{5+d-\frac{n}{2}}(B_{r_\varepsilon}(0)\backslash\{0\}))$ such that
\begin{equation}\label{eq068}
    L_{\varepsilon,R,a}(w_{\varepsilon,R})=-\mathscr{P}(u_{\varepsilon,R}).
\end{equation}
This auxiliary function $w_{\varepsilon,R}$ will eliminate the terms 
that were preventing us from applying the right inverse to the 
right hand-side of \eqref{p2}. From this,  the expansion \eqref{eq032} and \eqref{p2} we obtain
\begin{equation}\label{eq051}
    \begin{array}{rcl}
         L_{\varepsilon,R,a}(v) & = & \displaystyle (\Delta^2-P_g)(u_{\varepsilon,R,a}-u_{\varepsilon,R}{+\Upsilon}+v_{\phi_0,\phi_2}+w_{\varepsilon,R}+v) 
	\\
	& & + \dfrac{n(n^2-4)(n+4)}{16}u_{\varepsilon,R,a}^{\frac{8}{n-4}}({\Upsilon+}v_{\phi_0,\phi_2})- \mathcal R_{\varepsilon,R,a}({\Upsilon+}v_{\phi_0,\phi_2} +w_{\varepsilon,R}+ v)-\mathscr{R},
    \end{array}
\end{equation}
where Lemma \ref{lem008} implies $\mathscr{R}=O(|x|^{2d+4-\frac{n}{2}})$.

\begin{lemma} Suppose $n\geq 10$, $\mu \in (1,3/2)$, $\kappa>0$ and 
$c>0$ be fixed constants. Under the same assumption of 
Lemma \ref{lem010} the right-hand side of $\eqref{eq051}$ belongs 
to $C^{0,\alpha}_{\mu-4}(B_{r_{\varepsilon}}(0)\backslash \{0\})$. 
\end{lemma}
\begin{proof}
The terms $(\Delta^2-P_g)(v_{\phi_0,\phi_2}+v)$, 
$u_{\varepsilon,R,a}^{\frac{8}{n-4}}v_{\phi_0,\phi_2}$ and 
$\mathcal R_{\varepsilon,R,a}({\Upsilon+}v_{\phi_0,\phi_2} +w_{\varepsilon,R}+ v)$ are 
similar to the proof in Lemma \ref{lem010}, while the 
term $\mathscr{R} \in C^{0, \alpha}_{\mu-4} (B_{r_\varepsilon}
(0) \backslash \{ 0 \} )$ because $2d+4-\frac{n}{2}>\mu-4$. 

Since $w_{\varepsilon,R}=O(|x|^{5+d-\frac{n}{2}})$, \eqref{eq035} 
and \eqref{eq052} imply  $(\Delta^2-\Delta_g^2)
w_{\varepsilon,R}=O(|x|^{2+2d-\frac{n}{2}})$, with $2+2d-n/2>\mu-4$. 
Observe that $\Delta_g-P_g^2$ is a second order operator and 
$3+d-n/2>\mu-4$. Thus $(\Delta^2-P_g)w_{\varepsilon,R}=O(|x|^{\mu-4})$. 
In an analogous way, by \eqref{eq032},  we obtain $(\Delta^2-P_g)(u_{\varepsilon,R,a}-u_{\varepsilon,R})=
O(|x|^{2+d-\frac{n}{2}})$ with $2+d-\frac{n}{2}>\mu-4$. This 
completes the proof{, since $\Upsilon=O(|x|^2)$.}
\end{proof}

\begin{remark}\label{remark003}
    The vanishing of the Weyl tensor up to order 
    $d=\left[\frac{n-8}{2}\right]$ is sharp in the following 
    sense: If the Weyl tensor vanishes up to order $d-1$, 
    then for $n\geq 10$, $g_{ij}=\delta_{ij}+O(|x|^{d+2})$ 
    and
    $$(\Delta^2-\Delta^2_g)(u_{\varepsilon,R,a}-
    u_{\varepsilon,R})=O(|x|^{1+d-\frac{n}{2}}),$$
    which implies that $(\Delta^2-\Delta^2_g)
    (u_{\varepsilon,R,a}-u_{\varepsilon,R})
    \not\in C^{0,\alpha}_{\mu-4}(B_{r_\varepsilon}(0)\backslash\{0\})$.
\end{remark}

\subsubsection{Fixed point argument}\label{sec:fixedpointargument}

In this section we prove the map 
$$\mathcal N_{\varepsilon}(R,a,\phi_0,\phi_2,v)=
G_{\varepsilon,R,a,r}(\mathscr M_{\varepsilon,R,a}(v)),$$ 
defined in \eqref{eq053} has a fixed point, where 
$\mathscr M_{\varepsilon,R,a}(v)$ 
is the right hand side of \eqref{p2} for $5\leq n\leq 9$, and 
the right hand side of \eqref{eq051} for $n\geq 10$.

\begin{lemma}\label{fp2} For all $\mu \in \mathbb{R}$ and $v \in C^{4,\alpha}_{\mu}(B_{r}(0)\backslash\{0\})$ there exists a constant $c>0$ that does not depend on $r$ and $\mu$ such that 
\begin{align*}
   \|(P_g - \Delta^{2})v\|_{(0,\alpha),\mu-4,r} \leq  cr^{d+3}\|v\|_{(4,\alpha),\mu,r}
\end{align*}
\end{lemma}

\begin{proof}
First we use \eqref{eq035} and \eqref{eq052} to obtain
$$\sigma^{4-\mu}\|(\Delta_g^2-\Delta^2)v\|_{(0,\alpha),[\sigma,2\sigma]}\leq C\sigma^{d+3-\mu}\|v\|_{(4,\alpha),[\sigma,2\sigma]}.$$
Also, by \eqref{paneitz-branson} we get
$$\sigma^{4-\mu}\|(P_g-\Delta_g^2)v\|_{(0,\alpha),[\sigma,2\sigma]}\leq C\sigma^{\max\{3,d+3\}-\mu}\|v\|_{(4,\alpha),[\sigma,2\sigma]}.$$
Combining these two inequalities implies the result.
\end{proof}

\begin{lemma}\label{lem012}
Given $a\in\mathbb R^n$ with $|a|r_\varepsilon^{1-\delta_2}\leq 1$ 
and $\|(\phi_0,\phi_2)\|_{(4,\alpha),r_{\varepsilon}}\leq 
\kappa r_{\varepsilon}^{m-\delta_1}$, if  
$\varepsilon>0$ is sufficiently small there exists a 
constant $\beta>0$ such that
    $$\|u_{\varepsilon,R,a}^{\frac{8}{n-4}}({\Upsilon+}v_{\phi_0,\phi_2})\|_{(0,\alpha),\mu-4,r_\varepsilon}
    \leq c\varepsilon^\beta r_\varepsilon^{m-\mu}$$
    where $m$ is defined in Remark \ref{remark002}.
\end{lemma}
\begin{proof}
First we note that by Remark  \ref{remark002} we 
obtain $\log R=\frac{2}{n-4}\log \alpha_\varepsilon+\frac{2}{4-n}\log(2+2b)$. By \eqref{eq028} and \eqref{eq0003} we have
\begin{equation}\label{eq0002}
    |x|^{\frac{4-n}{2}}\left|\frac{x}{|x|}-a|x|\right|^{\frac{4-n}{2}}= |x|^{\frac{4-n}{2}}(1+O(r_\varepsilon^{\delta_2}))\quad\mbox{ and }\quad\log\left|\frac{x}{|x|}-a|x|\right|=O(r_\varepsilon^{\delta_2})
\end{equation}
respectively. This implies that if $r_\varepsilon^{1+\ell}\leq |x|\leq r_\varepsilon$ and $\ell\geq 0$ is small enough, we have
$$\left(\frac{2}{n-4}-s\right)\log\alpha_\varepsilon-C\leq -\log|x|+\log\left|\frac{x}{|x|}-a|x|\right|+\log R\leq \left(\frac{2}{n-4}-s(1+\ell)\right)\log\alpha_\varepsilon+C.$$
Observe $2/(n-4)-s>0$ by Remark \ref{remark002}. Thus, for $0<\ell<\frac{2}{s(n-4)}-1$ and $\varepsilon>0$ small enough
\eqref{v-bounded-by-cosh} implies 
$$v_\varepsilon\left(-\log|x|+\log\left|\frac{x}{|x|}-a|x|\right|+\log R\right)\leq c_nr_\varepsilon^{\frac{n-4}{2}}.$$
By \eqref{approx} and \eqref{eq0002} we have
$u_{\varepsilon,R,a}(x)\leq C|x|^{\frac{4-n}{2}}r_\varepsilon^{\frac{n-4}{2}},$
and so 
$$u_{\varepsilon,R,a}(x)^{\frac{8}{n-4}}\leq c_n|x|^{-4}r_\varepsilon^4$$
for all $r_\varepsilon^{1+\ell}\leq |x|\leq r_\varepsilon$. Then, for $\frac{1}{4}r_\varepsilon^{1+\ell}\leq \sigma\leq \frac{1}{2}
r_\varepsilon$, we have
\begin{align*}
     \sigma^{4-\mu}\|u_{\varepsilon,R}^{\frac{8}{n-4}}
     (\Upsilon+v_{\phi_0,\phi_2})\|_{(0,\alpha),[\sigma,2\sigma]}  \leq &  ~
     c\sigma^{-\mu} r_\varepsilon^{4}(\|v_{\phi_0,\phi_2}
     \|_{(0,\alpha),[\sigma,2\sigma]}+\|\Upsilon
     \|_{(0,\alpha),[\sigma,2\sigma]})
     \\
     \leq & ~ c(r_\varepsilon^{m-\mu+4-\delta_1}+r_\varepsilon^{6-\mu}\alpha_\varepsilon^{-\frac{4}{n-4}})
      \leq  ~
     cr_\varepsilon^\beta r_\varepsilon^{m-\mu},
\end{align*}
with $6s-4/(n-4)>0$ by Remark \ref{remark002}. For $0\leq\sigma
\leq 4^{-1}r_\varepsilon^{1+\ell}$, we use $(2-\mu)\ell- \delta_1 > 0$ 
and $(2+\mu(2-\mu)) s - \frac{4}{n-4} > 0$ to see
\begin{align*}
    \sigma^{4-\mu}\|u_{\varepsilon,R}^{\frac{8}{n-4}}(\Upsilon+v_{\phi_0,\phi_2})\|_{(0,\alpha),[\sigma,2\sigma]} 
     \leq &~ c\sigma^{-\mu} (\|v_{\phi_0,\phi_2}
     \|_{(0,\alpha),[\sigma,2\sigma]}+\|\Upsilon
     \|_{(0,\alpha),[\sigma,2\sigma]})\\
     \leq &~ c r_\varepsilon^{(2-\mu)\ell+\delta_1} 
     r_\varepsilon^{m-\mu}+r_\varepsilon^{2-\mu+(2-\mu)\ell}\alpha_\varepsilon^{-\frac{4}{n-4}}\leq r_\varepsilon^\beta r_\varepsilon^{m-\mu},
\end{align*}
which completes our proof. 
\end{proof}

\color{black}

\begin{proposition} \label{IN1}
Let $\mu\in(1,3/2)$, $\tau>0$ and $\kappa>0$ be fixed constants. 
There exists $\varepsilon_0>0$ such that for each $\varepsilon
\in(0,\varepsilon_0)$, $a\in\mathbb R^n$ with 
$|a|r_\varepsilon^{1-\delta_2}\leq 1$, $\phi_0\in C^{4,\alpha}
(\mathbb S_{r_\varepsilon}^{n-1})$ and $\phi_2\in
\pi''(C^{4,\alpha}(\mathbb S_{r_\varepsilon}^{n-1}))$ with 
$\|(\phi_0,\phi_2)\|_{(4,\alpha),r_{\varepsilon}}\leq 
\kappa r_{\varepsilon}^{{m}-\delta_1}$, there 
exists a fixed point of the map 
$\mathcal{N}_{\varepsilon}(R,a,\phi_0,\phi_2,\cdot)$ in 
the ball of radius $\tau r_{\varepsilon}^{{m}-\mu}$.
\end{proposition}
\begin{proof}
We will prove that the map $\mathcal N_\varepsilon(R,a,\phi_0,
\phi_2,\cdot)$ is a contraction for $\varepsilon>0$ small enough  
by proving the following two inequalities
\begin{equation}\label{eq054}
    \|\mathcal N_{\varepsilon}(R,a,\phi_0,\phi_2,0)
    \|_{(4,\alpha),\mu,r_{\varepsilon}} < \frac{1}{2}
    \tau r_{\varepsilon}^{{m}-\mu}
\end{equation}
and 
\begin{equation}\label{eq062}
\|\mathcal N_{\varepsilon}(R,a,\phi_0,\phi_2,v_{1}) - 
\mathcal N_{\varepsilon}(R,a,\phi_0,\phi_2,v_{2})
\|_{(4,\alpha),\mu,r_{\varepsilon}} < 
\frac{1}{2}\|v_{1}-v_{2}\|_{(4,\alpha),\mu,r_{\varepsilon}}
\end{equation}
for $v_{i} \in C^{4,\alpha}_{\mu}(B_{r_{\varepsilon}}
(0)\backslash\{0\})$, $i=1,2$ with $\| v_i \|_{(4,\alpha), 
\mu} \leq\tau r_{\varepsilon}^{{m}-\mu}$.

\medskip

\noindent \textbf{Case} $\mathbf{5\leq n\leq 9:}$ Since the right 
inverse $G_{\varepsilon,R,a,r_{\varepsilon}}$ given by 
Proposition \ref{Prop003} is bounded in-dependently of $\varepsilon$, 
$R$, $a$ and $r$, the inequality \eqref{eq054} follows once we
estimate the $C^{0,\alpha}_{\mu-4}$-norm of the right 
hand side of \eqref{p2} when $v=0$. Using  \eqref{eq032}, Lemma \ref{lem011} and the estimates in 
the proof of Lemma \ref{lem009} we get
\begin{align*}
    \|(\Delta^{2}-\Delta^{2}_{g})u_{\varepsilon,R,a}
\|_{(0,\alpha),\mu-4,r_{\varepsilon}}    = & ~ \
\|(\Delta^{2}-\Delta^{2}_{g}) (u_{\varepsilon,R,a} - 
u_{\varepsilon,R})\|_{(0,\alpha),\mu-4,r_{\varepsilon}}
\\
 \leq & ~  c|a|r_\varepsilon^{6+d-\frac{n}{2}-\mu}
\leq cr_\varepsilon^{\delta_2}r_\varepsilon^{{m}-\mu}.
\end{align*}
Here we have used $|a|r_\varepsilon^{1-\delta_2}\leq 1$, with $\delta_2>0$. 
By the proof of Lemma \ref{lem010} we have 
$(\Delta_g^2-P_g)u_{\varepsilon,R,a}=
O(|x|^{2+d-\frac{n}{2}})$,  which implies that
\begin{equation}\label{eq055}
    \|(\Delta_g^2-P_g)u_{\varepsilon,R,a}\|_{(0,\alpha),
    \mu-4,r_\varepsilon}\leq 
    cr^{\delta_2}r_\varepsilon^{{m}-\mu}.
\end{equation}

By Corollary \ref{cor:poissoninterior} and Lemma \ref{fp2} we have
\begin{equation}\label{eq056}
    \|(\Delta^{2}-P_g)v_{\phi_0,\phi_2}\|_{(0,\alpha),\mu-4,r} \leq  cr^{d+3}\|v_{\phi_0,\phi_2}\|_{(4,\alpha),\mu,r}\leq cr_\varepsilon^{d+3-\mu}\|(\phi_0,\phi_2)\|_{(4,\alpha),r}
\end{equation}

and
\begin{equation}\label{eq0006}
    \|(\Delta^{2}-P_g)\Upsilon\|_{(0,\alpha),\mu-4,r} \leq  c\varepsilon^{-\frac{4}{n-4}+s(d+5-m)}r_\varepsilon^{m-\mu},
\end{equation}
since $\Upsilon=O(\varepsilon^{-\frac{4}{n-4}}|x|^2)$.
\color{black}
We also have
\begin{equation}\label{eq057}
\|Q_{g}v_{\phi_0,\phi_2}\|_{(0,\alpha),\mu-4,r} \leq cr_\varepsilon^{4-\mu}\|(\phi_0,\phi_2)\|_{(4,\alpha),r}.
\end{equation}

Using Lemma \ref{rest}, one has 
\begin{equation}\label{eq061}
\|\mathcal R_{\varepsilon,R,a}({\Upsilon+}v_{\phi_0,\phi_2})\|_{(0,\alpha),\mu-4,r_{\varepsilon}} \leq {c\varepsilon^{\delta'}r^{m-\mu}},
\end{equation}
{where $\delta'>0$.}
Therefore, by \eqref{eq055}, \eqref{eq056}, \eqref{eq0006}, \eqref{eq057}, 
\eqref{eq061} and Lemma \ref{lem012} we obtain \eqref{eq054} 
for $\varepsilon>0$ small enough.
Now, using \eqref{eq053} with $\mathscr M_{\varepsilon,R,a}(v)$ as 
the right hand side of \eqref{p2} we obtain
\begin{align*}
\|\mathcal N_{\varepsilon}(R,a,\phi_0,\phi_2,v_{1}) - 
\mathcal N_{\varepsilon}(R,a,\phi_0,\phi_2,v_{2})
\|_{(4,\alpha),\mu,r_{\varepsilon}} \leq C\left( \|(\Delta^{2} 
- P_{g})(v_{1}-v_{2})\|_{(0,\alpha),\mu-4,r_{\varepsilon}} 
\right.\\ \left.+ \|\mathcal R_{\varepsilon,R,a}({\Upsilon+}v_{\phi_0,\phi_2}
+v_1)- \mathcal R_{\varepsilon,R,a}({\Upsilon+}v_{\phi_0,\phi_2}+v_2)
\|_{(0,\alpha),\mu-4,r_{\varepsilon}}\right).      
\end{align*}

By Lemmas \ref{fp2} and \ref{rest} we obtain directly \eqref{eq062} 
for $\varepsilon>0$ small enough.

\medskip

\noindent{\bf Case} $\mathbf{n\geq 10}$: Now we use \eqref{eq053} 
with $\mathscr M_{\varepsilon,R,a}$ as the right hand side 
of \eqref{eq051}. As in the previous case, to 
prove \eqref{eq054} it is enough to estimate the 
$C^{0,\alpha}_{\mu-4}$-norm of the right hand side of \eqref{eq051} 
when $v=0$.
Similar to the proof of Lemma \ref{fp2}, we obtain
\begin{equation}\label{eq063}
    \|(\Delta^2-P_g)w_{\varepsilon,R}\|_{(0,\alpha),\mu-4,
    r_\varepsilon}\leq  c\sigma^{8+2d-\frac{n}{2}-\mu}
    \|w_{\varepsilon,R}\|_{(4,\alpha),5+d-\frac{n}{2},
    r_{\varepsilon}} .
\end{equation}

By Lemma \ref{rest} we know that 

\begin{equation}\label{eq064}
     \|\mathcal R_{\varepsilon,R,a}({\Upsilon+}v_{\phi_0,\phi_2} + 
     w_{\varepsilon,R})\|_{(0,\alpha),\mu-4,r_{\varepsilon}}  
     \leq  \varepsilon^{\delta'}r_\varepsilon^{m-\mu},
\end{equation}
where $\delta'>0$.
\color{black}
Therefore, using $\mathscr R=O(r_\varepsilon^{2d+4-\frac{n}{2}})$, 
Lemma \ref{lem012}, \eqref{eq056}, \eqref{eq063}, \eqref{eq064} and 
a similar calculation to that of \eqref{eq055} we conclude 
that \eqref{eq054} holds.
Moreover, 
\begin{align*}
\|\mathcal N_{\varepsilon}(R,a,\phi_0,\phi_2,v_{1}) - 
\mathcal N_{\varepsilon}(R,a,\phi_0,\phi_2,v_{2})\|_{(4,\alpha),
\mu,r_{\varepsilon}} \leq C\left( \|(\Delta^{2} - 
P_{g})(v_{1}-v_{2})\|_{(0,\alpha),\mu-4,r_{\varepsilon}} \right.\\ 
\left.+ \|\mathcal R_{\varepsilon,R,a}({\Upsilon+}v_{\phi_0,\phi_2}+ 
w_{\varepsilon,R} + v_1)- \mathcal R_{\varepsilon,R,a}
({\Upsilon+}v_{\phi_0,\phi_2}+ w_{\varepsilon,R} + v_2)
\|_{(0,\alpha),\mu-4,r_{\varepsilon}}\right).
\end{align*}

By Lemma \ref{rest} 

\begin{align*}
  &\|\mathcal R_{\varepsilon,R,a}
({\Upsilon+}v_{\phi_0,\phi_2}+ w_{\varepsilon,R}+ v_1) - 
\mathcal R_{\varepsilon,R,a}({\Upsilon+}v_{\phi_0,\phi_2}+ 
w_{\varepsilon,R} + v_2) \|_{(0,\alpha),\mu-4 ,r_{\varepsilon}}  
       \leq  c\varepsilon^{\delta'}
      \|v_{1}-v_{2}\|_{(4,\alpha),\mu,r_{\varepsilon}},
\end{align*}
\color{black}
with {$\delta'>0$}. From this and Lemma \ref{fp2}, we 
obtain \eqref{eq062} for $\varepsilon$ sufficiently small.
\end{proof}

\begin{theorem}\label{teo001}
Let $\mu\in(1,3/2)$, $\tau>0$, $\kappa>0$ and $\delta_2>\delta_1>0$  
be fixed constants. There exists $\varepsilon_0>0$, such that for 
each $\varepsilon\in (0,\varepsilon_{0}]$, $|b|\leq 1/2$, 
$a \in \mathbb{R}^{n}$, $\phi_2 \in C^{4,\alpha}
(\mathbb{S}_{r_{\varepsilon}}^{n-1})$ and $\phi_0 \in 
\pi''(C^{4,\alpha}(\mathbb{S}_{r_{\varepsilon}}^{n-1}))$,  with 
$|a|r_{\varepsilon}^{1-\delta_2}\leq 1$, 
$\|(\phi_0,\phi_2)\|_{(4,\alpha),r_{\varepsilon}}\leq \kappa 
r_{\varepsilon}^{{m}-\delta_1}$, there exists 
$U_{\varepsilon,R,a,\phi_0,\phi_2} \in C^{4,\alpha}_{\mu}
(B_{r_{\varepsilon}}(0)\backslash \{0\})$ satisfying 
\begin{align*}
\left\{
\begin{array}{lcl}
H_{g}(u_{\varepsilon,R,a}{+\Upsilon} + w_{\varepsilon,R} + v_{\phi_0,\phi_2}+ 
U_{\varepsilon,R,a,\phi_0,\phi_2}) = 0 & 
{\rm in}& B_{r_{\varepsilon}}(0)\backslash\{0\} \\
\pi''(\Delta U_{\varepsilon,R,a,\phi_0,\phi_2}) 
= \pi''(U_{\varepsilon,R,a,\phi_0,\phi_2}) =0   
& {\rm on} & \partial B_{r_{\varepsilon}}(0)\\
\end{array}
\right.
\end{align*}
where $w_{\varepsilon,R}\equiv 0$ when $5\leq n \leq 9$, 
$w_{\varepsilon,R} \in \pi''(C^{4,\alpha}_{5+d-\frac{n}{2}}
(B_{r_{\varepsilon}}(0)\backslash\{0\}))$ is the solution 
of $\eqref{eq068}$ when $n\geq 10$.
Moreover,
\begin{equation}\label{eq065}
    \|U_{\varepsilon,R,a,\phi_0,\phi_2}\|_{(4,\alpha),
    \mu,r_{\varepsilon}} \leq \tau r_{\varepsilon}^{{m}-\mu}
\end{equation}
and
\begin{equation}\label{eq066}
    \|U_{\varepsilon,R,a,\phi_{1},\phi_{2}}-U_{\varepsilon,R,a,
    \tilde\phi_0,\tilde\phi_2}\|_{(4,\alpha),\mu,r_{\varepsilon}} \leq cr^{\delta_3-\mu}\|(\phi_0,\phi_2)-(\tilde\phi_0,\tilde\phi_2)\|_{(4,\alpha),r_\varepsilon},
\end{equation}
for some constants $c>0$ and $\delta_3>0$ that do not depend on $\varepsilon$, $R$, $a$ and $(\phi_0,\phi_2)$, $(\tilde\phi_0,\tilde\phi_2)$.
\end{theorem}
\begin{proof}
The solution $U_{\varepsilon,R,a,\phi_0,\phi_2}$ is the fixed point of the map $\mathcal N_\varepsilon
(R,a,\phi_0,\phi_2,\cdot)$ given by Proposition \ref{IN1} with the estimate \eqref{eq065}. Using this fact and \eqref{eq062}, we obtain that
\begin{align*}
 & \|U_{\varepsilon,R,a,\phi_0,\phi_2}-U_{\varepsilon,R,a,\tilde\phi_0,\tilde\phi_2}\|_{(4,\alpha),\mu,r_{\varepsilon}}\\
 \leq & ~ 2\|\mathcal N_\varepsilon
(R,a,\phi_0,\phi_2,U_{\varepsilon,R,a,\phi_0,\phi_2})-\mathcal N_\varepsilon
(R,a,\tilde\phi_0,\tilde\phi_2,U_{\varepsilon,R,a,\tilde\phi_0,\tilde\phi_2})\|_{(4,\alpha),\mu,r_{\varepsilon}}.
\end{align*}

The estimate \eqref{eq066} now follows from the 
definition of $\mathcal N_\varepsilon$ together with the 
estimates obtained in Proposition \ref{IN1}.
\end{proof}

\section{Exterior problem}\label{sec:exterioranalysis}
In the previous section we exploited the conformal structure of the equation to perturb 
a Delaunay-type solution to a exact solution to the Q-curvature problem in a small 
neighborhood of the singularity. To find a solution on the complement of this neighborhood 
whose boundary data matches the solution constructed in the previous section, we need to 
choose the approximate solution that will be perturbed on the exterior domain carefully. 
Such approximate solution must be close to the Delaunay-type solution $u_{\varepsilon,R,a}$ 
near to the boundary. Our asymptotic expansion in Corollary \ref{cor001} shows that 
as $\varepsilon \rightarrow 0$ the function $u_{\varepsilon,R,a}$ is close to the 
Green's function of the flat bi-Laplacian with a pole on the singularity, at least in a sufficiently 
small annulus. Consequently, since our manifold has constant $Q$-curvature (the constant 
function $1$ is a solution to $H_{g_0}(1)=0$), the natural approximate solution on the 
exterior domain will be the constant 1 plus the Green function with a pole at $p$.

Let $r_{1}\in (0,1)$ and $\Psi: B_{r_{1}}(0)\rightarrow M$ be a normal coordinate system 
with respect to $g=\mathcal F^{\frac{4}{n-4}}g_0$	on M centered in $p$, where $\mathcal F$ 
is defined in Section \ref{sec001}. In these coordinates we have $g_{ij}=\delta_{ij}+O(|x|^2)$
and $\mathcal F=1+O(|x|^2)$, which implies that $(g_0)_{ij}=\delta_{ij}+O(|x|^2)$. Remember 
that $g_0$ has constant $Q$-curvature equal to $n(n^2-8)/8$. Denote by $G_{p}(x)$ the Green's 
function with pole at $p$ for $L_{g_0} = P_{g_{0}} - \frac{n(n^{2}-4)(n+4)}{16}$, the 
linearization of $H_{g_0}$ at the constant function 1. The 
Green's function $G_p$ exists by our hypothesis that $g_0$ is 
nondegenerate. By a similar argument to the one in Proposition 2.7 
of \cite{MR3518237} we can normalize $G_{p}$ 
such that $\displaystyle\lim_{x\rightarrow 0}|x|^{n-4}G_p(x)=1$. This implies that 
$|G_p(x)|\leq C|x|^{4-n}$, for all $x\in B_{r_1}(0)$.  Our main goal in this section is 
to solve the following problem
\begin{equation}\label{EX}
    H_{g_0}(1 + \lambda G_{p} + u) = 0 \quad {\rm on} \quad M\backslash B_{r}(p),
\end{equation}
with $\lambda \in \mathbb{R}$, $r \in (0,r_1)$ and prescribed boundary data.
Similarly to the strategy of the previous section, expanding this equation we obtain 
\begin{equation*}
    H_{g_0}(1+\lambda G_{p} + u) = L_{g_{0}}(u) 
    + \mathcal R(\lambda G_{p}+u) =0,
\end{equation*}
since $H_{g_0}(1)= 0$ and $L_{g_{0}}(G_{p})=0$. Here $\mathcal R$ is defined in \eqref{eq024} 
with $u_0=1$. In order to apply a fixed point theorem in this case, it is necessary to show 
that the linearized operator in $M \backslash B_{r}(p)$ has a right inverse. At this point, 
the nondegeneracy of the metric will be necessary.

\subsection{Linear analysis}

The goal of this subsection is to find a right inverse of the 
operator $L_{g_0}$.  
To this aim, given $0<r<s$ consider in $\mathbb R^n$ the 
annulus $\Omega_{r,s}:=
B_s(0)\backslash B_r(0)$. We start proving the following result.
	\begin{proposition}\label{Prop004}
		Let $\nu \in (4-n,5-n)$ fixed and $0 < 2r < s$. There 
		exists an operator
		$
		\mathcal G_{r,s}: C^{0,\alpha}_{\nu-4}(\Omega_{r,s}) 
		\rightarrow C^{4,\alpha}_{\nu}(\Omega_{r,s})
		$
		such that, for all $f \in C^{0,\alpha}_{\nu-4}(\Omega_{r,s})$, 
		the function $w = \mathcal G_{r,s}(f)$ 
		satisfies
		\begin{equation}\label{eq077}
		\left\{
		\begin{array}{lcl}
		\Delta^{2}w =f   & {\rm in} & \Omega_{r,s}  \\
		\Delta w = w = 0  & {\rm on} & \partial B_{s}.
		\end{array}
		\right.    
		\end{equation}
		Moreover, there exists a positive constant $C$ that 
		does not depend on $s$ and $r$ such that
		\begin{equation*}
		\|\mathcal G_{r,s}(f)\|_{C^{4,\alpha}_{\nu}(\Omega_{r,s})} 
		\leq C\|f\|_{C^{0,\alpha}_{\nu-4}(\Omega_{r,s})}.
		\end{equation*}
	\end{proposition}
\begin{proof}
The existence of $w$ satisfying \eqref{eq077} 
follows by standard elliptic theory. Now we need to prove the 
existence of a constant $c>0$ which does not depend on $r$ 
and $s$ such that
\begin{equation}\label{ex03}
\sup_{\Omega_{r,s}}r^{-\nu}| w| \leq c     \sup_{\Omega_{r,s}}r^{4-\nu}| f|,
\end{equation}
Suppose by contradiction that the estimate does not hold. Then 
there exist a sequence of numbers $0<r_{i}<s_{i}$, a sequence 
of functions $f_{i}$ satisfying
$
\sup_{\Omega_{r_i,s_i}}r^{4-\nu}|f_{i}|=1
$,
and a sequence of solutions $w_i$ to the problem
$$\left\{\begin{array}{rcl}
     \Delta^{2} w_{i} =  f_{i}   & {\rm in} & \Omega_{r_i,s_i}\\
     \Delta  w_{i} = w_{i} = 0  & {\rm on} & \partial \Omega_{r_i,s_i}
\end{array}\right.$$
such that 
$
A_{i}:=\sup _{\Omega_{r_i,s_i}} r^{-\nu}\left|{w}_{i}\right| 
\rightarrow+\infty. 
$
Consider $(t_i,\theta_i) \in (r_{i},s_i)\times \mathbb{S}^{n-1}$, a 
point where this supremum $A_{i}$ is attained. Up to a subsequence, 
we can assume that $r_{i}/t_i\rightarrow\tau_1\in[0,1]$ and 
$s_{i}/t_i\rightarrow \tau_2\in [1,+\infty]$. With a similar argument 
applied in page \pageref{pg001} we can prove that $\tau_1$ and 
$\tau_2$ is not equal to $1$.

Define the sequence of rescaled functions
\begin{equation*}
\hat{w}_{i}(r\theta):=\frac{t_{i}^{-\nu}}{A_{i}} 
{w}_{i}\left(t_{i}r\theta\right) \quad \text { and } 
\quad \hat{f}_{i}(r\theta):=\frac{t_{i}^{2-\nu}}{A_{i}}
{f}_{i}\left(t_{i}r\theta\right).
\end{equation*}
Note that $1\in(\tau_1,\tau_2)$, $\hat w_i(\theta_i)=1$ and 
$|\hat w_i(r\theta)|
\leq r^\nu$. By Arzel\'a-Ascoli's Theorem we can assume, up to a 
subsequence, the sequence of functions $\hat w_{i}$ converges uniformly
on compact subsets 
of $(\tau_1,\tau_2)\times \mathbb{S}^{n-1}$ to a biharmonic 
function $\hat w_{\infty}$ which is not identically zero. Moreover, 
$\hat w_{\infty}$ is bounded by $r^{\nu}$ on $(\tau_1,\tau_2)
\times \mathbb{S}^{n-1}$, and  $\Delta\hat w_\infty=\hat w_\infty=0$ 
on the boundary if $\tau_1>0$ or if $\tau_2<\infty$. Writing 
\begin{equation*}
\hat w_{\infty}(r,\theta) = \sum_{j= 0}^\infty\hat w_{j}(r)e_{j}(\theta),
\end{equation*}
where $\hat w_{j}$ is given explicitly  by 
\begin{equation*}
\hat w_{j}:= a_{j}^{-}r^{j}+a_{j}^{+}r^{2-n-j} 
+b_{j}^{-}r^{2+j}+b_{j}^{+}r^{4-n-j}.
\end{equation*}
We examine the various possibilities of the values for 
$\tau_1$ and $\tau_2$ case by case.

\noindent {\bf Case 1: $\boldsymbol{\tau_1>0}$ and 
$\boldsymbol{\tau_2<+\infty}$.} In this case, $\hat w_{\infty}$ is 
biharmonic with zero Navier boundary data, which implies that 
$\hat w_{\infty}\equiv 0$

\noindent {\bf Case 2: $\boldsymbol{\tau_1=0}$ and 
$\boldsymbol{\tau_2=+\infty}$.}  Since $\hat w_{\infty}$ is bounded 
by $r^{\nu}$, using that $\tau_2=+\infty$ and $\nu<5-n\leq 0$ we 
find that $a_j^-=b_j^-=0$, and using that $\tau_1=0$ and $\nu>4-n$ 
we find that $a_j^+=b_j^+=0$.

\noindent {\bf Case 3: $\boldsymbol{\tau_1=0}$ and 
$\boldsymbol{\tau_2<+\infty}$.} As in the case 2 we find 
that $a_j^+=b_j^+=0$. Now, using that $\hat w_{\infty}=
\Delta \hat w_{\infty} = 0$ on $\partial B_{\tau_2}$ we find 
that $a_j^+=b_j^+=0$.

\noindent {\bf Case 4: $\boldsymbol{\tau_1>0}$ and 
$\boldsymbol{\tau_2=+\infty}$.} Similarly to the previous 
case we conclude that $\hat w_\infty\equiv 0$.

In each of these cases above, we obtain the zero function in 
the limit, contradicting the normalization $\hat w_i(\theta_i) 
= 1$.  Consequently, we proved the existence of a function satisfying \eqref{eq077} and the estimate \eqref{ex03}. The estimates for the derivatives follow from standard elliptic estimates.
\end{proof}

\begin{corollary}
Let $\nu\in(4-n,5-n)$. There exists $s_0>0$ such that for all 
$0<2r<s<s_0$ there is an operator 
$$\mathscr H_{r,s}:C^{0,\alpha}_{\nu-4}(\Omega_{r,s})
\rightarrow C^{0,\alpha}_{\nu}(\Omega_{r,s})$$ 
with the function $w=\mathscr H_{r,s}(f)$ satisfying 
\begin{align*}
		\left\{
		\begin{array}{lcl}
		L_{g_0}(w) =f   & {\rm in} & \Omega_{r,s}  \\
		\Delta w = w = 0  & {\rm on} & \partial B_{s}.
		\end{array}
		\right.
		\end{align*}
		Moreover, there exists a positive constant $C$ that 
		does not depend on $s$ and $r$ such that
		\begin{equation*}
		\|\mathscr H_{r,s}(f)\|_{C^{4,\alpha}_{\nu}(\Omega_{r,s})} \leq C\|f\|_{C^{0,\alpha}_{\nu-4}(\Omega_{r,s})}.
		\end{equation*}
\end{corollary}
\begin{proof}
Using \eqref{eq052} and the fact that in normal coordinates one has $g_0=\delta+O(|x|^2)$, we find $\|(\Delta_{g_0}^2-\Delta^2)v\|_{C^{0,\alpha}_{\nu-4}(\Omega_{r,s})}\leq s^2\|v\|_{C^{2,\alpha}_{\nu}(\Omega_{r,s})}$. From this and \eqref{paneitz-branson} we obtain
$$\|(L_{g_0}-\Delta^2)v\|_{C^{0,\alpha}_{\nu-4}(\Omega_{r,s})}\leq s^2\|v\|_{C^{2,\alpha}_{\nu}(\Omega_{r,s})}.$$
Finally, the result follows by a perturbation argument and Proposition \ref{Prop004}.
\end{proof}

For the remainder of the construction we let $M_r = M \backslash 
B_r(p)$. 

\begin{proposition}\label{RIEX}
Let $\nu \in (4-n,5-n)$. There exists $r_{2} < r_{1}$ such that for 
all $r \in (0,r_2)$ we can define an operator 
$
G_{r,g_0}: C^{0,\alpha}_{\nu-4}(M_r)\rightarrow C^{4,\alpha}_{\nu}(M_r)
$
where the function $w = G_{r,g_0}(f)$ satisfies
$$
L_{g_0}(w) = f, \qquad \left. w \right |_{\partial B_r}
= 0, \qquad \left. (\Delta w ) \right |_{\partial B_r} = 0
$$
in $M_{r}$. In addition, there exists a constant $C>0$ independent 
of $r$ such that
\begin{equation*}
\|G_{r,g_0}(f)\|_{C^{4,\alpha}_{\nu}(M_r)} 
\leq C\|f\|_{C^{0,\alpha}_{\nu-4}(M_r)}.
\end{equation*}
\end{proposition}

\begin{proof}
The proof follows the ideas of the proof of Proposition 13.28 
in \cite{jleli} (see also \cite{MR2639545}). Given $f\in 
C^{0,\alpha}_{\nu-4}(M_r)$ define a function $w_0\in 
C^{4,\alpha}_{\nu}(M_r)$ by $w_0:=\eta \mathscr 
H_{r,r_1}(f|_{\Omega_{r,r_1}})$, where $\eta$ is a smooth 
radial function equal to 1 in $B_{\frac{1}{2}r_1}(p)$ and 
equal to zero in $M_{r_1}$. There is a constant $C>0$ 
independent of $r$ and $r_1$ so that
\begin{equation}\label{eq070}
    \|w_0\|_{C^{4,\alpha}_{\nu}(M_r)}\leq 
    C\|f\|_{C^{0,\alpha}_{\nu-4}(M_r)}.
\end{equation}
Note that the function $h:=f-L_{g_0}(w_0)$ is supported in $M_{\frac{1}{2}r_1}$, and so it follows that 
$$\|h\|_{C^{0,\alpha}(M)}\leq \|f \|_{C^{0, \alpha} (M)} 
+ \| L_{g_0} (w_0) \|_{C^{0,\alpha}_{\nu-4}(M_r)} 
\leq C_1\|f\|_{C_{\nu-4}^{0,\alpha}(M_r)},$$
for some constant $C$ independent of $r$. By the non degeneracy of 
the metric, we can define $w_1:=\chi L_{g_0}^{-1}(h)$, where 
$\chi$ is a smooth radial function equal to 1 in $M_{2r_2}$ and 
equal to zero in $B_{r_2}(p)$ with $4r_2<r_1$. Hence, there 
exists a constant $C>0$  independent of $r$ and $r_2$ such that
\begin{equation}\label{eq071}
    \|w_1\|_{C_\nu^{4,\alpha}(M_r)}\leq C\|f\|_{C_{\nu-4}^{0,\alpha}(M_r)}.
\end{equation}
This implies that there exists $C>0$ independent of $r$ and $r_2$ so that
$$\|L_{g_0}(w_0+w_1)-f\|_{C_{\nu-4}^{0,\alpha}(M_r)}\leq 
Cr_2^{-1-\nu}\|f\|_{C_{\nu-4}^{0,\alpha}(M_r)}.$$
From this, by \eqref{eq070}, \eqref{eq071} and a perturbation 
argument we obtain the result.
\end{proof}	

\subsection{Nonlinear analysis}\label{sec:constant-q-curvature}
Now that we have shown the linearized operator about the 
constant function $1$ has a right-inverse on $M_r$ for 
$r$ sufficiently small, we use a 
fixed point argument to show the existence of a function 
satisfying $\eqref{EX}$. However, because we must prescribe
the Navier boundary data appropriately, in addition to the 
constant 1 and the Green function, we add the exterior Poisson 
operator (constructed in section \ref{EPO}) in the expansion 
of equation \eqref{EX}.

Recall that $\Psi$ parameterizes $B_{r_1}(p)$ in 
conformal normal coordinates. We let 
$$\eta: \R^n \rightarrow \R, \quad \eta(x) = \left \{ 
\begin{array}{rl} 1, & |x| < r_1/2 \\ 0, & |x|> r_1 
\end {array} \right .$$ 
be a smooth, radial cut-off function satisfying 
$|\partial^{(k)}_r \eta | \leq c |x|^{-k}$ on $B_{r_1}(0)$ 
for each $k=0,1,\dots, 4$. Consequently, we have 
$\|\eta\|_{(4, \alpha),[\sigma, 2 \sigma]} \leq c$ 
for every $r \leq \sigma \leq \frac{1}{2} r_{1}$.
For each $\psi_0,\psi_2 \in C^{4, \alpha}
\left(\mathbb{S}_{r}^{n-1}\right)$ we define $u_{\psi_0, \psi_2}$ 
by 
$$u_{\psi_0, \psi_2} \equiv 0 \textrm{ on } M_{r_1}, \quad
u_{\psi_0, \psi_2} \circ \Psi = \eta \mathcal{Q}_r(\psi_0, \psi_2 ) \textrm{ on } B_{r_1} \backslash B_r,$$
where $\mathcal Q_{r}$ is
the Poisson operator on the exterior domain constructed in Corollary~\ref{cor:poissonexterior}. 
By Corollary \ref{cor:poissonexterior} we obtain
\begin{equation}\label{eq072}
    \left\|u_{\psi_0,\psi_2}\right\|_{C_{\nu}^{4, \alpha}
    \left(M_{r}\right)}\leq c r^{-\nu}\|(\psi_0,\psi_2)
    \|_{(4, \alpha), r}
\end{equation}
for $\nu \geq 4-n$. 

Finally, we seek a solution to \eqref{EX} of the form 
$u_{\psi_0, \psi_2} + v$, which is equivalent to 
finding a fixed point of the mapping 
$$\mathcal{M}_r(\lambda, \psi_0, \psi_2,\cdot) : 
C^{4,\alpha}_\nu(M_r) \rightarrow C^{4,\alpha}_\nu(M_r)$$
defined by 
$$\mathcal M_{r}(\lambda,\psi_0, \psi_2,v) := - G_{r,g_{0}}
(\mathcal R(\lambda G_{p} +u_{\psi_0, \psi_2}+v) + 
L_{g_{0}}(u_{\psi_0,\psi_2})).$$

\begin{proposition}\label{prop005}
Let $\nu \in (9/2-n,5-n)$, $\delta_4 \in (0,1/2)$ and 
$\beta,\gamma>0$ fixed constants. There exists $r_{2}>0$, such 
that for all $r\in (0,r_{2})$, $\lambda \in \mathbb R$, 
$\psi_0,\psi_2 \in C^{4,\alpha}(\mathbb{S}^{n-1}_{r})$ 
satisfying {$|\lambda|\leq r^{n-4+\frac{m}{2}}$,} and 
$\|(\psi_0,\psi_2)\|_{(4,\alpha),r} \leq \beta  
r^{{m}-\delta_4}$, there exists a fixed point to the map 
$\mathcal M_{r}(\lambda,\psi_0,\psi_2,\cdot)$ in the 
ball of radius $\gamma r^{{m}-\nu}$, {where $m$ is defined in Remark \ref{remark002}}.
\end{proposition}	

\begin{proof} We follow the same strategy as in the proof of 
Proposition \ref{IN1}, that is, showing that $\mathcal M_{r}
(\lambda,\psi_0,\psi_2,\cdot)$ is a contraction in the ball 
of radius $\gamma r^{{m}-\nu}$. First, since $u_{\psi_0,\psi_2}\equiv 0$ in $M_{r_{1}}$, we have
\begin{equation}\label{eq073}
\|\mathcal R(\lambda G_p + u_{\psi_0,\psi_2})
\|_{C^{0,\alpha}_{\nu-4}(M_r)} = \|\mathcal R
(\lambda G_p)\|_{C^{0,\alpha}_{\nu-4}(M_{\frac{1}{2}r_1})}+ 
\|\mathcal R(\lambda G_{p}+u_{\psi_0,\psi_2})
\|_{C^{0,\alpha}_{\nu-4}(\Omega_{r,r_{1}})}
\end{equation}

Note that $1+\lambda G_p>0$ in $M_{r_1}$ for $r$ small enough. 
Additionally, the $C^{0,\alpha}(M_{\frac{1}{2}r_1})$-norm of 
$\left(1+s t \lambda G_{p}\right)^{\frac{12-n}{n-4}}$ and $G_p$ 
are bounded independently of $r$. Thus, it follows 
from\eqref{eq024} that
\begin{equation}\label{er1}
\|\mathcal R(\lambda G_{p})\|_{C^{0, \alpha}(M_{\frac{1}{2}r_{1}})} 
\leq C|\lambda|^{2} \leq C r^{\delta^{\prime}} r^{{m}-\nu},
\end{equation}
where the constant $C>0$ does not depend on $r$. The fact 
that $\nu > 9/2 -n$ implies $\delta'>0$.  Since 
$r<\sigma$ and $19/2-2n-\nu<0$, similarly to \eqref{eq072} 
we can prove that 
$$\sigma^{4-\nu}\|u_{\psi_0,\psi_2}\|_{(0, \alpha),
[\sigma, 2 \sigma]}^{2}\leq C  r^{{\frac{3}{2}-\delta_4+m-\nu}}.$$
By \eqref{eq072} we have $|u_{\psi_0,\psi_2}(x)|
\leq cr^{{m}-\nu-\delta_4}$, for all $x\in M_r$ 
with ${m}-\nu-\delta_4>0$, and we have 
$|\lambda G_p(x)|\leq cr^{{m/2}}$, 
where the exponent is positive. This implies that 
$$0<c<1+t(\lambda G_p+u_{\psi_0,\psi_2})<C$$ 
for every $t\in[0,1]$, 
and so we obtain that the H\"older norm of $(1+
t(\lambda G_p+u_{\psi_0,\psi_2}))^{\frac{12-n}{n-4}}$ 
is bounded independently of $r$. {Thus, since  $19/2-2n-\nu<0$, 
by \eqref{eq024} we obtain
\begin{equation}\label{er2}
\|\mathcal R\left(\lambda G_{p}+u_{\psi_0,\psi_2}\right)
\|_{C_{\nu-4}^{0, \alpha}\left(\Omega_{r, r_{1}}\right)} 
\leq C r^{\frac{3}{2}-\delta_4} r^{{m}-\nu}.
\end{equation}}
From \eqref{er1}, \eqref{er2}, it follows that the 
right hand side of \eqref{eq073} 
is bounded by $cr^{\delta''}r^{{m}-\nu}$,
for some constants $\delta''>0$ and $c>0$  independent of $r$.
Using that 
$u_{\psi_0,\psi_2}=\mathcal Q(\psi_0,\psi_2),$ 
the fact that 
the derivative of $\eta$ is zero in $B_{\frac{1}{2}r_1}
\backslash B_r$, and \eqref{eq052} we obtain
$$\|\Delta^2_{g_0}u_{\psi_0,\psi_2}\|_{(0,\alpha),[\sigma,2\sigma]}
\leq C_{r_1}\|\mathcal Q(\psi_0,\psi_2)\|_{(4,\alpha),[\sigma,2\sigma]},$$
for some constant $C_{r_1}$ independently of $r$. 
By \eqref{paneitz-branson} and \eqref{eq021} the operator
$L_{g_0}-\Delta_{g_0}^2$ is second order, and so 
for some $C>0$ independent of $r$ we have
\begin{equation*}
\|L_{g_0}(u_{\psi_0,\psi_2})\|_{(0,\alpha),[\sigma,2\sigma]} \leq C\sigma^{-2}\|Q_{r}(\psi_0,\psi_2)\|_{(4,\alpha),[\sigma,2\sigma]}.
\end{equation*}
Note that $6-n-\nu>0$ and $n-4+\nu-\delta_4>0$, and so
\begin{equation}\label{epoisson}
\|L_{g_0}(u_{\psi_0,\psi_2})\|_{C^{0,\alpha}_{\nu-4}(\Omega_{r,r_{1}})} \leq Cr^{n-4+\nu-\delta_{4}}r^{{m}-\nu}.
\end{equation}
Using  \eqref{eq073}, \eqref{er1}, \eqref{er2} and \eqref{epoisson},
and the fact that the right inverse of $G_{r, g_0}$ 
constructed in Proposition \ref{RIEX} is bounded independently 
of $r$ together with 
 $u_{\psi_0,\psi_2}\equiv 0$ in $M_{r_1}$ for $r$ small enough, 
 we have 
\begin{equation}\label{fixpoint1}
\|\mathcal M_{r}(\lambda,\psi_0,\psi_2,0)
\|_{C^{4,\alpha}_{\nu}(M_{r})} \leq 
\frac{1}{2}\gamma r^{{m}-\nu} .
\end{equation}

Now, note that for $r>0$ small enough we have $0<c<z_{t} = 
1+\lambda G_p + u_{\psi_0,\psi_2} + v_2 + t(v_2-v_1)<C$, 
for all $t\in[0,1]$. Thus, the $C^{0,\alpha}$-norms of 
$\left(1+s z_{t}\right)^{\frac{12-n}{n-4}}$ and $\left(1+s 
z_{t}\right)^{\frac{12-n}{n-4}}$  are bounded independently of $r$. 
Using an analogous equality as in the beginning of the proof of 
Lemma \ref{rest}, we obtain
$$\left\|\mathcal R\left(\lambda G_{p}+v_{1}\right)-
\mathcal R\left(\lambda G_{p}+v_{2}\right)\right
\|_{C^{0, \alpha}\left(M_{r_{1}}\right)} \leq 
C\left(r^{{n-4+\frac{m}{2}}}
+r^{{m}-\nu}\right)\left\|v_{1}-v_{2}
\right\|_{C_{\nu}^{4, \alpha}\left(M_{r}\right)}$$
and
$$\left\|\mathcal R^{1}\left(\lambda G_{p}+
u_{\phi_0,\phi_2}+v_{1}\right)-\mathcal R^{1}
\left(\lambda G_{p}+u_{\phi_0,\phi_2}+v_{2}\right)
\right\|_{C^{0,\alpha}_{\nu-4}(\Omega_{r,r_{1}})} 
\leq C {(r^{\frac{m}{2}}+r^{m-\delta_4})}\left\|v_{1}-v_{2}\right\|_{C_{\nu}^{4, \alpha}\left(M_{r}\right)}.$$
Using again the fact that the right inverse $G_{r,g_0}$ is bounded 
independently of $r$, $u_{\psi_0,\psi_2}\equiv 0$ in $M_{r_1}$, and 
the two previous estimates, for $r>0$ small enough, we get that 
\begin{equation}\label{fixpoint2}
\|\mathcal M_{r}(\lambda,\psi_0,\psi_2,v_1)- 
\mathcal M_{r}(\lambda,\psi_0,\psi_2,v_2)
\|_{C^{4,\alpha}_{\nu}(M_{r})} \leq 
\frac{1}{2}\|v_{1}-v_{2}\|_{C^{4,\alpha}_{\nu}(M_{r})}.
\end{equation}

Therefore, by \eqref{fixpoint1} and \eqref{fixpoint2} the proof of the proposition is complete.
\end{proof}

\begin{theorem}\label{existenceexterior}
Let $\nu \in (9/2-n,5-n)$, $\delta_4 \in (0,1/2)$ and $\beta,\gamma>0$ 
fixed constants. There exists $r_{2}>0$, such that for all $r\in 
(0,r_{2})$, $\lambda \in \mathbb R$ with $|\lambda|\leq 
r^{{n-4+\frac{m}{2}}}$ and $\psi_0,\psi_2 \in C^{4,\alpha}
(\mathbb{S}^{n-1}_{r})$ satisfying $\|(\psi_0,\psi_2)\|_{(4,\alpha),r} 
\leq \beta  r^{{m}-\delta_4}$, there exists a solution 
$V_{\lambda,\psi_0,\psi_2} \in C^{4,\alpha}_{\nu}(M_{r})$ of 

$$H_{g_0}(1+\lambda G_p + u_{\psi_0,\psi_2} + V_{\lambda,\psi_0,\psi_2}) 
= 0,\quad\mbox{ in } M_r.$$ Moreover,
$
\|V_{\lambda,\psi_0,\psi_2}\|_{C^{4,\alpha}_{\nu}(M_r)}\leq 
\gamma r^{{m}-\nu}
$
and 
\begin{equation}\label{eq074}
\|V_{\lambda,\psi_0,\psi_2} - V_{\lambda,\tilde\psi_0,\tilde\psi_2}
\|_{C^{4,\alpha}_{\nu}(M_r)} \leq cr^{\delta_5-\nu}\|(\psi_0,\psi_2) - (\tilde\psi_0,\tilde\psi_2)  \|
\end{equation}
for some sufficiently small constant $\delta_5 >0$ independent of $r$.
\end{theorem}

\begin{proof}
After the previous proposition, it remains only to show \eqref{eq074}, 
which we prove using the same method as in the
proof of inequality \eqref{eq066}.
\end{proof}

\section{One-point gluing procedure}\label{sec:onegluingprocedure}
In this section, we complete our gluing construction 
in the case of a single puncture by matching the boundary 
data of our interior and exterior solutions. 

We begin our discussion by 
explaining why matching the 
boundary data of our interior and exterior solutions across their 
common boundary gives us a smooth, global solution. The weak form of \eqref{ourequation} states
\begin{eqnarray}\label{weak_form1}  
0 & = & \int_{M \backslash \{ p \} } \left ( \Delta_g\phi \Delta_g u 
- \frac{4}{n-2} \operatorname{Ric}_g (\nabla \phi, \nabla u) 
+ \frac{(n-2)^2 + 4}{(2(n-1)(n-2)} R_g \langle \nabla \phi, \nabla u\rangle 
\right . \\ \nonumber 
&& \quad \quad \quad \left. + \frac{n-4}{2} Q_g u \phi 
- \frac{n(n-4)(n^2-4)}{16} u^{\frac{n+4}{n-4}}\phi\right ) d\mu_g 
\end{eqnarray} 
for every $\phi \in C^\infty_0 (M \backslash \{ p \} )$. 
Now let $\mathcal{A}$ denote the interior solution, defined on 
$\overline{B_r(p)} \backslash \{ p \}$ for a sufficiently small $r$, 
and let $\mathcal{B}$ denote the exterior solution defined 
on $M \backslash B_r(p)$. Integrating by parts in \eqref{weak_form1} 
we obtain 
\begin{eqnarray*} 
0 & = & \int_{M \backslash \{ p \} } \left ( \Delta_g\phi \Delta_g u 
- \frac{4}{n-2} \operatorname{Ric}_g (\nabla \phi, \nabla u) 
+ \frac{(n-2)^2 + 4}{(2(n-1)(n-2)} R_g \langle \nabla \phi, \nabla u\rangle 
\right . \\ \nonumber 
&& \quad \quad \quad \left. + \frac{n-4}{2} Q_g u \phi 
- \frac{n(n-4)(n^2-4)}{16} u^{\frac{n+4}{n-4}}\phi\right ) d\mu_g \\ 
& = & \int_{\partial B_r(p)}\left ( \frac{\partial \phi}{\partial r} 
(\Delta_g \mathcal{B} - \Delta_g \mathcal{A} \right ) + \phi 
\left ( \frac{\partial \Delta_g \mathcal{A}}{\partial r} - 
\frac{\partial \Delta_g \mathcal{B}}{\partial r} \right ) d\mu_g 
\\ \nonumber 
&& + \int_{\partial B_r(p)} \left ( - \frac{4}{n-2} 
\operatorname{Ric}_g (\partial_r, \nabla \mathcal{A} 
- \nabla \mathcal{B} ) + \frac{(n-2)^2+4}{2(n-1)(n-2)} R_g \phi
(\partial_r \mathcal{A} - \partial_r \mathcal{B}) \right ) d\mu_g.
\end{eqnarray*} 
Thus we see that, provided $\mathcal{A}$ and $\mathcal{B}$ 
satisfy the compatibility equations 
\begin{equation*} \label{compat1} 
\mathcal{A} = \mathcal{B}, \quad \partial_r \mathcal{A}
= \partial_r \mathcal{B}, \quad \Delta_g \mathcal{A} 
= \Delta_g \mathcal{B}, \quad \partial_r \Delta_g \mathcal{A}
 = \partial_r \Delta_g \mathcal{B}
 \end{equation*}
 along the boundary $\partial B_r(p)$ we obtain a weak solution 
 to \eqref{ourequation} which is smooth, and therefore also a 
 strong solution.

By Theorem \ref{teo001}, for each sufficiently small $\varepsilon
> 0$ there exists a mapping $\mathcal{A}_\varepsilon$ 
such that the metric defined in $\overline{B_{r_\varepsilon}
(p)} \backslash \{ p \} )$ by 
\begin{eqnarray} \label{eq079}
\hat g & =& \left ( \mathcal A_\varepsilon
(R, a, \phi_0, \phi_2) \right )^{\frac{4}{n-4}}g \\ \nonumber
& =&  
\left ( u_{\varepsilon,R, a}{+\Upsilon} + v_{\phi_0, \phi_2} 
+ w_{\varepsilon,R} + U_{\varepsilon, R, a, \phi_0, \phi_2}
\right )^{\frac{4}{n-4}} g\end{eqnarray} 
has $Q$-curvature equal to $\frac{n(n^2-4)}{8}$. Also, 
by Theorem \ref{existenceexterior}, for each sufficiently 
small $\varepsilon> 0$ there exists a mapping $\mathcal{B}_\varepsilon$
such that the metric defined in $M_{r_\varepsilon}$ by 
\begin{eqnarray} \label{eq080}
\tilde g & = & \left ( \mathcal B_{\varepsilon}
(\lambda, \psi_0, \psi_2) \right )^{\frac{4}{n-4}}g \\ \nonumber 
& = & \left ( f (1+ \lambda  G_p + u_{\psi_0, \psi_2} 
+ V_{\lambda, \psi_0, \psi_2} )
\right )^{\frac{4}{n-4}} g\end{eqnarray}
has $Q$-curvature equal to $\frac{n(n^2-4)}{8}$, 
where $f-1 = O(|x|^2)$ in conformal 
normal coordinates.

	We would like now to define our metric to be $\hat g
	= \mathcal{A}_\varepsilon^{\frac{4}{n-4}} g$ in 
	the punctured ball $B_r(p) \backslash \{ p \}$ and 
	$\tilde g = \mathcal{B}_\varepsilon^{\frac{4}{n-4}} g$ in $M 
	\backslash B_r (p)$. However, this is 
	not {\it a priori} a smooth (or even continuous!) metric. 
	To complete our gluing construction, we show that one can 
	choose geometric parameters $a$, $R$, $\lambda$, 
	$\phi_0$ and $\phi_2$ such that 
	\begin{equation}\label{gluingsystem}
	\left\{
	\begin{array}{rcl}
	\mathcal A_{\varepsilon} &=& \mathcal B_{\varepsilon}\\
	\partial_{r}\mathcal A_{\varepsilon} &=& \partial_{r}\mathcal B_{\varepsilon}\\
	\Delta_g \mathcal A_{\varepsilon} &=& \Delta_g \mathcal B_{\varepsilon}\\
	\partial_r \Delta_g \mathcal A_{\varepsilon}& =& \partial_r \Delta_g \mathcal B_{\varepsilon}.
	\end{array}
	\right.    
	\end{equation}
	
	To solve this system of equations, we once again decompose our 
	boundary data into high and low Fourier modes, writing 
	$$\phi_i(\theta) = \sum_{j=0}^\infty (\phi_i)_j e_j(\theta), 
	\quad \pi'(\phi_i) = \sum_{j=0}^n (\phi_i)_j e_j, 
	\quad \pi''(\phi_i) = \sum_{j=n+1}^\infty (\phi_i)_j e_j,$$ 
	where $e_j$ is the $j$th eigenfunction of 
	$\Delta_{\mathbb{S}^{n-1}}$, counted with multiplicity and 
	normalized to have $L^2$-norm equal to $1$. We will 
	sometimes abbreviate $\pi'(\phi_i) = \phi_i'$ 
	and $\pi''(\phi_i) = \phi_i''$. We similarly decompose 
	$\psi_i = \psi_i' + \psi_i''$ into low and high Fourier 
	modes, respectively. 

\subsection{High Fourier modes}
In this subsection use the model Navier-to-Neumann operator 
we constructed in Corollary \ref{cor002} to show that we 
can match the high Fourier modes of our boundary data. 
{Observe that $\Upsilon$, defined in \eqref{eq0008}, is radial in the ball  
${B}_{r_\varepsilon}(p)\subset M$, which implies that it will only appear 
in the constant term of the Fourier expansions. In particular, 
$$\partial_r^k \pi'' (\left. \Upsilon \right |_{\partial 
{B}_{r_\varepsilon}(p)} ) = 0, \qquad k=0,1,2,3,$$ 
and so we do not encounter $\Upsilon$ in matching the high 
Fourier modes of the boundary data of the interior and 
exterior solutions.}

Before we do this, however, we must first ensure that 
$$\pi'' (\left. \mathcal{A}_\varepsilon \right |_{\partial 
B_{r_\varepsilon}
(p)} ) = \pi'' (\left . \mathcal{B}_\varepsilon
\right |_{\partial B_{r_\varepsilon}(p)} ) , \qquad 
\pi'' (\left. \Delta_g \mathcal{A}_\varepsilon
\right |_{\partial B_{r_\varepsilon}
(p)} ) = \pi'' (\left . \Delta_g \mathcal{B}_\varepsilon 
\right |_{\partial B_{r_\varepsilon}(p)} ).$$
Using 
$$\pi''(\left. v_{\phi_0, \phi_2} \right |_{\partial 
B_{r_\varepsilon} (p)} )= \phi_0'' , \quad \pi''
(\left. w_{\varepsilon,R} \right |_{\partial B_{r_\varepsilon}
(p)} ) = 0 , \quad 
\pi'' ( \left. U_{\varepsilon,R,a,\phi_0,\phi_2} \right |_{\partial 
B_{r_\varepsilon} (p)} ) = 0$$ 
and 
$$\pi'' (\left. f \right |_{\partial B_{r_\varepsilon} (p)} ) 
= f'', \quad \pi'' (\left. u_{\psi_0,\psi_2} \right |_{\partial 
B_{r_\varepsilon} (p)}) = \psi_0'' , \quad 
\left. V_{\lambda, \psi_0, \psi_2} \right |_{\partial 
B_{r_\varepsilon} (p)} = 0,$$
we see from \eqref{eq079} and \eqref{eq080} that  
\begin{eqnarray*} 
\pi'' (\left. \mathcal{A}_\varepsilon \right |_{\partial B_{r_\varepsilon}
(p)} ) & = & \pi'' (\left. u_{\varepsilon,R,a} \right |_{\partial B_{r_\varepsilon}
(p)} ) + \phi_0''  \\ 
\pi'' (\left . \mathcal{B}_\varepsilon
\right |_{\partial B_{r_\varepsilon}(p)} )& = & f'' + \lambda \pi'' (
\left. fG_p \right |_{\partial B_{r_\varepsilon} (p)} ) 
+ \pi''((f-1)\psi_0)+\psi_0''.
\end{eqnarray*} 
Thus taking 
\begin{equation} \label{match_high_0}
\phi_0'' = f''+ \lambda \pi'' (\left. fG_p \right|_{\partial 
B_{r_\varepsilon}(p)} ) + \pi''((f-1)\psi_0)+\psi_0'' - \pi'' (\left. u_{\varepsilon,R,a} \right |_{\partial 
B_{r_\varepsilon}(p)} )
\end{equation} 
gives us the first identity 
$$\pi'' (\left. \mathcal{A}_\varepsilon \right |_{\partial 
B_{r_\varepsilon}
(p)} ) = \pi'' (\left . \mathcal{B}_\varepsilon 
\right |_{\partial B_{r_\varepsilon}(p)} ). $$

Next we use the fact that 
$$\pi'' (\left. \Delta v_{\phi_0,\phi_2} \right |_{\partial 
B_{r_\varepsilon} (p)} ) = \phi_2'' , \quad \pi'' (\left. 
\Delta w_{\varepsilon, R} \right |_{\partial B_{r_\varepsilon}
(p)} ) = 0, \quad \pi'' (\left . \Delta U_{\varepsilon,R,a,
\phi_0,\phi_2} \right |_{\partial B_{r_\varepsilon}(p) } ) =0$$ 
and 
$$\pi'' (\left. \Delta u_{\psi_0,\psi_2} \right |_{\partial 
B_{r_\varepsilon}(p)}) = \psi_2'' , \quad \left. \Delta 
V_{\lambda,\psi_0,\psi_2} \right |_{\partial B_{r_\varepsilon}(p)} 
= 0$$ 
to see 
\begin{align*}
\pi'' (\left. \Delta_g \mathcal{A}_\varepsilon \right |_{\partial 
B_{r_\varepsilon} (p)})  = & \pi''(\left(\Delta_g-\Delta)
(w_{\varepsilon,R}+v_{\phi_0,\phi_2}+U_{\varepsilon,R,a,\phi_0,\phi_2})
\right |_{\partial B_{r_\varepsilon} (p)})+\phi_2''\\
& +\pi'' (\left. \Delta_g (u_{\varepsilon,
R,a}{+\Upsilon}) \right |_{\partial B_{r_\varepsilon} (p)})\\
\pi'' (\left. \Delta_g \mathcal{B}_\varepsilon \right |_{\partial 
B_{r_\varepsilon} (p)})  = &\pi'' (\left. \Delta_g ((f-1)
(1+u_{\psi_{0},\psi_2}+V_{\lambda,\psi_0,\psi_2})+\lambda fG_p)
\right |_{\partial B_{r_\varepsilon} (p)} ) \\
&  +\pi'' (\left. (\Delta_g-\Delta)u_{\psi_0,\psi_2}
\right |_{\partial B_{r_\varepsilon} (p)})+\psi_2''. 
\end{align*}

Thus we need to choose $\phi_2''$ satisfying
\begin{equation}\label{match_high_2}
    \begin{array}{rcl}
         \phi_2''   & = & \pi'' (\left. \Delta_g 
         ((f-1)(1+u_{\psi_{0},\psi_2}+V_{\lambda,\psi_0,\psi_2})+\lambda fG_p)
\right |_{\partial B_{r_\varepsilon} (p)} )+\psi_2''  \\
& &  +\pi'' (\left. (\Delta_g-\Delta)u_{\psi_0,\psi_2}
\right |_{\partial B_{r_\varepsilon} (p)}) -\pi'' (\left. 
\Delta_g (u_{\varepsilon,
R,a}{+\Upsilon}) \right |_{\partial B_{r_\varepsilon} (p)})\\
& & -\pi''(\left(\Delta_g-\Delta)(w_{\varepsilon,R}
+v_{\phi_0,\phi_2}+U_{\varepsilon,R,a,\phi_0,\phi_2})\right |_{\partial B_{r_\varepsilon} (p)}) \\ 
& =: & \mathscr{X}_{\varepsilon,R,a,\lambda,\psi_0,\psi_2} (\phi_2). 
    \end{array}
\end{equation}
To solve \eqref{match_high_2} it suffices to find a fixed point of 
the mapping 
$$\phi_2'' \mapsto \mathscr{X}_{\varepsilon,R,a,\lambda, \psi_0, \psi_2} 
(\phi_2'')$$
for a given choice of parameters $\varepsilon, R, a,\lambda,\psi_0,\psi_2$. 
We have $\left. (f-1) \right |_{\partial B_{r_\varepsilon}(p)} = O(r_\varepsilon^2)$ 
and similarly all the coefficients of the second order operator $\Delta_g - \Delta$
are $O(r_\varepsilon^2)$. Therefore, following the same arguments as 
in the proof of Proposition \ref{IN1} we find that $\mathscr{X}$ is 
a contraction provided we choose all the parameters to be sufficiently 
small. We conclude that \eqref{match_high_2} admits a 
solution $\phi_2''$ depending continuously on the parameters 
$\varepsilon,R,a,\lambda, \psi_0, \psi_2$.

For the remainder of this construction, for given boundary 
functions $\psi_0$ and $\psi_2$, we choose $\phi_0$ to 
satisfy \eqref{match_high_0} and choose $\phi_2$ to 
satisfy \eqref{match_high_2}. In this fashion we may regard 
$\phi_0$ and $\phi_2$ as functions of the parameters $\psi_0$ 
and $\psi_2$. 

It remains to show that 
$$\pi''(\left. (\partial_r \mathcal{A}_\varepsilon )
\right |_{\partial B_{r_\varepsilon}}) 
= \pi'' (\left. (\partial_r\mathcal{B}_\varepsilon )
\right |_{\partial B_{r_\varepsilon}}), 
\qquad 
\pi'' (\left. (\partial_r\Delta_g \mathcal{A}_\varepsilon) 
\right |_{\partial B_{r_\varepsilon}})= \pi'' (\left. (\partial_r 
\Delta_g \mathcal{B}_\varepsilon) \right |_{\partial B_{r_\varepsilon}}).$$

We see from \eqref{eq079} and \eqref{eq080} that
\begin{align*}
     \pi''(\left. (\partial_r \mathcal{A}_\varepsilon )
\right |_{\partial B_{r_\varepsilon}}) 
= &  \pi'' \left ( \left. \left(\partial_r 
 (u_{\varepsilon,R,a}{+\Upsilon})+w_{\varepsilon,R}+U_{\varepsilon,R,a,\phi_0,\phi_2}\right)
\right |_{\partial B_{r_\varepsilon}} \right )+\left.\partial_rv_{\phi_0''-\psi_0'',\phi_2''-\psi_2''}\right|_{\partial B_{r_\varepsilon}}\\
& +\left.\partial_rv_{\psi_0'',\psi_2''}\right|_{\partial B_{r_\varepsilon}},\\
\pi'' (\left. (\partial_r\mathcal{B}_\varepsilon )
\right |_{\partial B_{r_\varepsilon}})= &  \pi'' \left ( \left. \partial_r
\left( f(1+\lambda G_p
 + V_{\lambda,\psi_0, \psi_2})+(f-1)u_{\psi_0,\psi_2} \right)
\right |_{\partial B_{r_\varepsilon} } \right ) +\left.\partial_ru_{\psi_0'',\psi_2''}\right|_{\partial B_{r_\varepsilon}},\\
\pi'' (\left. (\partial_r\Delta_g \mathcal{A}_\varepsilon) 
\right |_{\partial B_{r_\varepsilon}})= & \pi''\left( \left.\partial_r\left(\Delta_g\left( u_{\varepsilon,R,a}{+\Upsilon}+w_{\varepsilon,R}+U_{\varepsilon,R,a,\phi_0,\phi_2}\right)+(\Delta_g-\Delta)v_{\phi_0,\phi_2} \right)\right|_{\partial B_{r_\varepsilon}}\right)\\
& +\left.\partial_r\Delta v_{\phi_0''-\psi_0'',\phi_2''-\psi_2''}\right|_{\partial B_{r_\varepsilon}}+\left.\partial_r\Delta v_{\psi_0'',\psi_2''}\right|_{\partial B_{r_\varepsilon}},\\
\pi'' (\left. (\partial_r 
\Delta_g \mathcal{B}_\varepsilon) \right |_{\partial B_{r_\varepsilon}})= & \pi'' \left ( \left. \partial_r 
\Delta_g \left ( f(1+\lambda G_p 
+ V_{\lambda, \psi_0, \psi_2} \right ) +(f-1)u_{\psi_0,\psi_2})\right |_{\partial 
B_{r_\varepsilon} } \right )\\
& +\pi''\left(\left. \partial_r(\Delta_g-\Delta)u_{\psi_0,\psi_2}\right |_{\partial 
B_{r_\varepsilon} }\right)+\left. \partial_r\Delta u_{\psi_0'',\psi_2''} \right |_{\partial 
B_{r_\varepsilon} }.
\end{align*}

Multiplying the first two equations by 
$r_\varepsilon$ and the other two by $r_\varepsilon^3$ we obtain the system 
\begin{equation}\label{gluingsystem2} 
\left\{\begin{array}{rcl}
\mathcal{S}_\varepsilon (\psi_0'', \psi_2'') & = &  r_\varepsilon 
\left. \partial_r (v_{\psi_0'', \psi_2''} - 
u_{\psi_0'', \psi_2''} ) \right |_{\partial B_{r_\varepsilon}
}\\
\mathcal{T}_\varepsilon (\psi_0'', \psi_2'') & = &  
r_\varepsilon^3 \left. \partial_r \Delta (v_{\psi_0'', \psi_2''}
- u_{\psi_0'', \psi_2''}) \right |_{\partial B_{r_\varepsilon}},
\end{array}\right.
\end{equation} 
where
\begin{align*}
    \mathcal{S}_\varepsilon (\psi_0'', \psi_2'') = & \pi'' \left ( \left. \partial_r
\left( f(1+\lambda G_p
 + V_{\lambda,\psi_0, \psi_2})+(f-1)u_{\psi_0,\psi_2} \right)
\right |_{\partial B_{r_\varepsilon} } \right )  \\ 
&-\pi'' \left ( \left(\left. \partial_r 
 (u_{\varepsilon,R,a}{+\Upsilon})+w_{\varepsilon,R}+U_{\varepsilon,R,a,\phi_0,\phi_0}\right)
\right |_{\partial B_{r_\varepsilon}} \right ) -\left.\partial_rv_{\phi_0''-\psi_0'',\phi_2''-\psi_2''}\right|_{\partial B_{r_\varepsilon}}\\
    \mathcal{T}_\varepsilon (\psi_0'', \psi_2'') =  & \pi'' \left ( \left. \partial_r 
\Delta_g \left ( f(1+\lambda G_p 
+ V_{\lambda, \psi_0, \psi_2} \right ) +(f-1)u_{\psi_0,\psi_2})\right |_{\partial 
B_{r_\varepsilon}} \right )\\
& +\pi''\left(\left. \partial_r(\Delta_g-\Delta)u_{\psi_0,\psi_2}\right |_{\partial 
B_{r_\varepsilon}}\right)-\left.\partial_r\Delta v_{\phi_0''-\psi_0'',\phi_2''-\psi_2''}\right|_{\partial B_{r_\varepsilon}}\\
& -\pi''\left( \left.\partial_r\left(\Delta_g\left( u_{\varepsilon,R,a}{+\Upsilon}+w_{\varepsilon,R}+U_{\varepsilon,R,a,\phi_0,\phi_2}\right)+(\Delta_g-\Delta)v_{\phi_0,\phi_2} \right)\right|_{\partial B_{r_\varepsilon}}\right).
\end{align*}
(Recall that $\psi_0$ and $\psi_2$ determine 
$\phi_0$ and $\phi_2$ in the formulas above.)

Recalling the definition of the isomorphism  
$\mathcal{Z}_{r_\varepsilon}$ we constructed in 
Corollary \ref{cor002}, we see that 
solving \eqref{gluingsystem2} is equivalent to 
$$\mathcal{Z}_{r_\varepsilon} (\psi_0'', \psi_2'')
= (\mathcal{S}_\varepsilon (\psi_0'', \psi_2''), 
\mathcal{T}_\varepsilon (\psi_0'', \psi_2'') ),$$ 
and so it suffices to find a fixed point of the mapping 
$$\mathcal{H}_\varepsilon : \mathcal{D}_\varepsilon 
\rightarrow \pi'' (C^{4,\alpha}
(\mathbb{S}^{n-1}_{r_\varepsilon}, \mathbb{R}^2), 
\qquad \mathcal{H}_\varepsilon (\psi_0'', \psi_2'') 
=  \mathcal{Z}_{r_\varepsilon}^{-1} (\mathcal{S}_\varepsilon
(\psi_0'', \psi_2''), \mathcal{T}_\varepsilon (\psi_0, \psi_2) ),$$
where 
$$\mathcal{D}_{\varepsilon}:=\{(\psi_0'',\psi_2'')
\in \pi''(C^{4,\alpha}(\mathbb{S}_{r_\varepsilon}^{n-1}; 
\mathbb{R}^{2})): \|(\psi_0'',\psi_2'')\|_{(4,\alpha),1} 
\leq r_{\varepsilon}^{{m}-\delta_1}\}.$$

\begin{lemma}\label{lem013}
Let $a\in \mathbb{R}^{n}$, $b$, $\lambda \in \mathbb R$ be constants 
with  $|a|^{2} \leq r_\varepsilon^{{m-2}}$, $|b|\leq 1/2$ 
and $|\lambda|\leq r_\varepsilon^{{n-4+\frac{m}{2}}}$. 
Then there exists $\varepsilon_0>0$ such 
that for $\varepsilon\in (0,\varepsilon_0)$, the map 
$\mathcal H_{\varepsilon}$ has a fixed point. Moreover, this 
fixed point depends continuously on the parameters.
\end{lemma}
\begin{proof} The proof of this lemma is based in the same ideas 
of the previous results where we have used the fixed point argument. 
We have to prove that the map $\mathcal H_\varepsilon$ is a 
contraction with the norm of $\mathcal H_\varepsilon(0,0)$ bounded 
by $\frac{1}{2}r_{\varepsilon}^{{m}-\delta_1}$. To this 
end we use that the map $\mathcal Z_{r_\varepsilon}$ has norm bounded
independently of $r_\varepsilon$ and the estimates obtained in 
Propositions \ref{IN1} and \ref{prop005}, as 
well \eqref{eq067}, \eqref{eq066} and \eqref{eq074}. This is a 
long and technical computation, which we be omitted.
\end{proof}

\subsection{Constant functions}
We denote the fixed point of the mapping $\mathcal{H}_\varepsilon$ 
by $(\boldsymbol \psi_0'', \boldsymbol \psi_2'')$ and, abusing 
notation slightly, we denote the associated boundary functions 
for $\mathcal{A}_\varepsilon$ as $(\boldsymbol \phi_0'', 
\boldsymbol\phi_2'')$.

By 
Remark \ref{remark002}, \eqref{eq059} and \eqref{eq0008} we have
$$u_{\varepsilon,R}(x)+\Upsilon=1+b+\frac{\alpha_\varepsilon^2}{4(1+b)}|x|^{4-n}
+\frac{\beta_\varepsilon}{2}\left(\frac{\alpha_\varepsilon}{2(1+b)}\right)^{\frac{n}{n-4}}|x|^{2-n}
+O^{(4)}(\varepsilon^{2\frac{n+4}{n-4}}|x|^{-n}).$$
Note that 
$$\frac{\beta_\varepsilon}{2}\left(\frac{\alpha_\varepsilon}{2(1+b)}\right)^{\frac{n}{n-4}}r_\varepsilon^{2-n}=O(\varepsilon^{\frac{2(n-2)}{n-4}-s(n-2)})\quad\mbox{ and }\quad \varepsilon^{2\frac{n+4}{n-4}}r_\varepsilon^{-n}=\varepsilon^{2\frac{n+4}{n-4}-sn},$$
with $\frac{2(n-2)}{n-4}-s(n-2)>0$ and $2\frac{n+4}{n-4}-sn>0$. Assuming the hypotheses of 
Lemma \ref{lem013}, combining \eqref{eq033} and \eqref{eq076} with the bound $|a|^2 
\leq r_\varepsilon^{{m-2}}$ gives us 
\begin{align}
    \mathcal A_{\varepsilon}(r_{\varepsilon}\theta) & = 
\displaystyle 1+b+\frac{\alpha_\varepsilon^2}{4(1+b)}r_\varepsilon^{4-n}+\left((n-4)u_{\varepsilon,R}
(r_{\varepsilon}\theta)+r_{\varepsilon}
\partial_r u_\varepsilon(r_{\varepsilon}\theta)\right)
r_{\varepsilon} a\cdot\theta  +w_{\varepsilon,R}(r_{\varepsilon}\theta) 
\label{eq081}\\
& \displaystyle + v_{\boldsymbol{\phi_0},\boldsymbol{\phi_2},}(r_{\varepsilon}\theta)+ U_{\varepsilon,R,a,\boldsymbol{\phi_0},\boldsymbol{\phi_2}}
(r_{\varepsilon}\theta)+ O( r_\varepsilon^{m}) + O(
r_\varepsilon^{4})\nonumber
\end{align}
\color{black}
with  the last term independent
of $\theta$. We also have
\begin{eqnarray}\label{eq082}
    \mathcal B_{\varepsilon}(r_{\varepsilon}\theta)  & = &  
    1+\lambda r_{\varepsilon}^{4-n} 
    + u_{\boldsymbol{\psi_0},\boldsymbol{\psi_2}}
    (r_{\varepsilon}\theta) + (f-1)(r_{\varepsilon}\theta) 
    \\ \nonumber 
    && 
    + (f-1)u_{\boldsymbol{\psi_0}, \boldsymbol{\psi_2}}
(r_{\varepsilon}\theta)+ 
(fV_{\lambda, \boldsymbol{\psi_0},\boldsymbol{\psi_2}})
(r_{\varepsilon}\theta) + O(|\lambda|r_{\varepsilon}^{5-n}).
\end{eqnarray}

Now we project the system \eqref{gluingsystem} onto the space generated 
by the constant functions. Let us define $\xi_0$, $\xi_2$ as 
the zero Fourier modes of $\boldsymbol{\phi_0}$
and $\boldsymbol{\phi_2}$, respectively. 
Using Corollary \ref{cor:poissoninterior} and \ref{cor:poissonexterior}, 
we obtain

\begin{equation}\label{gs3}
	\left\{\begin{array}{rcl}
	\displaystyle b
	-\left(\frac{\alpha_\varepsilon^2}{4(1+b)}-\lambda\right) r_\varepsilon^{4-n} + \frac{1}{2n}\xi_0 & = 
	& \mathcal H^{0}_{\varepsilon}\\
	\displaystyle(4-n)\left(\frac{\alpha_\varepsilon^2}{4(1+b)}-\lambda\right) r_\varepsilon^{4-n} + \frac{1}{n}\xi_0
	-\frac{1}{4-n}\xi_2 & =  & r_\varepsilon\partial_r
	\mathcal H_{\varepsilon}^{0}\\
	\displaystyle 2(4-n)\left(\frac{\alpha_\varepsilon^2}{4(1+b)}-\lambda\right) r_\varepsilon^{4-n} + \xi_{0}-\xi_2& = 
	& r_{\varepsilon}^{2}\Delta_{g}\mathcal H_{\varepsilon}^{0} \\
	\displaystyle -2(4-n)(2-n)\left(\frac{\alpha_\varepsilon^2}{4(1+b)}-\lambda\right) r_\varepsilon^{4-n} - (2-n)\xi_{2} 
	&=& r_{\varepsilon}^{3}\partial_r\Delta_g 
	\mathcal{H}_{\varepsilon}^{0},
	\end{array}\right.
\end{equation}
\color{black}
where  $\mathcal{H}_{\varepsilon}^{0}$, 
$r_\varepsilon\partial_r\mathcal{H}_{\varepsilon}^{0}$, 
$r_\varepsilon^2\Delta_g\mathcal{H}_{\varepsilon}^{0}$ and 
$r_\varepsilon^3\partial_r\Delta_g\mathcal{H}_{\varepsilon}^{0}$ depends 
continuously on the parameters $b$, $\lambda$, $\xi_0$ and $\xi_2$. Also, 
by \eqref{eq081}, \eqref{eq082} and the estimates obtained in 
Sections \ref{sec:fixedpointargument} and \ref{sec:constant-q-curvature} 
we obtain that  $\mathcal{H}_{\varepsilon}^{0}$, 
$r_\varepsilon\partial_r\mathcal{H}_{\varepsilon}^{0}$, 
$r_\varepsilon^2\Delta_g\mathcal{H}_{\varepsilon}^{0}$ and 
$r_\varepsilon^3\partial_r\Delta_g\mathcal{H}_{\varepsilon}^{0}$ have 
order $r_\varepsilon^{{m}}$.

\begin{lemma}\label{lem014}
Let $a\in \mathbb{R}^{n}$ with $|a|^{2} \leq r_\varepsilon^{{m-2}}$ and 
$\omega_i\in C^{4,\alpha}(\mathbb{S}^{n-1})$ belonging to the space spanned by the 
coordinate functions and with norm bounded by $r_{\varepsilon}^{{m}-\delta_{1}}$. 
There exists a constant $\varepsilon_1>0$ such that for $\varepsilon\in (0,\varepsilon_1)$ 
the system \eqref{gs3} has a solution $(b,\lambda,\xi_0,\xi_2)$ with $|b|\leq 1/2$, 
$|\lambda|\leq r_\varepsilon^{{n-4+\frac{m}{2}}}$ and $|\xi_i|
\leq r_\varepsilon^{{m}-\delta_1}$.
\end{lemma}
\begin{proof} Define a continuous map $\mathcal{G}_{\varepsilon,0}:\mathcal D_{\varepsilon,0}\rightarrow \mathbb{R}^4$ as

$$\begin{array}{rcl}
     \mathcal G_{\varepsilon,0}(b,\lambda,\xi_0,\xi_{2}) & =  
     &\displaystyle \left(\left. \mathcal H_{\varepsilon} + 
     \frac{r_{\varepsilon}}{n-2}\partial_r\mathcal H_{\varepsilon} 
     - \frac{r_{\varepsilon}^{2}}{2(n-2)}\Delta_g\mathcal H_{\varepsilon} 
     - \frac{r_{\varepsilon}^{3}}{2(n-4)(n-2)}\partial_r\Delta_g\mathcal 
     H_{\varepsilon},\right.\right.
     \\
     & &\displaystyle \frac{\alpha_\varepsilon^2}{4(1+b)} 
     +\frac{r_{\varepsilon}^{n-3}}{n-2}\partial_r\mathcal 
     H_{\varepsilon}-\frac{r^{n-2}_{\varepsilon}}{n(n-2)}
     \Delta_g\mathcal H_{\varepsilon} - \frac{r_{\varepsilon}^{n-1}}{n(n-4)(n-2)}
     \partial_r\Delta_g \mathcal H_{\varepsilon},
     \\
     & &\displaystyle  r_{\varepsilon}^{2}\Delta_g \mathcal{H}_{\varepsilon}
     + \frac{r^{3}_{\varepsilon}}{n-2}\partial_r\Delta_g\mathcal{H}_{\varepsilon},
     \\
     & & \displaystyle\left.\frac{2(4-n)}{(2-n)}\left(r_{\varepsilon}
     \partial_r\mathcal H_{\varepsilon} - \frac{r^{2}_{\varepsilon}}{n}
     \Delta_g\mathcal H_{\varepsilon} + \frac{r_{\varepsilon}^{3}}{2n}
     \partial_r\Delta_g \mathcal H_{\varepsilon}\right)\right),
\end{array}$$
\color{black}
where $\mathcal D_{\varepsilon,0}:=\left\{(b,\lambda,\xi_{0},\xi_{2})
\in\mathbb R^4 : \; |b|\leq 1/2,\; |\lambda|\leq 
r_{\varepsilon}^{{n-4+\frac{m}{2}}}\;\;{\rm and }\;\;|\xi_i|\leq r_\varepsilon^{{m}-\delta_1}\right\}$. Using 
the fact that that $\mathcal{H}_{\varepsilon}^{0}$, $r_\varepsilon\partial_r\mathcal{H}_{\varepsilon}^{0}$,
$r_\varepsilon^2\Delta_g\mathcal{H}_{\varepsilon}^{0}$ and $r_\varepsilon^3\partial_r\Delta_g\mathcal{H}_{\varepsilon}^{0}$ 
are all  $O(r_\varepsilon^{{m}})$, we can show that $\mathcal G_{\varepsilon,0}
(\mathcal{D}_{\varepsilon,0})\subset \mathcal D_{\varepsilon,0}$, and so it follows 
from Brouwer's fixed point theorem that $\mathcal G_{\varepsilon,0}$ has a fixed point. 
In fact, we can prove that $\mathcal G_{\varepsilon,0}$ is a contraction, hence 
the fixed point depends continuously on the parameters. It is not difficult to see that this fixed point is a 
solution to the system \eqref{gs3}.
\end{proof}

\subsection{Coordinate functions}
Finally, we consider the solutions given by Lemmas 
\ref{lem013} and \ref{lem014}, and project the 
system \eqref{gluingsystem} in the space spanned by the 
coordinate functions. 
Let us consider that the component of the functions $\boldsymbol{\phi_0}$,
and $\boldsymbol{\phi_2}$ in the direction of the coordinates functions are given by
$$\sum_{j=1}^n\tau_j e_j,\;\;\sum_{j=1}^n\zeta_j e_j\;\;{\rm and }\;\;\sum_{j=1}^n\varrho_j e_j,$$
Using Corollary \ref{cor:poissoninterior} and \ref{cor:poissonexterior}, projecting the system \eqref{gluingsystem} in direction of $e_j$ we obtain
\begin{equation}\label{gs4}
	\left\{\begin{array}{rcl}
	\displaystyle F(r_{\varepsilon})r_{\varepsilon}a_{j} +\frac{1}{2n+4}\tau_j-\varrho_j& = & \mathcal{H}_{\varepsilon j}\\
	
	\displaystyle G(r_{\varepsilon})r_{\varepsilon}a_{j} +\frac{3}{2n+4}\tau_j-\frac{2}{3-n}\zeta_j-(1-n)\varrho_j &=& r_{\varepsilon}\partial_{r}\mathcal H_{\varepsilon j}\\
	
    \displaystyle M(r_{\varepsilon})r_{\varepsilon}a_j +\tau_j - \zeta_j&=& r_{\varepsilon}^{2}\Delta\mathcal H_{\varepsilon j}\\
    
	N(r_{\varepsilon})r_{\varepsilon}a_j+\tau_j - (1-n)\zeta_j&=& r_{\varepsilon}^{3}\partial_{r}\Delta \mathcal H_{\varepsilon j},
	\end{array}\right.
\end{equation}
where 
\begin{eqnarray*}
F(r_{\varepsilon}) & = & \left((n-4)u_{\varepsilon,R} + r_{\varepsilon}\partial_r u_{\varepsilon,R}\right)(r_{\varepsilon}\theta),\\
G(r_{\varepsilon})&=&\left((n-4)u_{\varepsilon,R}+(n-2)r_{\varepsilon}\partial_ru_{\varepsilon,R}+r_{\varepsilon}^2\partial_r^2u_{\varepsilon,R}\right)(r_{\varepsilon}\theta),\\
M(r_{\varepsilon}) &=& \left((n-3)(n+1)r_{\varepsilon}\partial_ru_{\varepsilon,R} + (2n-1)r_{\varepsilon}^{2}\partial_{r}^{2}u_{\varepsilon,R}+ r^{3}_{\varepsilon}\partial_r^{3}u_{\varepsilon,R}\right)(r_{\varepsilon}\theta),\\
N(r_{\varepsilon})&=&\left((n^2-4)r_{\varepsilon}^{2}\partial_{r}^{2}u_{\varepsilon,R} + (2n+1)r^{3}_{\varepsilon}\partial_r^{3}u_{\varepsilon,R} + r^{4}_{\varepsilon}\partial_r^{4}u_{\varepsilon,R}\right)(r_{\varepsilon}\theta),
\end{eqnarray*}
and $\mathcal{H}_{\varepsilon j}$, $r_\varepsilon\partial_r\mathcal{H}_{\varepsilon j}$, $r_\varepsilon^2\Delta_g\mathcal{H}_{\varepsilon j}$ and $r_\varepsilon^3\partial_r\Delta_g\mathcal{H}_{\varepsilon j}$ depends continuously on the parameters $a_j$, $\tau_j$, $\zeta_j$ and $\varrho_j$. Also, by \eqref{eq081}, \eqref{eq082} and the estimates obtained in Sections \ref{sec:fixedpointargument} and \ref{sec:constant-q-curvature} we obtain that  $\mathcal{H}_{\varepsilon}^{0}$, $r_\varepsilon\partial_r\mathcal{H}_{\varepsilon}^{0}$, $r_\varepsilon^2\Delta_g\mathcal{H}_{\varepsilon}^{0}$ and $r_\varepsilon^3\partial_r\Delta_g\mathcal{H}_{\varepsilon}^{0}$ have order $r_\varepsilon^{{m}}$.

\begin{lemma}\label{lem015}
There exists a constant $\varepsilon_2>0$ such that if $\varepsilon\in(0,\varepsilon_2)$, then the system \eqref{gs4} has a solution $(a_j,\tau_j,\zeta_j,\varrho_j)$.
\end{lemma}
\begin{proof} The proof of this lemma is similar to the previous one, that is, by a fixed point argument. First, by \eqref{eq059} and Remark \ref{remark002} we have
$$\mathcal T= F+\frac{G}{n-1}+c_1M+c_2N=\frac{n(n-4)}{n-1}(1+b)+O(\varepsilon^\gamma),$$
for some $\gamma>0$, where $c_i$, $i=1,2$, are constants which depends only on $n$. Then, we can reduce the system \eqref{gs4} to the system
$$\left\{\begin{array}{rcl}
    a_j & = &\mathcal T^{-1}r_\varepsilon^{-1}(\mathcal{H}_{\varepsilon j}+c_3r_\varepsilon\partial_r\mathcal{H}_{\varepsilon j}+c_4r^2_\varepsilon\Delta_g \mathcal{H}_{\varepsilon j}+c_5r^3_\varepsilon\partial_r\Delta_g\mathcal{H}_{\varepsilon j})  \\
   \vspace{-0,1cm} \\
    \tau_j & = & \displaystyle c_6r_\varepsilon^2\Delta_g\mathcal{H}_{\varepsilon j}+c_7r_\varepsilon^3\partial_r\Delta_g\mathcal{H}_{\varepsilon j}+((1-n)M-N)n^{-1}r_\varepsilon a_j\\
    \vspace{-0,1cm} \\
    \zeta_j & = & \displaystyle c_8r_\varepsilon^2\Delta_g\mathcal{H}_{\varepsilon j}+c_9r_\varepsilon^3\partial_r\Delta_g\mathcal{H}_{\varepsilon j}+(M-N)n^{-1}r_\varepsilon a_j\\
    \vspace{-0,1cm} \\
    \varrho_j & = & c_{10}r_\varepsilon\partial_r\mathcal{H}_{\varepsilon j}+c_{11}r_\varepsilon^2\Delta_g\mathcal{H}_{\varepsilon j}+c_{12}r_\varepsilon^3\partial_r\Delta_g\mathcal{H}_{\varepsilon j}+(c_{13}G+c_{14}M+c_{15}N)r_\varepsilon a_j,
\end{array}\right.$$
where $c_i$ are constants which depends only on $n$. Now, using that $\mathcal{H}_{\varepsilon j}$, $r_\varepsilon\partial_r\mathcal{H}_{\varepsilon j}$, $r_\varepsilon^2\Delta_g\mathcal{H}_{\varepsilon j}$ and $r_\varepsilon^3\partial_r\Delta_g\mathcal{H}_{\varepsilon j}$ are bounded by $r_\varepsilon^{{m}}$, we obtain a solution to this system as a fixed point of a map $\mathcal K_{j}(a_j,\tau_j,\zeta_j,\varrho_j)=(a_j,\tau_j,\zeta_j,\varrho_j)$.
\end{proof}

We summarize this construction with the following theorem. 
\begin{theorem} Let $(M^{n},g_0)$ be a nondegenerate compact Riemannian manifold of dimension $n\geq 5$, with constant Q-curvature equal to $n(n^2-4)/8$. For $n\geq 8$, suppose there exists a point $p \in M$ with the Weyl tensor satisfying 
$$\nabla^{k}W_{g_0}(p) = 0,\;\;\;\mbox{ for }\;\;k=0,\ldots, \left[\frac{n-8}{2}\right].$$
Then there exists $\varepsilon_0 >0$ and a one-parameter family of complete metrics $g_{\varepsilon}$ on $M\backslash\{p\}$, for $\varepsilon\in (0,\varepsilon_0)$, such that 
each $g_{\varepsilon}$ is conformal to $g_0$ and has constant Q-curvature equal to $n(n^2-4)/8$, $g_{\varepsilon}$ is asymptotically Delaunay
and $g_{\varepsilon}$ converges to $g_{0}$ uniformly on compact sets of $M\backslash \{p\}$ as $\varepsilon \rightarrow 0$.
\end{theorem}

\begin{proof} We keep the previous notations. Consider the metric $g=\mathcal F^{\frac{4}{n-4}}g_0$ given in Section \ref{sec:interioranalysis}.
For suitable parameters $R$, $a$ and $\phi_0,\phi_2\in C^{4,\alpha}(\mathbb S_{r_\varepsilon}^{n-1})$ with $\phi_0=\xi_0e_0+\sum_{j=1}^n\tau_je_j+\phi_0''$, $\pi''(\phi_0'')=\phi_0''$ and $\pi''(\phi_2)=\phi_2$, by Theorem \ref{teo001} there exists a family of constant $Q$-curvature complete metrics in $\overline{B_{r_\varepsilon}(p)}\backslash\{p\}\subset M$, for $\varepsilon>0$ small enough, given by $\hat g={\mathcal A}_\varepsilon^{\frac{4}{n-4}}g$. 

Also, for suitable parameters $\lambda$ and $\psi_0,\psi_2\in C^{4,\alpha}(\mathbb S_{r_\varepsilon}^{n-1})$ with $\psi_0=\xi_2e_0+\sum_{j=1}^n\zeta_je_j+\psi_0''$, $\psi_2=\sum_{j=1}^n\varrho_je_j+\psi_2''$, by Theorem \ref{existenceexterior} there exists a family of constant $Q$-curvature metrics in $M\backslash B_{r_\varepsilon}(p)$, for $\varepsilon>0$ small enough, given by $\tilde g={\mathcal B}_\varepsilon^{\frac{4}{n-4}}g$.

From Lemmas \ref{lem013}, \ref{lem014} and \ref{lem015} there exists $\varepsilon_0>0$ such that for all $\varepsilon\in(0,\varepsilon_0)$ there are parameters  for which the functions $\mathcal A_\varepsilon$ and $\mathcal B_\varepsilon$ satisfy the system \ref{gluingsystem}. Hence using elliptic regularity we can show that the function $\mathcal W_\varepsilon$, defined by $\mathcal W_\varepsilon:=\mathcal A_\varepsilon$ in $B_{r_\varepsilon}(p)\backslash\{p\}$ and $\mathcal W_\varepsilon:=\mathcal B_\varepsilon$ in $M\backslash B_{r_\varepsilon}(p)$, is a positive smooth function in $M\backslash\{p\}$. Moreover, $\mathcal W_\varepsilon$ is asymptotically to some Delaunay-type solution $u_{\varepsilon,R}$. This implies that the metric $g_\varepsilon:=\mathcal W_\varepsilon^{\frac{4}{n-4}}g$ is the one that proves the theorem.
\end{proof}

\section{Proof of the main result}\label{sec:mainresults}
Finally we are ready to prove the main result of this work, particularly
adapting our construction in the case of a single puncture to 
handle finitely many punctures. 

	Recall that the singular set is given by $\Lambda=\left\{p_{1}, 
	\ldots, p_{k}\right\}$ and at each point it satisfies 
	$\nabla^{j} W_{g_{0}}\left(p_{i}\right)=0,$ 
	for $j=0, \dots, [\frac{n-8}{2}]$ and $i=1,\dots,k$.
	
	\begin{proof}[Proof of Theorem~\ref{maintheorem1}]
	As before, we also have three steps: we construct an interior solution, 
	construct an exterior solution, and finally match boundary data across 
	an interface to obtain a globally smooth solution. Fix a positive 
	number $\delta$ and fix $t_i \in (\delta, \delta^{-1})$.  For 
	each $i =1,\dots, k$ let $\varepsilon_i = \varepsilon t_i$. Our 
	interior domain is $\cup_{i=1}^k B_{r_{\varepsilon_i}} (p_i)$, 
	where we choose $\varepsilon$ sufficiently small such that these
	balls are pairwise disjoint. Thus one can perform the same interior 
	analysis on each of these balls individually, as we have already 
	done in Section \ref{sec:interioranalysis}. Given geometric parameters 
	$R_i>0$ and $a_i\in \mathbb{R}^n$ and boundary functions $\phi_0^i \in 
	\pi'' (C^{4,\alpha} (\mathbb{S}^{n-1}_{r_{\varepsilon_i}}))$ 
	and $\phi_2^i \in C^{4,\alpha}(\mathbb{S}^{n-1}_{r_{\varepsilon_i}})$
	we obtain a constant $Q$-curvature metric of the form 
	$$\hat g = (u_{\varepsilon_i, R_i, a_i} +{\Upsilon_i+} v_{\phi_0^i, \phi_2^i}
	+ w_{\varepsilon_i, R_i} + U_{\varepsilon_i, R_i,a_i, 
	\phi_0^i, \phi_2^i})^{\frac{4}{n-4}} g_0$$ 
	on the union $\bigcup_{i=1}^k B_{r_{\varepsilon_i}}(p_i) 
	\backslash \{ p_i \}.$

    On the other hand, some adjustments are necessary in order to construct 
    a family of metrics as in Section~\ref{sec:exterioranalysis}.
    Let $\Psi_{i}: B_{2r_{0}}(0) \rightarrow M$ be a normal coordinate system with 
    respect to $g_{i}=\mathcal{F}_{i}^{{4}/{n-4}} g_{0}$ on $M$ 
    centered at $p_{i}\in\Lambda$. 
    Here, $\mathcal{F}_{i}$ is defined in Section~\ref{sec:exterioranalysis} for all $i=1,\dots,k$.
    Hence, each metric $g_{i}$ yields conformal normal coordinates centered at $p_{i}$.
    Recall that $\mathcal{F}_{i}=1+O(|x|^{2})$ in the coordinate 
    system $\Psi_{i}$. Denote by $G_{p_{i}}$ the Green's function 
    of $L_{g_{0}}$ with pole at $p_{i}$ and assume that 
    $\lim_{x \rightarrow 0}|x|^{n-4} G_{p_{i}}(x)=1$ in the 
    coordinate system $\Psi_{i}$ for all $i=1,\dots,k$. 
        Define $G_{p_{1},\dots, p_{k}}\in C^{\infty}\left(M \backslash\Lambda\right)$ to be 
        \begin{equation*}
            G_{p_{1}, \ldots, p_{k}}=\sum_{i=1}^{k} \lambda_{i} G_{p_{i}},
        \end{equation*}
        where $\lambda_{i}\in \mathbb{R}$.
        Let $r=\left(r_{\varepsilon_{1}}, \dots, r_{\varepsilon_{k}}\right)$ 
        and denote by $M_{r}:=M\backslash \bigcup_{i=1}^k\Psi_{i}
        \left(B_{r_{\varepsilon_{i}}}(0)\right)$.
        Define the space $C_{\nu}^{l, \alpha}
        \left(M \backslash\Lambda\right)$ as in 
        Definition~\ref{def4} with the following norm
        \begin{equation*}
            \|v\|_{C_{\nu}^{l, \alpha}(M \backslash\{p\})}
            :=\|v\|_{C^{l, \alpha}\left(M_{\frac{1}{2}r_{0}}\right)}
            +\sum_{i=1}^{k}\left\|v \circ \Psi_{i}
            \right\|_{(l, \alpha), \nu, r_{0}}.
        \end{equation*}
        The space $C_{\nu}^{l, \alpha}\left(M_{r}\right)$ is defined similarly.
        
        We now outline how to prove the analog of Proposition~\ref{RIEX}. 
        Choose a radial cutoff function 
        $$\eta : \mathbb{R}^n \rightarrow [0,\infty), \qquad \eta(x) = 
        \left \{ \begin{array}{rl} 1, & |x| < r_1/2 \\ 0, & |x|> r_1 
        \end{array} \right . $$
        and let 
        $$w_0 = \sum_{i=1}^k \eta \mathscr{H}_{r,r_1}^i (\left .
        f \right |_{\Omega_{r,r_1}^i} ).$$
        Here $\Omega_{r,r_1}$ is the annulus centred at $p_i$ with 
        inner radius $r$ and outer radius $r_1$, with respect to our 
        chosen conformal normal coordinates described above. Letting
        $$ h = f - L_{g_0} ( w_0)$$ 
        we see as in the proof of Proposition~\ref{RIEX} 
        that there exists $C_1>0$ such that 
        $$\| h \|_{C^{0,\alpha}(M)} \leq C_1 \| f \|_{C^{0,\alpha}_{\nu-4}
        (M_r)}.$$
        Now choose 
        $$\chi: M \rightarrow [0,\infty), \qquad 
        \chi(x) = \left \{ \begin{array}{rl} 0 & x \in 
        \bigcup_{i=1}^k B_{r_2}(p_i) \\ 1 & 
        x \in M\backslash \left ( \bigcup_{i=1}^k 
        B_{2r_2} (p_i) \right ), \end {array} \right . $$
        where $4r_2 < r_1$ and let 
        $$w_1 = \chi L_{g_0}^{-1} (h).$$ 
        The same estimates as before tell us 
        $$\| w_1 \|_{C^{4,\alpha}_\nu (M_r)} \leq 
        C_2 \| f \|_{C^{0,\alpha}_{\nu-4} (M_r)}$$ 
        for some constant $C_2$ independent of $r$. 
        Finally, as in the previous proof, we see 
        $$\| L_{g_0} ( w_0 + w_1) - f 
        \|_{C^{0,\alpha}_{\nu-4} (M_r)} \leq C_3 r_2^{-1-\nu} 
        \| f \|_{C^{0,\alpha}_{\nu-4}(M_r)},$$ 
        and so we may find our right-inverse for 
        the exterior problem by perturbations, such that 
        the resulting function $w$ is constant 
         on any component of $\partial M_{r}$. 
         
         Choosing boundary functions $\psi_0^i, \psi_2^i 
         \in C^{4,\alpha} (\mathbb{S}^{n-1}_{r_{\varepsilon_i}})$
         we define $u_{\psi_0^i ,\psi_2^i}$ by 
         $$u_{\psi_0^i, \psi_2^i}\circ \Psi_i  = 
         \eta \mathcal{Q}_{r_{\varepsilon_i}} (\psi_0^i, 
         \psi_2^i) \qquad \textrm{ in } 
         B_{2r_0}(p_i)$$
         and letting $u_{\psi_0^i, \psi_2^i}$ be zero 
         otherwise. 
         The same fixed-point argument as before will give 
         us $V_{\lambda_i, \psi_0^i, \psi_2^i}$ such that 
         $$H_{g_0} (1+G_{p_1,\dots, p_k} + u_{\psi_0^i, \psi_2^i}
         + V_{\lambda_i, \psi_0^i, \psi_2^i} ) = 0,$$ 
         which gives us our exterior solution. 
         
         We complete the gluing construction by once
         more using fixed-point theorems to show that 
         one can choose the boundary date and parameters 
         such that the interior and exterior solutions match 
         to third order across the interface. This part of the 
         proof is the same as in the one-point gluing 
         construction. 
	\end{proof}
	


\begin{thebibliography}{10}

\bibitem{MR1719213}
A.~Ambrosetti and A.~Malchiodi, A multiplicity result for the {Y}amabe problem
  on {$S^n$}, \emph{J. Funct. Anal.} {\bf 168} (1999) 529--561.

\bibitem{arXiv:2003.03487}
J.~H. Andrade and J.~M. do~\'O, Asymptotics for singular solutions to
  conformally invariant fourth order systems in the punctured ball,
  \emph{arXiv:2003.03487 [math.AP]}  (2020).

\bibitem{arXiv:2101.11304}
J.~H. Andrade, J.~M. do~\'O and J.~Ratzkin, Compactness within the space of
  complete, constant Q-curvature metrics on the sphere with isolated
  singularities, \emph{to appear in Int. Math. Res. Not. IMRN}  (2021).

\bibitem{MR1936047}
S.~Baraket and S.~Rebhi, Construction of dipole type singular solutions for a
  biharmonic equation with critical {S}obolev exponent, \emph{Adv. Nonlinear
  Stud.} {\bf 2} (2002) 459--476.

\bibitem{BCS}R.~Beig, P.~Chru\'sciel and R.~Schoen, KIDS 
are non-generic, \emph{Ann. Henri Poincar\'e} {\bf 6} (2005), 155--194.

\bibitem{MR4251294}
R.~G. Bettiol, P.~Piccione and Y.~Sire, Nonuniqueness of conformal metrics
  with constant {Q}-curvature, \emph{Int. Math. Res. Not. IMRN}  (2021)
  6967--6992.

\bibitem{BOUWEVANDENBERG}
J.~{B. van den Berg}, The phase-plane picture for a class of fourth-order
  conservative differential equations, \emph{J. Differential Equations}
  {\bf 161} (2000) 110--153.

\bibitem{MR832360}
T.~P. Branson, Differential operators canonically associated to a conformal
  structure, \emph{Math. Scand.} {\bf 57} (1985) 293--345.

\bibitem{MR904819}
T.~P. Branson, Group representations arising from {L}orentz conformal geometry,
  \emph{J. Funct. Anal.} {\bf 74} (1987) 199--291.

\bibitem{MR2407527}
T.~P. Branson and A.~R. Gover, Origins, applications and generalisations of the
  {$Q$}-curvature, \emph{Acta Appl. Math.} {\bf 102} (2008) 131--146.

\bibitem{MR2425176}
S.~Brendle, Blow-up phenomena for the {Y}amabe equation, \emph{J. Amer. Math.
  Soc.} {\bf 21} (2008) 951--979.

\bibitem{MR2119285}
G.~Brown and F.~Dai, Approximation of smooth functions on compact two-point
  homogeneous spaces, \emph{J. Funct. Anal.} {\bf 220} (2005) 401--423.

\bibitem{MR2010322}
A.~Byde, Gluing theorems for constant scalar curvature manifolds, \emph{Indiana
  Univ. Math. J.} {\bf 52} (2003) 1147--1199.

\bibitem{arXiv:2009.01787}
R.~Caju, J.~M. do~\'O and A.~Silva Santos, Singular solutions to Yamabe-type
  systems with prescribed asymptotics, \emph{arXiv:2009.01787 [math.AP]}
  (2020).

\bibitem{MR991959}
J.~G. Cao, The existence of generalized isothermal coordinates for
  higher-dimensional {R}iemannian manifolds, \emph{Trans. Amer. Math. Soc.}
  {\bf 324} (1991) 901--920.

\bibitem{MR2407525}
S.-Y.~A. Chang, M.~Eastwood, B.~\O rsted and P.~C. Yang, What is
  {$Q$}-curvature?, \emph{Acta Appl. Math.} {\bf 102} (2008) 119--125.

\bibitem{MR312040}
D.~R. Dunninger, Maximum principles for solutions of some fourth-order elliptic
  equations, \emph{J. Math. Anal. Appl.} {\bf 37} (1972) 655--658.

\bibitem{MR3869387}
R.~L. Frank and T.~König, Classification of positive singular solutions to a
  nonlinear biharmonic equation with critical exponent, \emph{Anal. PDE}
  {\bf 12} (2019) 1101–1113.

\bibitem{Gazzola_Grunau_Sweers_2010}
F.~Gazzola, H.-C. Grunau and G.~Sweers, \emph{Polyharmonic boundary value
  problems. Positivity preserving and nonlinear higher order elliptic equations
  in bounded domains}, Lecture Notes in Mathematics, 1991, Springer-Verlag,
  Berlin (2010).

\bibitem{MR1225437}
M.~G\"{u}nther, Conformal normal coordinates, \emph{Ann. Global Anal. Geom.}
  {\bf 11} (1993) 173--184.

\bibitem{MR3420504}
M.~J. Gursky and A.~Malchiodi, A strong maximum principle for the {P}aneitz
  operator and a non-local flow for the {$Q$}-curvature, \emph{J. Eur. Math.
  Soc. (JEMS)} {\bf 17} (2015) 2137--2173.

\bibitem{MR3618119}
F.~Hang and P.~C. Yang, Lectures on the fourth-order $Q$-curvature equation,
  \emph{Geometric analysis around scalar curvatures}, \emph{Lect. Notes Ser.
  Inst. Math. Sci. Natl. Univ. Singap.}, vol.~31, World Sci. Publ., Hackensack,
  NJ (2016) 1--33.

\bibitem{MR3518237}
F.~Hang and P.~C. Yang, {$Q$}-curvature on a class of manifolds with dimension
  at least 5, \emph{Comm. Pure Appl. Math.} {\bf 69} (2016) 1452--1491.

\bibitem{MR1427765}
E.~Hebey, From the {Y}amabe problem to the equivariant {Y}amabe problem,
  \emph{Actes de la {T}able {R}onde de {G}\'{e}om\'{e}trie {D}iff\'{e}rentielle
  ({L}uminy, 1992)}, \emph{S\'{e}min. Congr.}, vol.~1, Soc. Math. France, Paris
  (1996) 377--402, joint work with M. Vaugon.

\bibitem{MR1216009}
E.~Hebey and M.~Vaugon, Le probl\`eme de {Y}amabe \'{e}quivariant, \emph{Bull.
  Sci. Math.} {\bf 117} (1993) 241--286.

\bibitem{MR274792}
R.~R. Huilgol, On {L}iouville's theorem for biharmonic functions, \emph{SIAM J.
  Appl. Math.} {\bf 20} (1971) 37--39.

\bibitem{MR4170788}
A.~Hyder and Y.~Sire, Singular solutions for the constant {$Q$}-curvature
  problem, \emph{J. Funct. Anal.} {\bf 280} (2021) 108819, 39.

\bibitem{jleli}
M.~Jleli, \emph{Constant mean curvature hypersurfaces}, PhD Thesis, University
  of Paris 12 (2004).

\bibitem{MR2194146}
M.~Jleli and F.~Pacard, An end-to-end construction for compact constant mean
  curvature surfaces, \emph{Pacific J. Math.} {\bf 221} (2005) 81--108.

\bibitem{MR2477893}
M.~A. Khuri, F.~C. Marques and R.~M. Schoen, A compactness theorem for the
  {Y}amabe problem, \emph{J. Differential Geom.} {\bf 81} (2009) 143--196.

\bibitem{MR888880}
J.~M. Lee and T.~H. Parker, The {Y}amabe problem, \emph{Bull. Amer. Math. Soc.
  (N.S.)} {\bf 17} (1987) 37--91.
  
  \bibitem{MR4028770}
G.~Li, A compactness theorem on {B}ranson's {$Q$}-curvature equation, \emph{Pacific J. Math.} {\bf 302} (2019) 119--179.

\bibitem{MR3899029}
Y.~Y.~Li and J.~Xiong, Compactness of conformal metrics with constant
  {$Q$}-curvature. {I}, \emph{Adv. Math.} {\bf 345} (2019) 116--160.

\bibitem{MR2164927}
Y.~Y. Li and L.~Zhang, Compactness of solutions to the {Y}amabe problem. {II},
  \emph{Calc. Var. Partial Differential Equations} {\bf 24} (2005) 185--237.

\bibitem{MR2309836}
Y.~Y. Li and L.~Zhang, Compactness of solutions to the {Y}amabe problem. {III},
  \emph{J. Funct. Anal.} {\bf 245} (2007) 438--474.

\bibitem{Lin98}
C.-S. Lin, A classification of solutions of a conformally invariant fourth
  order equation in {${\bf R}^n$}, \emph{Comment. Math. Helv.} {\bf 73} (1998)
  206--231.

\bibitem{MR3333110}
Y.-J. Lin, Connected sum construction of constant {$Q$}-curvature manifolds in
  higher dimensions, \emph{Differential Geom. Appl.} {\bf 40} (2015) 290--320.

\bibitem{MR2197144}
F.~C. Marques, A priori estimates for the {Y}amabe problem in the non-locally
  conformally flat case, \emph{J. Differential Geom.} {\bf 71} (2005) 315--346.

\bibitem{MR2551136}
F.~C. Marques, Blow-up examples for the {Y}amabe problem, \emph{Calc. Var.
  Partial Differential Equations} {\bf 36} (2009) 377--397.

\bibitem{MR1425579}
R.~Mazzeo and F.~Pacard, A construction of singular solutions for a semilinear
  elliptic equation using asymptotic analysis, \emph{J. Differential Geom.}
  {\bf 44} (1996) 331--370.

\bibitem{Mazzeo-pacard-1999}
R.~Mazzeo and F.~Pacard, Constant scalar curvature metrics with isolated
  singularities, \emph{Duke Math. J.} {\bf 99} (1999) 353--418.

\bibitem{Mazziri-Ndiaye}
L.~Mazzieri and C.~B. Ndiaye, Existence of solutions for the singular
  $\sigma_k$-Yamabe problem, \emph{preprint}  (2010).

\bibitem{Nicolaescu_2007}
L.~I. Nicolaescu, \emph{Lectures on the geometry of manifolds}, World
  Scientific Publishing Co. Pte. Ltd., Hackensack, NJ, 2nd ed. (2007).

\bibitem{MR1763040}
F.~Pacard and T.~Rivi\`ere, \emph{Linear and nonlinear aspects of vortices.} The Ginzburg-Landau
  model. \emph{Progress in Nonlinear Differential Equations and their Applications},
  vol.~39, Birkh\"{a}user Boston, Inc., Boston, MA (2000). 

\bibitem{MR2393291}
S.~M. Paneitz, A quartic conformally covariant differential operator for
  arbitrary pseudo-{R}iemannian manifolds (summary), \emph{SIGMA Symmetry
  Integrability Geom. Methods Appl.} {\bf 4} (2008) Paper 036, 3pp.

\bibitem{arXiv:2002.05939}
J.~Ratzkin, On a fourth order conformal invariant,
  \emph{arXiv:2002.05939 [math.DG]}  (2020).

\bibitem{arxiv.2001.07984}
J.~Ratzkin, On constant Q-curvature metrics with isolated singularities,
  \emph{arXiv:2001.07984 [math.DG]}  (2020).
  
\bibitem{MR929283}
R. Schoen, The existence of weak solutions with prescribed singular behavior
  for a conformally invariant scalar equation, \emph{Comm. Pure Appl. Math.}
  {\bf 41} (1988) 317--392.

\bibitem{MR1144528}
R. Schoen, A report on some recent progress on nonlinear problems in
  geometry, \emph{Surveys in differential geometry ({C}ambridge, {MA}, 1990)},
  Lehigh Univ., Bethlehem, PA (1991) 201--241.

\bibitem{MR2639545}
A.~Silva Santos, A construction of constant scalar curvature manifolds with
  {D}elaunay-type ends, \emph{Ann. Henri Poincar\'{e}} {\bf 10} (2010)
  1487--1535.
  
  \bibitem{MR3663325}
A.~Silva Santos, Solutions to the singular {$\sigma_2$}-{Y}amabe problem with isolated singularities, \emph{Indiana Univ. Math. J.} {\bf 66} (2017)
  741--790.

\bibitem{MR1679783}
J.~Wei and X.~Xu, Classification of solutions of higher order conformally
  invariant equations, \emph{Math. Ann.} {\bf 313} (1999) 207--228.

\end{thebibliography}
    \end{document}